\documentclass[11pt]{amsart}
\usepackage{amsmath,amsthm,verbatim,amssymb,amsfonts,amscd,graphicx,amsopn}
\usepackage{enumitem}
\usepackage{enumerate}
\usepackage{mathrsfs}
\usepackage{mathtools}
\usepackage[margin=1in]{geometry}
\usepackage{bm}
\usepackage{subcaption}
\usepackage{epsfig}
\usepackage{times}
\usepackage{epstopdf}

\DeclareMathOperator{\WF}{WF}
\DeclareMathOperator{\supp}{supp}
\DeclareMathOperator{\expo}{exp}
\theoremstyle{plain}
\newtheorem*{pp*}{Proposition}	
\newtheorem{pp}{Proposition}
\newtheorem*{cl*}{Claim}
\theoremstyle{definition}
\newtheorem*{remark*}{Remark}
\theoremstyle{definition}

\theoremstyle{definition}
\newtheorem{eg}{Example}[]
\newtheorem{lm}{Lemma}
\newtheorem{theorem}{Theorem}[]
\newtheorem{corollary}{Corollary}[]

\begin{document} 	
\title{Artifacts in the inversion of the broken ray transform in the plane}
\author[Y. Zhang]{Yang Zhang}
\address{Purdue University \\ Department of Mathematics}
\email{zhan1891@purdue.edu}
\thanks{Partly supported by NSF Grant DMS-1600327}
\begin{abstract}
	We study the integral transform over a general family of broken rays in $\mathbb{R}^2$. It is natural for broken rays to have conjugate points, for example, when they are reflected from a curved boundary. If there are conjugate points, we show that the singularities cannot be recovered from local data and therefore artifacts arise in the reconstruction. We apply these conclusions to two examples, the V-line Radon transform and the parallel ray transform. In each example, a detailed discussion of the local and global recovery of singularities is given and we perform numerical experiments to illustrate the results.
\end{abstract}
\maketitle
%%%%%%%%%%%%%%%%%%%%%%%% part0_Introduction %%%%%%%%%%%%%%%%%%%%%%%%%%%%%%%
%%%%%%%%%%%%%%%%%%%%%%%% part0_Introduction %%%%%%%%%%%%%%%%%%%%%%%%%%%%%%%
%%%%%%%%%%%%%%%%%%%%%%%% part0_Introduction %%%%%%%%%%%%%%%%%%%%%%%%%%%%%%%
\section{Introduction}\label{setup_sec}
The purpose of this work is to study the integral transform over a general family of broken rays in the plane. 
A broken ray in the Euclidean space is usually defined as the linear path following the law of reflection on the boundary,
which will be an important example of our setting in the following.
In fact, a major motivation of this work is the reconstruction of an unknown function from integrals over such broken rays in medical imaging. 
One promising application is the Single Photon Emission Computed Tomography (SPECT) with Compton cameras.

We define a more general family of broken rays. 
%Let a broken ray $\nu$ be the union of two non-intersecting half-lines related by a diffeomorphism $\chi$.
Suppose $f$ is a distribution with compact support.
A \textit{broken ray} is defined as the union of two rays $l_1$ and $l_2$ that do not intersect in the support of $f$ and are related by a diffeomorphism, as in Figure \ref{setupaabb}.  
The broken ray transform $\mathcal{B}f$ is the integral of $f$ along $\nu$
\begin{equation*}
\mathcal{B}f(\nu) = \int f({\nu}(t)) dt.
\end{equation*}
One way to think about this is to imagine that there is a curve smoothly connecting these two half-lines. 
Then it becomes an X-ray type transform over smooth curves.
%Let $f$ be an unknown distribution supported in a bounded domain. 
The connecting curve plays no rule in the analysis, if we always assume that the distribution $f$ is compactly supported away from it. 
\begin{figure}[h]
	\centering
	\begin{subfigure}{0.5\textwidth}
		\centering
		\includegraphics[width =0.6\textwidth, height=0.6\textwidth,keepaspectratio]{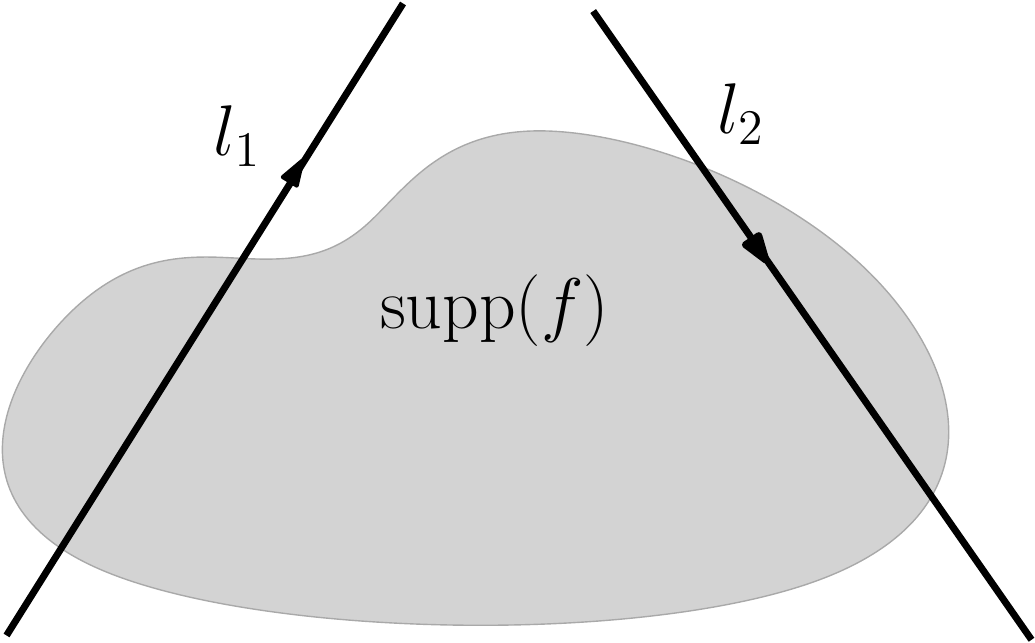}
	\end{subfigure}%
	\begin{subfigure}{0.5\textwidth}
		\centering
		\includegraphics[width =0.6\textwidth, height=0.6\textwidth,keepaspectratio]{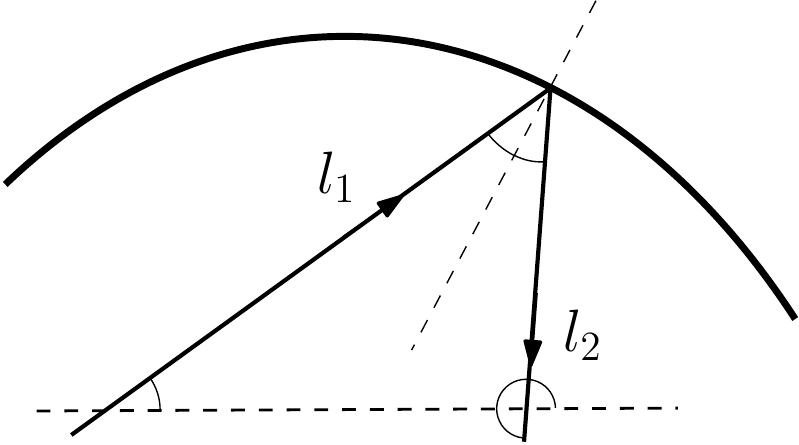}
	\end{subfigure}
	\caption{ Left: a general broken ray, where $l_1$ and $l_2$ are related by a diffeomorphism. Right: a broken ray in the reflection case.}
	\label{setupaabb}
\end{figure}

The goal is to understand what part of the wave front set $\text{\( \WF \)}(f)$ can be recovered from the transform $\mathcal{B} f$. 
The wave front set allows us to recover, in particular, the location and the size of jumps of $f$.
Conjugate points naturally exist for broken rays, for instance, when rays are reflected from a curved boundary. One would expect and we confirm that recovery of singularities are affected by the existence of conjugate points on $\nu$.
Much work has been done for the class of X-ray type transform with conjugate points \cite{MR3080199,MR2970707,MR3339183,Holman2017}. 
In the case of transform for a generic family of smooth curves \cite{MR2365669}, if there are no conjugate points, the localized normal operator is an elliptic pseudodifferential operator of order $-1$. 
Injectivity and the stability estimates are established, which in particular implies that we can recover the singularities uniquely. When conjugate points exist, however, artifacts may arise, and in some situations they cannot be resolved. 
A similar situations occurs in synthetic aperture radar imaging \cite{MR3080199}; it is impossible to recover $\text{\( \WF \)}(f)$ if the singularities hit the trajectory only once, because of the existence of mirror points. On the other hand, if the trajectory is the boundary of a strictly convex domain and we have the priori that $f$ has singularities in a compact set, then we can recover
$\text{\( \WF \)}(f)$ from the global data. However, this is a global procedure and we cannot expect a local reconstruction. 
In the case of X-ray transforms over geodesic-like families of curves with conjugate points of fold type, a detailed description of the normal operator is given in \cite{MR2970707}. The analysis of the normal operator for general conjugate points is in \cite{HOLMAN2015OnTM}. 
Further, \cite{MR3339183} shows that regardless of the type of the conjugate points, the geodesic ray transform on Riemannian surfaces is always unstable and we have loss of all derivatives, which leads to the artifacts in the reconstruction. It is also proved that the attenuated geodesic ray transform is well posed under certain conditions. 
Most recently, \cite{Holman2017} provides a thorough analysis of the stability of attenuated geodesic ray transform and shows what artifacts we can expect when using the Landweber iterative reconstruction for unstable problems. 

As mentioned above, an important example of this setting is the broken ray transform in the usual sense, which is also called V-line Radon transform. As is shown in Figure \ref{setupaabb}, the diffeomorphism is given by the law of reflection. 
The inversion of the V-line Radon transform arises in Single Photon Emission Computed Tomography (SPECT) with Compton cameras in two dimensions.
SPECT based on Anger camera is a widely used technique for functional imaging in medical diagnosis and biological research. The using of Compton camera in SPECT is proposed to greatly improve the sensitivity and resolution \cite{W.Todd1974,Everett1977,Singh1983}. 
The gamma photons are emitted proportionally to markers density and then are scattered by two detectors. Photons can be traced back to broken lines. 
The mathematical model is the cone transform (or conical Radon transform) of an unknown density. 
Various inversion approaches for certain cases are proposed in \cite{Basko1998,Parra2000,Tomitani2002a,Smith2005,Maxim2009,Gouia-Zarrad2014a,Terzioglu2015,Kuchment2016,Jung2015,Schiefeneder2016,Moon2016,Terzioglu2017,Moon2017a,Haltmeier2017}. 
The V-line Radon transform is a special case in two dimensions where each vertex is restricted on a curve and is associated with a single axis. 
%In \cite{Haltmeier2017},  an analytical inversion approach based on the circular harmonics expansion is derived and it shows the uniqueness of reconstruction for the attenuated V-line transform with vertices on the circle. 
There are also some injectivity and stability results when we allow the rays to reflect on the boundary more than once \cite{Ilmavirta2013,Hubenthal2014,Ilmavirta2015,Hubenthal2015}. 
These reconstructions are from full data and most of them assume specific boundaries at least for the reflection part, for example a flat one or a circle. 
It also should be mentioned that the broken ray transform or the V-line Radon transform sometimes refers to a different transform from the one we consider in this work, see \cite{Morvidone2010,Florescu2011,Ambartsoumian2012,Krylov2015}. In their settings, the V-line vertices are inside the object with a fixed axis direction. The integral near the vertices in the support of $f$ makes it possible to recover singularities there. In this work, however, the vertices are always away from support of $f$ and the axis direction changes, which make the recovery more difficult. 

Another motivation is the application of parallel ray transform in X-ray luminescence computed tomography (XLCT). A multiple pinhole collimator based on XLCT is proposed in \cite{Zhang2016} to promote photon utilization efficiency in a single pinhole collimator. 
In this method, multiple X-ray beams are generated to scan a sample at multiple positions simultaneously, which we mathematically model by the parallel ray transform, see Section \ref{parallel_section}.

We are inspired by the spirit of \cite{MR3080199,MR3339183,Holman2017}. In Section \ref{conj_section} and \ref{cancellation_section}, we introduce the definition of conjugate points and conjugate covectors along broken rays and give a characterization of them. 
Then we consider the local problem, that is, the data $\mathcal{B}f$ is known in a small neighborhood of a fixed broken ray. We show that $\mathcal{B}$ is an FIO and the image of two conjugate covectors under its canonical relation are identical. 
Singularities can be canceled by these conjugate covectors. This implies that we can only reconstruct $f$ up to an error in the microlocal kernel.  
We also provide a similar analysis for the numerical result as in \cite{Holman2017}, if the Landweber iteration is used to reconstruct $f$. The main results are listed in the following
\begin{itemize}[noitemsep,topsep=0pt]
	\item[(1)] The local problem is ill-posed if there are conjugate points. Singularities cannot be recovered uniquely.
	\item[(2)] The global problem might be well-posed in some cases for most singularities, depending on a dynamical system inside the domain.
\end{itemize}

In Section \ref{reflection_section} and Section \ref{parallel_section}, we apply these conclusions to two special examples, the V-line Radon transform and the parallel ray transform, as mentioned above. The conjugate points appearing in the V-line Radon transform coincide with the caustics in geometrical optics, see \cite{Boyle2015}. Additionally, when the boundary is a circle, we show that there exists conjugate points of fold type as well as cusps. Geometrically, the caustic inside a circle is an interesting problem itself, which can be traced back to the middle of 19th century \cite{Cayley1857}. We discuss the local and global recovery of singularities and we perform numerical experiments to illustrate the results. In particular, for the reflection case in a circle, we connect our analysis with the inversion formula derived in \cite{Moon2017a}.

Some assumptions and notation are introduced in the following. 
Throughout this work, we assume $f$ is a distribution supported in compact set in $\mathbb{R}^2$. 
As mentioned before, a broken ray $\nu$ is the union of two directed lines $l_1$ and $l_2$ related by a diffeomorphism $\chi$. 
The two lines do not intersect with each other in the support of $f$. 
We always suppose there is a smooth family of them. 
Let $v(\alpha) = (\cos \alpha, \sin \alpha)$  and $w(\alpha) = (-\sin \alpha, \cos \alpha)$.
We use $(s_j,\alpha_j)$ to represent a directed line $\{x \in \mathbb{R}^2 | x \cdot w(\alpha_j) = s_j\}$ with the direction $v(\alpha_j) $ and the unit normal  $w(\alpha_j) $, for $j = 1,2$.
Suppose $l_1$ is the incoming part represented by $(s_1,\alpha_1)$ and $l_2$ is the reflected part represented by $(s_2,\alpha_2)$. They are related by $\chi : (s_1,\alpha_1) \mapsto (s_2,\alpha_2)$. 

\subsection*{Acknowledgements}
The author would like to express the acknowledgments to Prof. Plamen Stefanov for suggesting this problem and for the patient guidance and helpful suggestions he has provided throughout this project.
%the helpful ideas obtained from conversations with him.
%%%%%%%%%%%%%%%%%%%%%%%% part1_Conjugate %%%%%%%%%%%%%%%%%%%%%%%%%%%%%%%
%%%%%%%%%%%%%%%%%%%%%%%% part1_Conjugate %%%%%%%%%%%%%%%%%%%%%%%%%%%%%%%
%%%%%%%%%%%%%%%%%%%%%%%% part1_Conjugate %%%%%%%%%%%%%%%%%%%%%%%%%%%%%%%

\section{Conjugate Points}\label{conj_section}
On a Riemannian manifold, the \textit{conjugate vectors} of a fixed point $p$ are vectors such that the differential of the exponential map $\text{\( \expo \)}_p(v)$ with respect to $v$ is not an isomorphism. The \textit{conjugate points} are the image of these vectors under the exponential map.
In \cite{MR3339183}, it is shown that singularities can be canceled by conjugate points in the geodesic ray transform case. Conjugate points also exist in the case of broken ray transform, for exmaple,  the caustics in geometrical optics, see \cite{Cayley1857,Boyle2015}. The light rays reflected or refracted by a curved surface form an envelope, which are conjugate points of the light source.

Suppose the incoming ray $l_1$ starts from $p$ with the angle $\alpha_1$.
Consider the conjugate points of $p$ belonging to $l_2$.
We show below the conjugate points on $l_2$ do not depend on what kind of connecting curve we choose, if $l_1$ and $l_2$ are given.
Notice the X-ray parametrization $(p,\alpha_1)$ gives us a Radon parameterization $(s_1,\alpha_1)$ by $s_1 = \langle p, w(\alpha_1) \rangle$.
Fix $p$, for each $\alpha_1$, the diffeomorphism $\chi$ gives another ray $(s_2,\alpha_2)$. Suppose the reflected ray starts at $q_0$ at time $t_2$. Here $q_0 = q_0(s_1,\alpha_1)$ is chosen in a smooth way and should always satisfy $\langle q_0, w(\alpha_2) \rangle = s_2$. 
%Suppose $p$ is on $l_1$ and we consider the conjugate points belonging to $l_2$.
%We show below the conjugate points on $l_2$ do not depend on what kind of connecting curve we choose, if $l_1$ and $l_2$ are given.
%The incoming ray $l_1$ starts from $p$ with the angle $\alpha_1$. The X-ray parametrization $(p,\alpha_1)$ gives us a Radon parameterization $(s_1,\alpha_1)$ by $s_1 = \langle p, w(\alpha_1) \rangle$.
%For each $\alpha_1$, the diffeomorphism $\chi$ gives another ray $(s_2,\alpha_2)$. Suppose the reflected ray starts at $q_0$ at time $t_2$. The point $q_0 = q_0(s_1,\alpha_1)$ is chosen in a smooth way. Notice $q_0$ should always satisfy $\langle q_0, w(\alpha_2) \rangle = s_2$. 
Then the exponential map is given in the following
$$
\begin{cases}\label{exponential}
{l_1}(t) = p + t {v}(\alpha_1) , \quad  & 0 \leq t \leq t_1,\\
%{l_c}(t) \text{: some unknown curve} & t_1 \leq t \leq t_2\\
{l_2}(t) = q_0 + (t-t_2){{v(\alpha_2)}}, \quad & t\geq t_2 \geq t_1.
\end{cases}
$$

To calculate the differential of the exponential map, we choose $\alpha_1$ and $t$ as the parameters, that is, we change the starting angle and the length. %Since $(s_1, \alpha_1)$ and $(s_2,\alpha_2)$ are related via $\chi$, with $p$ fixed, $s_2$ and $\alpha_2$ are smooth functions of $\alpha_1$. 
We have
\begin{align*}
\frac{\partial {l_2}}{\partial \alpha_1} 
&=(t-t_2) \frac{d {{v(\alpha_2)}}}{d \alpha_1} - \frac{d t_2}{d \alpha_1} {{v(\alpha_2)}} + \frac{d {q_0}}{d \alpha_1}\\
%&= (t-t_2) \big( \frac{d {{\alpha_2}}}{d \alpha_1} \big) w(\alpha_2)+ \langle \frac{d {q_0}}{d \alpha_1},  w(\alpha_2) \rangle  w(\alpha_2) +  \langle \frac{d {q_0}}{d \alpha_1}, v(\alpha_2) \rangle  {v(\alpha_2)} - \frac{d t_2}{d \alpha_1} {{v(\alpha_2)}}\\
&= \big((t-t_2) \big( \frac{d {{\alpha_2}}}{d \alpha_1} \big) + \langle \frac{d {q_0}}{d \alpha_1},  w(\alpha_2) \rangle \big)  w(\alpha_2) +  \big(\langle \frac{d {q_0}}{d \alpha_1}, v(\alpha_2) \rangle   - \frac{d t_2}{d \alpha_1} \big) {{v(\alpha_2)}}.
\end{align*}
Then the Jacobi matrix $\frac{\partial l_2(t)}{\partial(t,\alpha_1)}$ 
has the following determinant
\begin{align}\label{det}
\begin{split}
\det(d \text{\( \expo \)}_{{p}}(t {v(\alpha_1)}))= 
\det
\left[
\begin{array}{cc}
\frac{\partial {l_2}}{\partial t} & \frac{\partial {l_2}}{\partial \alpha_1} 
\end{array}
\right]
%=&
%\left|
%\begin{array}{cc}
%v(\alpha_2) & \big((t-t_2) \big( \frac{d {{\alpha_2}}}{d \alpha_1} \big) + \langle \frac{d {q_0}}{d \alpha_1},  w(\alpha_2) \rangle \big)  w(\alpha_2) +  \big(\langle \frac{d {q_0}}{d \alpha_1}, v(\alpha_2) \rangle   - \frac{d t_2}{d \alpha_1} \big) {{v(\alpha_2)}}
%\end{array}
%\right|.
= (t-t_2) \big( \frac{d {{\alpha_2}}}{d \alpha_1} \big) + \langle \frac{d {q_0}}{d \alpha_1},  w(\alpha_2) \rangle.
\end{split}
\end{align}
The last equality comes from the observation that $\det[v(\alpha_2)\ w(\alpha_2)] = 1$.
Thus, the determinant vanishes if and only if $(t-t_2)\big( \frac{d {{\alpha_2}}}{d \alpha_1} \big) = -\langle \frac{d {q_0}}{d \alpha_1},  w(\alpha_2)\rangle $. 
Notice if we assume $q_0$ is the starting point of $l_2$, this equation makes sense when $t - t_2>0$, that is, when $\big( \frac{d {{\alpha_2}}}{d \alpha_1} \big) \langle \frac{d {q_0}}{d \alpha_1},  w(\alpha_2)\rangle \leq 0$.

Then differentiating $\langle q_0, w(\alpha_2) \rangle = s_2$ with respect to $\alpha_1$ shows
\begin{equation}\label{diffq0}
\langle \frac{d q_0}{d \alpha_1}, w(\alpha_2) \rangle + \langle q_0, - v(\alpha_2) \rangle \big( \frac{d {{\alpha_2}}}{d \alpha_1} \big) = \frac{d s_2}{d \alpha_1}.
\end{equation}
If $q$ is the point on $l_2$ at $t$ such that ${d \text{\( \expo \)}_{{p}}(t {v_1})}	$ is not an isomorphism, then we have $t - t_2 = \langle q-q_0,  v(\alpha_2) \rangle$. By (\ref{det})(\ref{diffq0}),  $q$ should satisfy 
\begin{align}\label{equaitonq}
\langle q,  v(\alpha_2) \rangle  \frac{d {{\alpha_2}}}{d \alpha_1}  =  - \frac{d s_2}{d \alpha_1} .
\end{align}
On the contrary, if there exists $q$ on $l_2$ such that the equation (\ref{equaitonq}) is  true, then the determinant of the differential will be zero.
Notice $\frac{d {{\alpha_2}}}{d \alpha_1}$ and $\frac{d s_2}{d \alpha_1}$ cannot vanish at the same time with the assumption that $\chi$ is a diffeomorphism.
This proves the following. 
\begin{pp}\label{conjugatepp}
	Suppose the incoming ray $l_1$ starting from $p$ represented by $(s_1,\alpha_1)$ and the reflected ray $l_2$ starting from $q_0$ represented by $(s_2,\alpha_2)$ are related by a diffeomorphism $\chi$. 
	The point $q_0$ is chosen smoothly depending on $(s_1, \alpha_1)$. 
	Then 
	\begin{itemize}
		\item[(a)] $p$ has a conjugate point $q$ belonging to $l_2$ if and only if 
		$$ \frac{d {{\alpha_2}}}{d \alpha_1}  \langle \frac{d {q_0}}{d \alpha_1},  w(\alpha_2)\rangle < 0.$$
		\item[(b)]If this occurs, $q$ is uniquely determined by $\langle q,  v(\alpha_2) \rangle =  -\big( \frac{d {{\alpha_2}}}{d \alpha_1} \big)^{-1}  \frac{d s_2}{d \alpha_1} .$
	\end{itemize}
\end{pp}
\begin{remark*}\label{conjugateperturb}

	If we consider the whole straight line which $l_2$ belongs to instead of the ray, then we can always find one and the only one conjugate point $q$ satisfying (b), unless $
	\frac{d {{\alpha_2}}}{d \alpha_1}=0$. 
	The condition (a) is to check whether this $q$ belong to the reflected ray that we define. 
	Additionally, if we perturb $q_0$ a little bit, that is, let $q_0' = q_0 + \epsilon(\alpha_1) v(\alpha_2)$. Then 
	$\frac{d {q_0'}}{d \alpha_1} = \frac{d {q_0}}{d \alpha_1} +\epsilon(\alpha_1) w(\alpha_2) +  \frac{d \epsilon(\alpha_1)}{d \alpha_1}v(\alpha_2)$.
	We have
	$$ 
	\langle \frac{d {q_0'}}{d \alpha_1},  w(\alpha_2)\rangle = \langle \frac{d {q_0}}{d \alpha_1},  w(\alpha_2)\rangle + \epsilon(\alpha_1).
	$$	
	This shows a small enough perturbation of $q_0$ doesn't change the sign of $ 
	\langle \frac{d {q_0}}{d \alpha_1},  w(\alpha_2)\rangle$. Therefore the existence of conjugated points is not affected by the choice of $q_0$ in a small neighborhood.
\end{remark*}
%In fact assuming $p$ is a fixed point is equivalent to restricting $\chi$ on the set $\{ s_1 = p \cdot w(\alpha_1)\}$, which implies the following formula to calculate 
%$\frac{d \alpha_2}{d \alpha_1}$ from the diffeomorphism $\chi$ explicitly,
%\begin{align}\label{computeformula}
%&\frac{d \alpha_2}{d \alpha_1} = \frac{\partial \alpha_2}{\partial \alpha_1} + \frac{\partial \alpha_2}{\partial s_1} \frac{d s_1}{d \alpha_ 1} = \frac{\partial \alpha_2}{\partial \alpha_1} - \frac{\partial \alpha_2}{\partial s_1} \langle p, v(\alpha_1) \rangle.
%\end{align}
%
%Similarly, we have the formulas for $\frac{d s_2}{d \alpha_1}$ and $\frac{d q_0}{d \alpha_1}$. 

The proof above implies $p$ and $q$ are conjugate points to each other in some sense by the following reasons. Suppose we know that $p$ and $q$ belong to $\nu$. The point $q$ is the conjugate point of $p$ if and only if 
$$
\langle q,  v(\alpha_2) \rangle =  -\big( \frac{d {{\alpha_2}}}{d \alpha_1} \big)^{-1}  \frac{d s_2}{d \alpha_1} 
= -\frac{\frac{\partial s_2}{\partial \alpha_1} - \frac{\partial s_2}{\partial s_1} \langle p, v(\alpha_1) \rangle}{\frac{\partial \alpha_2}{\partial \alpha_1} - \frac{\partial \alpha_2}{\partial s_1} \langle p, v(\alpha_1) \rangle}.
$$
Solving $\langle p,v(\alpha_1)\rangle $ out, we have
\begin{equation}\label{solvep}
\langle p,  v(\alpha_1) \rangle
= -\frac{-\frac{\partial s_2}{\partial \alpha_1} - \frac{\partial \alpha_2}{\partial \alpha_1} \langle q, v(\alpha_2) \rangle}{\frac{\partial s_2}{\partial s_1} + \frac{\partial \alpha_2}{\partial s_1} \langle q, v(\alpha_2) \rangle}.
\end{equation}
Now let $\nu'$ be a broken ray coinciding with $\nu$ but in the opposite direction. Actually $\nu'$ is one of the family of broken rays that are associated with $\chi^{-1}$. We list the Jacobian matrix in the following 
\[
d\chi = 
\begin{bmatrix}
\begin{array}{cc}
\frac{\partial {s_2}}{\partial s_1} &  \frac{\partial {s_2}}{\partial \alpha_1}\\	
\frac{\partial \alpha_2}{\partial s_1}&  \frac{\partial \alpha_2}{\partial \alpha_1}
\end{array}
\end{bmatrix},
\qquad
d(\chi^{-1}) = (d\chi)^{-1} = \frac{1}{{\det(d\chi)}}
\begin{bmatrix}
\begin{array}{rr}
\frac{\partial \alpha_2}{\partial \alpha_1} & -\frac{\partial {s_2}}{\partial \alpha_1} \\	
-\frac{\partial \alpha_2}{\partial s_1} & \frac{\partial {s_2}}{\partial s_1} 
\end{array}
\end{bmatrix}.
\]
Notice equation (\ref{solvep}) exactly means $p$ is the conjugate point of $q$ along $\nu'$.

%%%%%%%%%%%%%%%%%%%%%%%% part2_Cancellation %%%%%%%%%%%%%%%%%%%%%%%%%%%%%%%
%%%%%%%%%%%%%%%%%%%%%%%% part2_Cancellation %%%%%%%%%%%%%%%%%%%%%%%%%%%%%%%
%%%%%%%%%%%%%%%%%%%%%%%% part2_Cancellation %%%%%%%%%%%%%%%%%%%%%%%%%%%%%%%

\section{Microlocal analysis for local problems}\label{cancellation_section}
Recall the definition of a broken ray in Section \ref{setup_sec}. We define the broken ray transform $\mathcal{B}f$ as the integral of $f$ along $\nu$
\begin{equation}
\mathcal{B}f(s,\alpha) = \int_{\nu_{s,\alpha}}{ a(y,\alpha) f(y) }dy,
\end{equation}
where $a(y,\alpha)$ is a smooth nonzero weight.
Notice $\chi$ could be the reflection operator, in which case we have a broken ray transform in the classical sense.    

Suppose $f$ has support in a compact subset away from the connecting part. Then the support of $f$ implies the transform can be interpreted as the sum of Radon transforms over two lines.
We can only expect to recover the singularities in their conormal bundle.
Thus, for a fixed broken ray $\nu_0$, we consider $(x_1,\xi^1)$ and $(x_2,\xi^2)$ in its incoming and outgoing part respectively, with $\xi^1$ and $\xi^2$ conormal to them. 
%If $x_1$ and $x_2$ are conjugate points with $xi^1$ and $x^2$ defined above, we say $(x_1,\xi^1)$ and $(x_2,\xi^2)$ are \textit{conjugate covectors}.
Let $\Gamma(\nu_0)$ be a small neighborhood of $\nu_0$, and $U_i$ be disjoint small open neighborhoods of $x_i$, $i=1, 2$.
We choose these neighborhoods small enough, such that $U_1$ is disjoint from all outgoing part and $U_2$ is disjoint from all incoming part in $\Gamma(\nu_0)$. 
We extend $U_i$ to some small conic neighborhood $V^i$ of $(x_i,\xi^i)$ in the conormal bundle, for $i=1, 2$.
\begin{figure}[h]\label{neighbor}
	\includegraphics[height=0.25\textwidth]{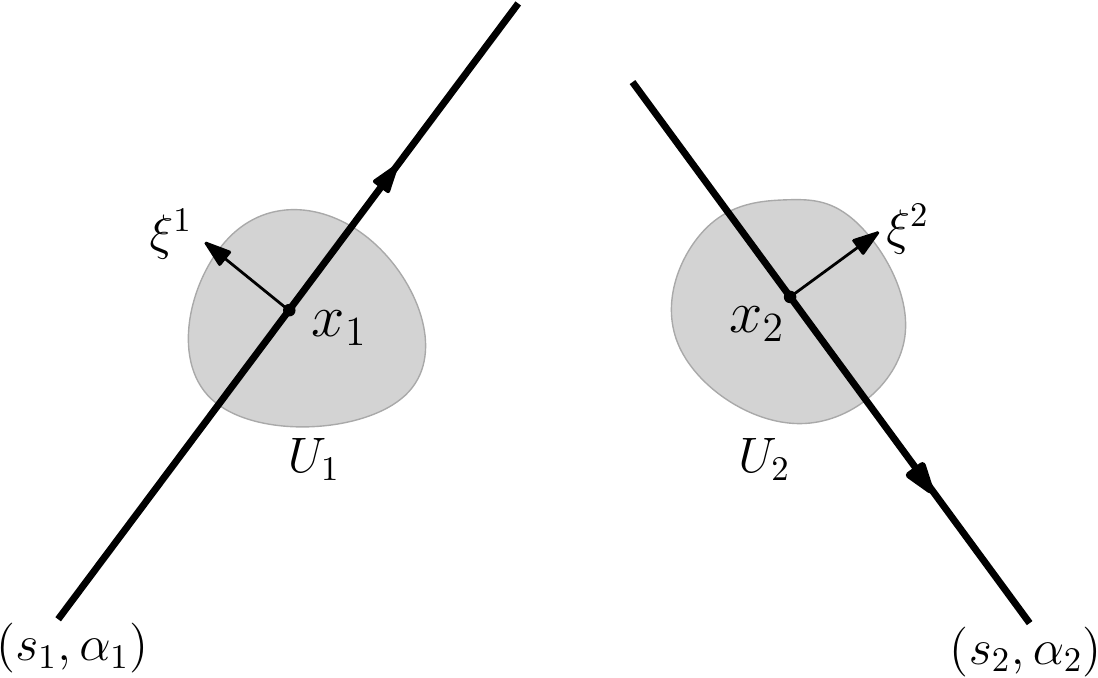}
	\caption{The small neighborhood $U_i$ and $(x_i, \xi^i)$, $i = 1,2$. }
\end{figure}

Suppose $\text{\( \WF \)}(f) \subset V^1 \cup V^2$. For convenience, we can simply assume $\supp f \subset U_1 \cup U_2$. Let $f_i$ be $f$ restricted to $U_i$, and $R_i$ be $\mathcal{B}$ restricted to distributions with wavefront set supported in $V^i$, $i=1,2$.
We have
\begin{equation}\label{B12}
\mathcal{B}f = R_1 f_1+R_2 f_2.
\end{equation}
Here we use the notation $R_i$ because in the small neighborhood $\Gamma(\nu_0)$, $\mathcal{B}f_1$ is the Radon transform of $f_1$ and $\mathcal{B}f_2$ can be regarded as the Radon Transform performing along the line $({s_2},{\alpha_2})$. 
More precisely, the restricted operators $R_1$ and $R_2$ have the following form:
$$R_1 = \phi R \varphi_1, \quad R_2 = \phi \chi^*R \varphi_2,$$
where $\phi(s,\alpha)$ is a smooth cutoff function with $\supp\phi \subset \Gamma(\nu_0)$; $\varphi_i$ are cutoff pseudodifferential operators with essential support in $V^i$, $i=1, 2$; the pull back $\chi^*g(s,\alpha) = g(\chi(s,\alpha)) $ is induced by the diffeomorphism $\chi$.
We should note that outside $\Gamma(\nu_0)$, there might be another broken ray which carries the singularities $(x_1,\xi^1)$ but with it in the outgoing part. Thus, we actually multiply $\phi$ to $\mathcal{B}$ itself as well to make equation (\ref{B12}) valid.
%instead of $(s_1,\alpha_1)$.
%%%%%%%%%%%%%%%%%%%%%%%%%%%%%%%%%%%%%%%%%%%%%%%%%%%%%%%%%%%%%%%%%%%%%%%%%%%%%%%%%%%%%%%%%%%

To analyze the canonical relation of $R_1$ and $R_2$, we need some conclusions of Radon transform.  
The weighted Radon transform $R$ is defined in the following:
\begin{equation}
R f(s,\alpha) = \int_{y \cdot w(\alpha) = s} a(y,\alpha) f(y) dy,
\end{equation}
where $w(\alpha) = (-\sin \alpha, \cos \alpha)$ as before, and $a(y, \alpha)$ is a smooth weight function.
\begin{pp}\label{Radon}
	The Radon transform $R$ is an FIO with the canonical relation 
	\begin{align*}
	C_R
	& = \{ (  
	\underbrace{ \langle y, w(\alpha) \rangle \vphantom{ \lambda \langle v, y \rangle} }_{s},
	\alpha,
	\underbrace{\lambda \vphantom{ \lambda \langle v, y \rangle} }_{\widehat{s}},
	\underbrace{\lambda \langle v(\alpha), y \rangle \vphantom{ \lambda \langle v, y \rangle}}_{\widehat{\alpha}},
	y,
	\underbrace{\lambda w \vphantom{ \lambda \langle v, y \rangle}}_{\eta}
	)
	\}.
	\end{align*}
	where $v(\alpha) = (\cos \alpha, \sin \alpha)$ and $w(\alpha)$ are defined as before.
	Specifically, $C_R$ has two components, corresponding to the choice of the sign of $\lambda$. Each component is a local diffeomorphism. The inverse is also a local diffeomorphism.
	\begin{align}\label{cr1}
	&C_{R}^{\pm}:  (y,\eta) \mapsto (s,\alpha,\lambda,\lambda \langle v,y \rangle)
	\quad \lambda = \pm |\eta|, \ \alpha = \arg(\pm \frac{\eta}{|\eta|}),\ s = \langle y, w(\alpha) \rangle\\
	&C_{R}^{-1}:  (s,\alpha,\widehat{s},\widehat{\alpha}) \mapsto (y,\eta) \quad y = \frac{\widehat{\alpha}}{\widehat{s}} v + s w, \ \eta = \widehat{s}w
	\end{align}
\end{pp}
\begin{proof}
	We write the Radon transform as 
	$$
	Rf(s,\alpha) = (2\pi)^{-1} \int \int e^{i\lambda(s-y\cdot {w}) }a(y,{\alpha}) f(y)d\lambda dy.
	$$
	The characteristic manifold is $Z = \{(s,\alpha,y) | \Phi(s,\alpha,y) =\lambda (s - y\cdot {w}) = 0\}$. Then the Lagrangian $\Lambda$ is given by
	\begin{align*}
	\Lambda = N^*Z  =
	& = \{ (s,
	\alpha,
	y,
	\underbrace{\lambda \vphantom{ \lambda \langle v, y \rangle} }_{\Phi_s},
	\underbrace{\lambda \langle v, y \rangle \vphantom{ \lambda \langle v, y \rangle}}_{\Phi_{\alpha}},
	\underbrace{-\lambda w \vphantom{ \lambda \langle v, y \rangle}}_{\Phi_y}
	), \ 
	y \cdot w = s\}.
	\end{align*}
	Therefore, the Radon transform is an FIO associated with $\Lambda$ and the canonical relation $C_R$ is obtained by twisting the Lagrangian. The sign of $\lambda$ is chosen corresponding to the orientation of $\eta$ with respect to $w$. 
	It is elliptic at $(y,\eta)$ if and only if $a(y,\alpha) \neq 0$ for $\alpha$ such that $w(\alpha)$ is colinear with $\eta$.
\end{proof}
%%%%%%%%%%%%%%%%%%%%%%%%%%%%%%%%%%%%%%%%%%%%%%%%%%%%%%%%%%%%%%%%%%%%%%%%%%%%%%%%%%%%
\begin{lm}
	Suppose $\chi$ is a diffeomorphism. Then $\chi^*$ is an FIO of which the canonical relation is a diffeomorphism
	\begin{align}\label{cchi}
	C_{\chi^*}
	& = \{ (s_1,
	{\alpha_1},
	{\widehat{s_1}},
	\widehat{\alpha_1},
	s_2,
	\alpha_2,
	\underbrace{ (\widehat{s_1},\widehat{\alpha_1})\big({d\chi}\big)^{-1} }_{(\widehat{s_2},\widehat{\alpha_2})}
	),\ 
	(s_2,\alpha_2) = \chi(s_1,\alpha_1)
	\}.
	\end{align}
\end{lm}
\begin{proof}
	The proof is similar to what we did in last proposition. Since $\chi^*: g(s_2,\alpha_2) \mapsto \chi^*g(s_1,\alpha_1) = g(\chi(s_1,\alpha_1))$ for any distribution $g$, it can be written as the following integral
	\begin{align*}
	\chi^*g(s_1,\alpha_1) & = \int \delta((s,\alpha) -\chi(s_1,\alpha_1)) g(s,\alpha) ds d\alpha\\
	&= (2\pi)^{-2}\int e^{i (\lambda_1(s - s_2)+\lambda_2(\alpha - \alpha_2)) } g(s,\alpha) d \lambda_1 d\lambda_2 ds d\alpha,
	\end{align*}
	where $\chi(s_1,\alpha_1) = (s_2,\alpha_2)$.
	The characteristic manifold is  $Z_{\chi^*} = \{(s,\alpha,s_1,\alpha_1) |\  \phi =\lambda_1( s_2 - s)+\lambda_2( \alpha_2 - \alpha) = 0 \}$. The Lagrangian is given by 
	\begin{align*}
	\Lambda_{\chi^*}
	& = \{ (s_1,
	{\alpha_1},
	s,
	\alpha,
	\underbrace{ (\lambda_1,\lambda_2)(d\chi)}_{\Phi_{s_1,\alpha_1}},
	\underbrace{ -(\lambda_1,\lambda_2) }_{\Phi_{s,\alpha}}
	),\ 
	(s,\alpha) = \chi(s_1,\alpha_1)
	\}.
	\end{align*}
	Let $(\lambda_1,\lambda_2)(d\chi) = (\widehat{s_1},\widehat{\alpha_1})$ and replace $(s,\alpha)$ by $(s_2,\alpha_2)$, we get the canonical relation as is shown above (\ref{cchi}).
\end{proof}
The composition operator $\chi^* R$ is also an FIO of which the canonical relation $C_{\chi^* R}= C_{\chi^*} \circ C_{R}$ is a local diffeomorphism. 
Additionally, since the multiplication of cutoff functions does not influence the Lagrangian, therefore, the restricted operator $R_i$ has the same canonical relation as above, call it $C_i$, in the small neighborhood of $\nu_0$, $i=1,2$.

Suppose $(s_1,\alpha_1,\widehat{s_1},\widehat{\alpha_1})$ and $(s_2,\alpha_2,\widehat{s_2},\widehat{\alpha_2})$ are images of $(x_1,\xi^1)$ and $(x_2,\xi^2)$ under $C_R$.
That is, with $s_i$and $\alpha_i$ given by (\ref{cr1}), we have
$$
(\widehat{s_1},\widehat{\alpha_1}) = \lambda_1(1, \langle x_1, v(\alpha_1)\rangle), \quad 
(\widehat{s_2},\widehat{\alpha_2}) = \lambda_2(1, \langle x_2, v(\alpha_2)\rangle).
$$
Then from the analysis above, $C_1(x_1,\xi^1)=C_2(x_2,\xi^2)$ if and only if
$$(s_2,\alpha_2) = \chi(s_1,\alpha_1), \quad
(\widehat{s_2},\widehat{\alpha_2}) = (\widehat{s_1},\widehat{\alpha_1})\big( d \chi \big)^{-1}$$
The first equality says there is a broken ray $\nu$ of which $(s_1,\alpha_1)$ and $(s_2,\alpha_2)$ are the incoming and outgoing part.  
The second condition is equivalent to 
\begin{align}\label{yyy}
\lambda_2 (1,\langle x_2,v(\alpha_2) \rangle) = \frac{\lambda_1}{{\det(d\chi)}} \big( \frac{\partial \alpha_2}{\partial \alpha_1}-
\langle x_1, v(\alpha_1) \rangle \frac{\partial \alpha_2}{\partial s_1}, -(\frac{\partial s_2}{\partial \alpha_1}-
\langle x_1, v(\alpha_1) \rangle \frac{\partial s_2}{\partial s_1}) \big).
\end{align}
Notice 
$\frac{\partial \alpha_2}{\partial \alpha_1}-
\langle x_1, v(\alpha_1) \rangle \frac{\partial \alpha_2}{\partial s_1}$
and $\frac{\partial s_2}{\partial \alpha_1}-
\langle x_1, v(\alpha_1) \rangle \frac{\partial s_2}{\partial s_1}$ are exactly $\frac{d \alpha_2}{d \alpha_1}$ and $\frac{d s_2}{d \alpha_1}$ if we fixed $x_1$ and consider $s_2, \alpha_2$ as functions of one variable $\alpha_1$. Therefore
(\ref{yyy}) is true if and only if 
\begin{itemize}
	\item[(a)] $\langle x_2,v(\alpha_2) \rangle = -\big(\frac{d \alpha_2}{d \alpha_1}\big)^{-1}\frac{d s_2}{d \alpha_1}$, which implies $x_1$ and $x_2$ are conjugate points along $\nu$,
	\item[(b)] $\lambda_2 =  \frac{\lambda_1}{\det(d\chi)}\big(\frac{d \alpha_2}{d \alpha_1}\big)$, with $\xi_1 = \lambda_1 w(\alpha_1)$ and $\xi_2 = \lambda_2 w(\alpha_2)$.
\end{itemize}
\begin{theorem}\label{cancelofC}
	%[Range!!]%$\mathcal{B}_1^{-1}\mathcal{B}_2 = 
	We have $ C_1(x,\xi) = C_2(y,\eta)$ if and only if there is a broken ray $\nu$ joining $x$ and $y$ such that
	\begin{itemize}
		\item[(a)] $x$ and $y$ are conjugate points along $\nu$.
		\item[(b)] $\xi$ and $\eta$ satisfies $\xi = \lambda w(\alpha_1), \eta = \frac{\lambda}{\det(d\chi)} \big(\frac{d \alpha_2}{d \alpha_1}\big) w(\alpha_2)$ for some $\lambda \neq 0$, where $\alpha_1$ is the angle of the incoming part and  $\alpha_2$ is the angle of the reflected part of $\nu$.
	\end{itemize} 
	%Here we suppose $\nu$ has the incoming part $(s_1,\alpha_1)$ and the reflected part $(s_2,\alpha_2)$. 
\end{theorem}
For $(y,\eta)$ satisfying this theorem, we call it the \textit{conjugate covectors} of $(x,\xi)$.
Since $C_1$ is a local diffeomorphism, it maps a small neighborhood of $(x,\xi)$ to a small neighborhood of $(s_1,\alpha_1,\widehat{s_1},\widehat{\alpha_1})$. 
The similar is true with $C_2$. 
Then by shrinking $V_1$ and $V_2$ a bit, we can assume $C_1(V^1) =C_2(V^2) \equiv \mathcal{V}$. 

Notice $R_1$ is elliptic at $(x,\xi)$. 
An application of the parametrix $R_1^{-1}$ to $\mathcal{B} (f_1+f_2)$ shows
$$
R_1^{-1} \mathcal{B}(f_1+f_2) = f_1 + F_{12} f_2,
$$
where we define $F_{12} = R_1^{-1} R_2$. Then $F_{12}$ is an FIO with canonical relation $C_{12} = C_1^{-1} \circ C_2: V^2 \to V^1 $.  We can also define $F_{21}$ and $C_{21}$ in a similar way.
%%% components!!!
%%% If we consider the sign of $\lambda$, $V^1$ and $V^2$ naturally have two components $V^1 = V^1_+ \cap V^1_-$ and $V^2 = V^2_+ \cap V^2_-$, where we use $V^i_+$ to correspond to the choice of $\lambda >0$.
This proves the following.
\begin{theorem}
	Suppose $f_j \in \mathcal{E}'(U_j)$ with $\text{\( \WF \)}(f_j) \subset V^j$, $j = 1,2$. Then 
	$$
	\mathcal{B} (f_1+f_2) \in H^s(\mathcal{V})
	$$
	if and only if 
	$$
	f_1 + F_{12} f_2 \in H^{s-1/2}(V^1) \Leftrightarrow
	F_{21} f_1 +  f_2 \in H^{s-1/2}(V^1).
	$$
\end{theorem}
Thus, given a distribution $f_1$ singular in $ V^1$, there exists a distribution $f_2$ singular in $V^2$ such that $\mathcal{B}(f_1+f_2)$ is smooth. One possible choice is $ f_2 = - F_{12} f_1$. It is also the only choice if we consider it up to smooth functions. If we introduce the definition of the microlocal kernel as in \cite{Holman2017}, then for any $h$ with the wave front set in $V^1$, $h- F_{12}h$  is in the microlocal kernel of $\mathcal{B}$. This implies the reconstruction of $f=f_1$ always has some error in form of  $h- F_{12}h$ for some $h$. In other words, the singularities of $f$ cannot be resolved from the singularities of $\mathcal{B}(f)$ and it can only be recovered up to an error in the microlocal kernel. A more detailed description of the kernel is in \cite{Holman2017}.

With the notation above, we are going to find out the artifacts arising when we use the backprojection $\mathcal{B}^*\mathcal{B}$ to reconstruct $f$. Without loss of generality, we assume the weight $a(y,\alpha) = 1$ in the following. Suppose $\nu$ is  the broken ray in Theorem \ref{cancelofC}. In the small neighborhood of $\nu$, we have
\begin{equation}\label{normallocal}
\mathcal{B}^*\mathcal{B}f = R_1^*R_1 f_1+R_1^*R_2 f_2+R_2^*R_1 f_1+R_2^*R_2 f_2.
\end{equation}
Recall $R_1$ and $R_2$ are defined microlocally. On the one hand, the assumption on $f_i$ play the same role as restricting the operator on $U^i$, $i = 1, 2$. For simplification, we just ignore them. On the other hand, if we concentrate on the small neighborhood of $\nu$, then we exclude the broken ray that carries $(x,\xi)$ on its outgoing part. Microlocally $R_1$ is equivalent to the Radon transform operator near $(x,\xi)$, which indicates $R_1^*R_1$ is an elliptic pseudodifferential operator of order $-1$ for distributions singular near $(x,\xi)$. Especially, it has the principal symbol
$\frac{4 \pi}{|\xi|} $. The similar is true for $R_2^*R_2$.

One can follow the same argument in \cite{MR3339183,Holman2017} to show the properties of the normal operators. Additionally, similar to Radon transform, we can apply a filter to the  backprojection to get a better reconstruction. Thus, we have
$$
\mathcal{B}^* \Lambda \mathcal{ B} = R_1^* \Lambda R_1 f_1+R_1^*\Lambda R_2 f_2+R_2^* \Lambda R_1 f_1+R_2^* \Lambda R_2 f_2,
$$
where $\Lambda = \frac{1}{4\pi}\sqrt{- \Delta_s}$ is a self-adjoint operator.

The canonical relation of $R_i^*$ is the inverse of that of $R_i$, and therefore by Egorov's theorem \cite{Hoermander2009}, $R_i^* \Lambda R_i$ is a pseudodifferential operator of order $0$ with principal symbol $\sigma(R_i^*R_i) (\sigma(\Lambda)\circ g_i) $, where $\sigma(P)$ refers to the principal symbol of $P$ and $g_i$ is the canonical transformation corresponding to  $C_{R_i}$ for $i=1,2$. 
Recall Proposition \ref{Radon}, we have $\sigma(\Lambda)\circ g_i = \frac{1}{4\pi}|\xi|$, which implies 
$$
R_i^* \Lambda R_i \equiv  I \mod \Psi^{-1},\quad i = 1,2.
$$
This also coincides with the inversion formula for Radon transform. Then with the observation $ R_1^*\Lambda R_2 F_{21} = R_1^*\Lambda R_1$ and $F_{21}R_1^*\Lambda R_2 \equiv I$, we have $ R_1^*\Lambda R_2 \equiv F_{12}$ up to a lower order.
The same is true with $R_2^*\Lambda R_1$.
Notice he following calculations are all microlocal and up to order $-1$.\\
Hence, we have 
\begin{align}
\mathcal{B}^*\Lambda \mathcal{B}
\equiv
\begin{bmatrix}
\begin{array}{ll}
\text{Id} &  F_{12}\\	
F_{12}^{-1}&  \text{Id}
\end{array}
\end{bmatrix}
\coloneqq M ,
\end{align}
where we follow the convention in \cite{Holman2017} to think $f = f_1 + f_2$ as vector functions. This implies when performing the filtered backprojection, the reconstruction has two parts of artifacts, $F_{12} f_2$ in $V^1$ and $F_{21} f_1$ in $V^2$. In \cite{Holman2017}, it is also shown that $F_{12}$ and $F_{21}$ are principally unitary in $H^{-\frac{1}{2}}$, and the artifacts have the same strength.

Next, consider the numerical reconstruction by using the Landweber iteration. 
We still focus on the local problem, that is, we consider $\mathcal{B}f$ in the small neighborhood $\Gamma(\nu_0)$ of fixed $\nu_0$. 
With the notation above, we use a slightly different Landweber iteration to solve the equation $\mathcal{B} f = g$, with $g$ being the local data and in the range of $\mathcal{B}$. 
%Instead of applying the adjoint $\mathcal{B}^*$, we apply $\mathcal{B}^\#$ and write the equation in form of 
Here, we set $\mathcal{L} = \Lambda^{\frac{1}{2} }\mathcal{B}$ to have
\begin{equation}\label{landweber}
(\text{Id} - (\text{Id}- \gamma \mathcal{L}^* \mathcal{L})) f = \gamma \mathcal{L}^*\Lambda^{\frac{1}{2} } g.
\end{equation}
Then with a small enough and suitable $\gamma > 0$, it can be solved by the Neumann series
$$
f = \sum_{k=0}^{+\infty}(\text{Id}- \gamma \mathcal{L}^*\mathcal{L})^k 
\gamma \mathcal{L}^*\Lambda^{\frac{1}{2} } g.
$$
Suppose the original function is $f = f_1 + f_2$. We track the terms of highest order, that is, order zero, to have the approximation sequence 
$$
f^{(n)} = \sum_{k=0}^{n}(\text{Id}- \gamma M)^k \gamma M f.
$$
With the observation $M^k = 2^{k-1}M$ for $k \geq 1$, a straightforward calculation shows 
$$
f^{(n)} = \gamma (\sum_{k=0}^{n} (1-2\gamma)^k) Mf.
$$
The numerical solution is
$$
f^{(n)} \to
\frac{1}{2} 
\begin{bmatrix}
\begin{array}{l}
f_1 + F_{12}f_2\\	
F_{21}f_1 +  f_2
\end{array}
\end{bmatrix},\ \text{as} \ n \to \infty.
$$
Therefore, the error equals to $\frac{1}{2}(f_1 - F_{12}f_2) + \frac{1}{2}(f_2 - F_{21}f_1)$, which belongs to the microlocal kernel.
%%%%%%%%%%%%%%%%%%%%%%%% part3_Reflectioneg %%%%%%%%%%%%%%%%%%%%%%%%%%%%%%%
%%%%%%%%%%%%%%%%%%%%%%%% part3_Reflectioneg %%%%%%%%%%%%%%%%%%%%%%%%%%%%%%%
%%%%%%%%%%%%%%%%%%%%%%%% part3_Reflectioneg %%%%%%%%%%%%%%%%%%%%%%%%%%%%%%%
\section{The V-line Radon transform}\label{reflection_section}
In this section we are going to apply the conclusions to the V-line Radon transform, that is, the case when $\chi$ is a reflection and obeys the law in geometric optics. We consider the situation when the weight $a(y,\alpha) = 1$. First we verify the reflection operator is a diffeomorphism. Then it is followed by the potential cancellation of singularities due to the existence of conjugates points. We derive an explicit formula to illustrate when conjugate points exist in this case.
\subsection{The Diffeomorphism}
Suppose $\Omega$ is a bounded domain with a smooth and negatively oriented boundary, which can be parameterized as a regular curve $\gamma$. This allows us to choose its arc length parameterization:
${\gamma}(\tau) = (x(\tau),y(\tau))$.
The unit tangent vector and unit outward normal are
$
\dot{\gamma}(\tau) = (\dot{x}(\tau),\dot{y}(\tau))$ and ${n}(\tau) = (-\dot{y}(\tau),\dot{x}(\tau))
$ respectively,
where $\dot{f}(\tau)$ refers to $\frac{d f}{d \tau}$. We consider the local problems, and $\gamma$ could be just one part of $\partial \Omega$.
%%%%%%%%%%%%%%%%%
Since $\gamma(\tau)$ is unit speed, the signed curvature of $\gamma$ is defined as the scalar function $\kappa(\tau)$ such that $\ddot{\gamma} =  \kappa(\tau) n $. Additionally, we have $\dot{n} =  -\kappa(\tau) \dot{\gamma}$, which will be used later.

%%%%%%%%%%%%%%%%%%%%
Suppose a ray $(s,\alpha)$ transversally hits $\partial \Omega $ at point $\gamma(\tau_0)=(x(\tau_0),y(\tau_0)) $ and then reflects, as is shown in Figure \ref{picref}.
In a small neighborhood of such a fixed ray, $\tau_0$ is a smooth function of $s$ and $\alpha$ . The proof is simply an application of implicit function theorem.
Since $\gamma(\tau_0)$ satisfies $F(\tau_0,s,\alpha) =\langle{w(\alpha)},\gamma(\tau_0) \rangle - s = 0 $ and we have $\frac{\partial F}{\partial \tau_0} = \langle w(\alpha), \dot{\gamma}(\tau_0)\rangle \neq 0$, $\tau_0$ could be written as a smooth function $\tau_0(s,\alpha)$.

\begin{figure}[h]
	\includegraphics[height=0.2\textwidth]{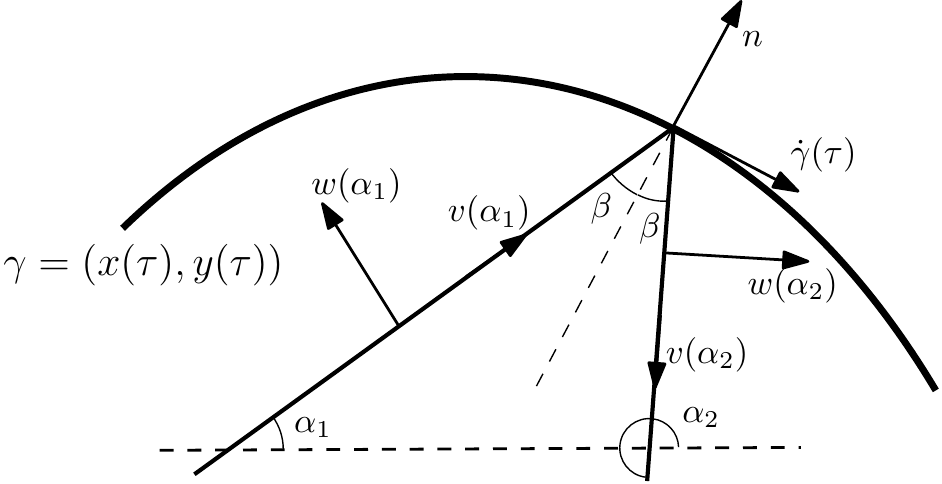}
	\caption{A sketch of a broken ray reflected on a smooth boundary and the notation.}
	\label{picref}
\end{figure}

Differentiating $F(\tau_0,s,\alpha)=0$ with respect to $s$ and ${\alpha}$, 
we get an equation of $\frac{\partial \tau_0}{\partial s}$ and $\frac{\partial \tau_0}{\partial \alpha}$. To distinguish $(s,\alpha)$ from the one we use for $l_{2}$, we replace them by $(s_1,\alpha_1)$ in the following
\begin{equation}\label{tauequ}
\frac{\partial \tau_0}{\partial s_1}  = \frac{1}{\langle {w(\alpha_1)}, \dot{\gamma}\rangle} \equiv k_s, 
\qquad 
\frac{\partial \tau_0}{\partial \alpha_1}  = \frac{\langle {v(\alpha_1)},\gamma(\tau_0)\rangle}{\langle {w(\alpha_1)},\dot{\gamma}\rangle} \equiv {k}_{\alpha}.
\end{equation}

%%%%%%%%%%%%______________________Claim 2___________________________________________
\begin{cl*}
	The map $\chi : (s_1,{\alpha_1})\mapsto (s_2,\alpha_2) $ is a local diffeomorphism.
\end{cl*}
\begin{proof}
	As is shown in Figure \ref{picref}, the reflection follows the rules:
	\begin{equation}\label{refangle}
	\begin{cases}
	{\alpha_2} = \alpha_1 + 2 \beta + \pi\\
	{s_2} = \langle \gamma(\tau_0) , {w(\alpha_2)}\rangle
	\end{cases}
	\end{equation}
	where $\beta \in (-\frac{\pi}{2},\frac{\pi}{2})$ is the incident angle and $\beta <0$ represents that $v(\alpha_1)$ has negative projection along 
	$\dot{\gamma}$.
	
	Since $\sin \beta = \langle v(\alpha_1), \dot{\gamma}(\tau_0)\rangle$,  $\beta$ is a smooth function of $s_1$ and $\alpha_1$, which has the derivative
	$$\label{beta}
	\frac{\partial {\beta}}{\partial s} = \kappa k_s, \qquad \frac{\partial {\beta}}{\partial \alpha_1} = \kappa k_\alpha -1,
	$$
	where $\kappa$ is the signed curvature.
	This is followed by  
	\begin{equation}\label{t_alph}
	\frac{\partial {\alpha_2}}{\partial s_1} = 2 \kappa k_s,
	\qquad \frac{\partial {\alpha_2}}{\partial \alpha_1} = 2\kappa k_\alpha -1.
	\end{equation}
	And 
	\begin{align*}
	&\frac{\partial {s_2}}{\partial {s_1}} = 
	\langle \frac{\partial {w(\alpha_2)}}{\partial s_1},\gamma(\tau_0) \rangle+\langle {w(\alpha_2)}, \frac{\partial \gamma(\tau_0)}{\partial s_1}\rangle 
	= - \langle {v(\alpha_2)},\gamma(\tau_0) \rangle \frac{\partial {\alpha_2}}{\partial s_1} + k_s \langle {w(\alpha_2)}, \dot{\gamma} \rangle,
	\\
	&\frac{\partial {s_2}}{\partial {\alpha_1}} = 
	\langle \frac{\partial {w(\alpha_2)}}{\partial \alpha_1},\gamma(\tau_0) \rangle+\langle {w(\alpha_2)}, \frac{\partial \gamma(\tau_0)}{\partial \alpha_1}\rangle
	= - \langle {v(\alpha_2)},\gamma(\tau_0) \rangle \frac{\partial {\alpha_2}}{\partial \alpha_1} + k_\alpha \langle {w(\alpha_2)}, \dot{\gamma}\rangle.\\
	\end{align*}
	
	By row reduction, we have
	\[
	\det{(d\chi)} =\det
	\begin{bmatrix}
	\frac{\partial {s_2}}{\partial {s_1}}  & \frac{\partial {s_2}}{\partial \alpha_1}\\[1.3ex]
	\frac{\partial {\alpha_2}}{\partial {s_1}} & \frac{\partial {\alpha_2}}{\partial \alpha_1} \\
	\end{bmatrix}
	=\det
	\begin{bmatrix}
	k_s \langle {w(\alpha_2)}, \dot{\gamma} \rangle  & k_\alpha \langle {w(\alpha_2)}, \dot{\gamma} \rangle\\[1.3ex]
	\frac{\partial {\alpha_2}}{\partial {s_1}} & \frac{\partial {\alpha_2}}{\partial \alpha_1} \\
	\end{bmatrix}.
	\]
	Thus,
	\begin{align*}
	\det{(d\chi)} 
	&=\langle {w(\alpha_2)}, \dot{\gamma} \rangle \det
	\begin{bmatrix}
	k_s   & k_\alpha \\[1.3ex]
	2\kappa k_s & 2\kappa k_\alpha -1 \\
	\end{bmatrix}
	= -\langle w(\alpha_1), \dot{\gamma} \rangle (-k_s)  = 1.
	\end{align*}
	When $l_1$ hits the boundary transversally, the differential of $\chi$ is nonzero and the reflection map $\chi$ is a local diffeomorphism.
\end{proof}

%%%%%%%%%%%%%%%%%%%%%%%%%%%
%%%%%%%%%%%%%%%%%%%%%%%%%%%
%%%%%%%%%%%%%%%%%%%%%%%%%%%
%%%%%%%%%%%%%%%%%%%%%%%%%%%
\subsection{Conjugate Points}
The incoming ray ${l_{1}}(t)$ and reflected ray ${l_{2}}(t) $ are given in the following
$$
\begin{cases}
{l_{1}}(t) = {p} + t {v(\alpha_1)} , \quad  & 0 \leq t \leq t_1,\\
{l_{2}}(t) = {\gamma(\tau_0)} + (t-t_1){{v(\alpha_2)}}, \quad & t\geq t_1.
\end{cases}
$$
where $q_0 = \gamma(\tau_0)$ is the intersection point on the boundary. Compared with (\ref{exponential}), now $q_0$ connects $l_1$ and $l_2$ and $t_1 = t_2$. We use $t_1$ instead of $t_2$ in the following.
By equation %(\ref{computeformula})
(\ref{tauequ})(\ref{t_alph}), a straightforward calculation shows
$$
\frac{d {{\alpha_2}}}{d \alpha_1} =  \frac{2\kappa  t_1}{\langle w(\alpha_1), \dot{\gamma} \rangle} -1 
,\quad \frac{d {{q_0}}}{d \alpha_1} =  \frac{t_1}{\langle w(\alpha_1), \dot{\gamma} \rangle}\dot{\gamma}
%,\quad  \frac{d {{s_2}}}{d \alpha_1} =  \frac{d {\langle q_0,w(\alpha_2)\rangle}}{d \alpha_1}= -t_1 - \langle q_0,v(\alpha_2) \rangle \frac{d {{\alpha_2}}}{d \alpha_1} 
.
$$
where $t_1 = \langle q_0-p,v(\alpha_1)\rangle$ is the time or length from $p$ to $q_0$.
Plugging these back into (\ref{det}), we get the determinant is $(t-t_1) \big( \frac{d {{\alpha_2}}}{d \alpha_1} \big)-t_1$. Especially, the matrix is in the following,
\begin{align}\label{expmatrix}
\begin{split}
d \text{\( \expo \)}_{{p}}(t {v_1})=
\left[
\begin{array}{cc}
v(\alpha_2), & ((t-t_1) \big( \frac{d {{\alpha_2}}}{d \alpha_1} \big)-t_1)w(\alpha_2)
\end{array}
\right].
\end{split}
\end{align}
\begin{corollary}\label{conjugatecor}
	Suppose an incoming ray $l_1$ hits the boundary $\gamma$ transversally at point $\gamma(\tau_0)$ and then reflects. Then 
	\begin{itemize}
		\item[(a)] $p$ on $l_1$ has a conjugate point $q$ in $l_2$ if and only if  $\frac{d \alpha_2 }{d \alpha_1} > 0$, 
		specifically, if and only if it satisfies $\kappa(\tau_0) < \frac{\langle w(\alpha_1), \dot{\gamma}(\tau_0)\rangle}{2t_1}$.
		\item[(b)]If this happens, $q$ is uniquely determined by $\Delta t_2 =  \big( \frac{d {{\alpha_2}}}{d \alpha_1} \big)^{-1} \Delta t_1$, where$\Delta t_1 = t_1$ is the time or length from $p$ to $\gamma(\tau_0)$ and $\Delta t_2 = \langle q-\gamma(\tau_0),v(\alpha_2)\rangle$ is the time or length from $q$ to $\gamma(\tau_0)$.
	\end{itemize}
\end{corollary}
The statement (a) comes from the observation that the other factor $\langle \frac{d {q_0}}{d \alpha_1},  w(\alpha_2) \rangle$ in Proposition \ref{conjugatepp}(a) is always negative in the reflection case, as shown in Figure \ref{picref}.
\begin{figure}[h]
	\includegraphics[height=0.2\textwidth]{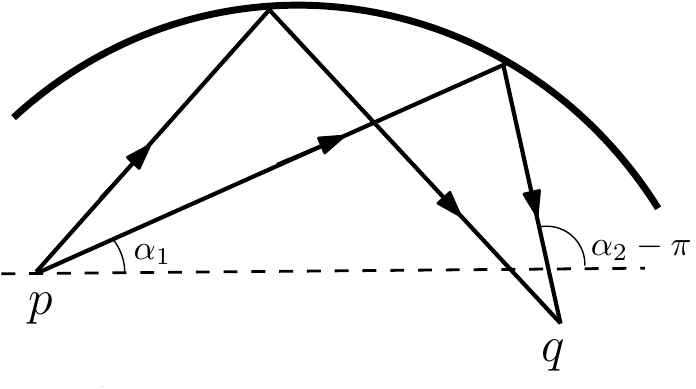}
	\caption{Two broken rays intersect when $\alpha_2$ increases as $\alpha_1$ increases.}
	\label{two}
\end{figure} 
This statement has a straightforward geometrical explanation, see Figure \ref{two}. There are conjugate points if and only if $\alpha_2$ increases as $\alpha_1$ increases.
 
For negatively oriented smooth curve which is the boundary of a convex set,  the curvature $\kappa < 0 $ and the inner product  $\langle w(\alpha_1), \dot{\gamma}\rangle < 0$.  The inequality actually says
$|\kappa(\tau_0)| > \frac{|\langle w(\alpha_2), \dot{\gamma}\rangle|}{2t_1}$. 
Additionally, 
since $\langle w(\alpha_1), \dot{\gamma}\rangle = - \cos \beta$ where $\beta$ is the incident and the reflected angle, each component involved in the criterion is geometrical and therefore is invariant regardless of what kind of parameterization we choose for the boundary.
We should mention the equation in (b) coincides with the Generalized Mirror Equation in \cite{Boyle2015} but is in different form and is derived from the perspective of the exponential map.

\begin{eg}
	Consider a parabolic mirror $-4ay = x^2$, which has the focus at $(0,-a)$. Suppose there is a light source located at the point $p=(0,-d)$. Here $a$ and $d$ are positive constants we are going to choose later. We want to know in which directions of the light from $p$ there are conjugate points and this will verify the criterion of conjugate points.
	
	Let $\gamma(x) = (x,-\frac{x^2}{4 a})$ be the boundary curve. The intersection point is $q_0= \gamma(x_0)$. Then the incoming ray has the direction along $\overrightarrow{pq_0}$, and $w(\alpha_1), \dot{\gamma}(x_0), \kappa(x_0), t_1$ can be calculated directly by definition.

	After simplification, the criterion is equivalent to
	$$
	(a-d)(\frac{3}{4}x_0^2 - ad) >0.
	$$
	We have the following three cases.\\
	\textit{case 1:} If $d > a$,  $p$ has conjugate points if and only if the incoming ray hits the boundary at the region $x^2 < \frac{4}{3}ad$, as is shown in Figure \ref{parabolic}(a).  \\
	\textit{case 2:} If $d < a$,  $p$ has conjugate points if and only if the incoming ray hits the boundary at the region $x^2 > \frac{4}{3}ad$, as is shown in Figure \ref{parabolic}(b).\\
	\textit{case 3:} If $d = a$,  $p$ has no conjugate points for all directions, which coincides with the fact that all rays of light emitting from the focus reflect and travel parallel to the y axis, as is shown in Figure \ref{parabolic}(c).\\
	\begin{figure}[h]
		\centering
		\begin{subfigure}{0.3\textwidth}
			\centering
			\includegraphics[width=\linewidth]{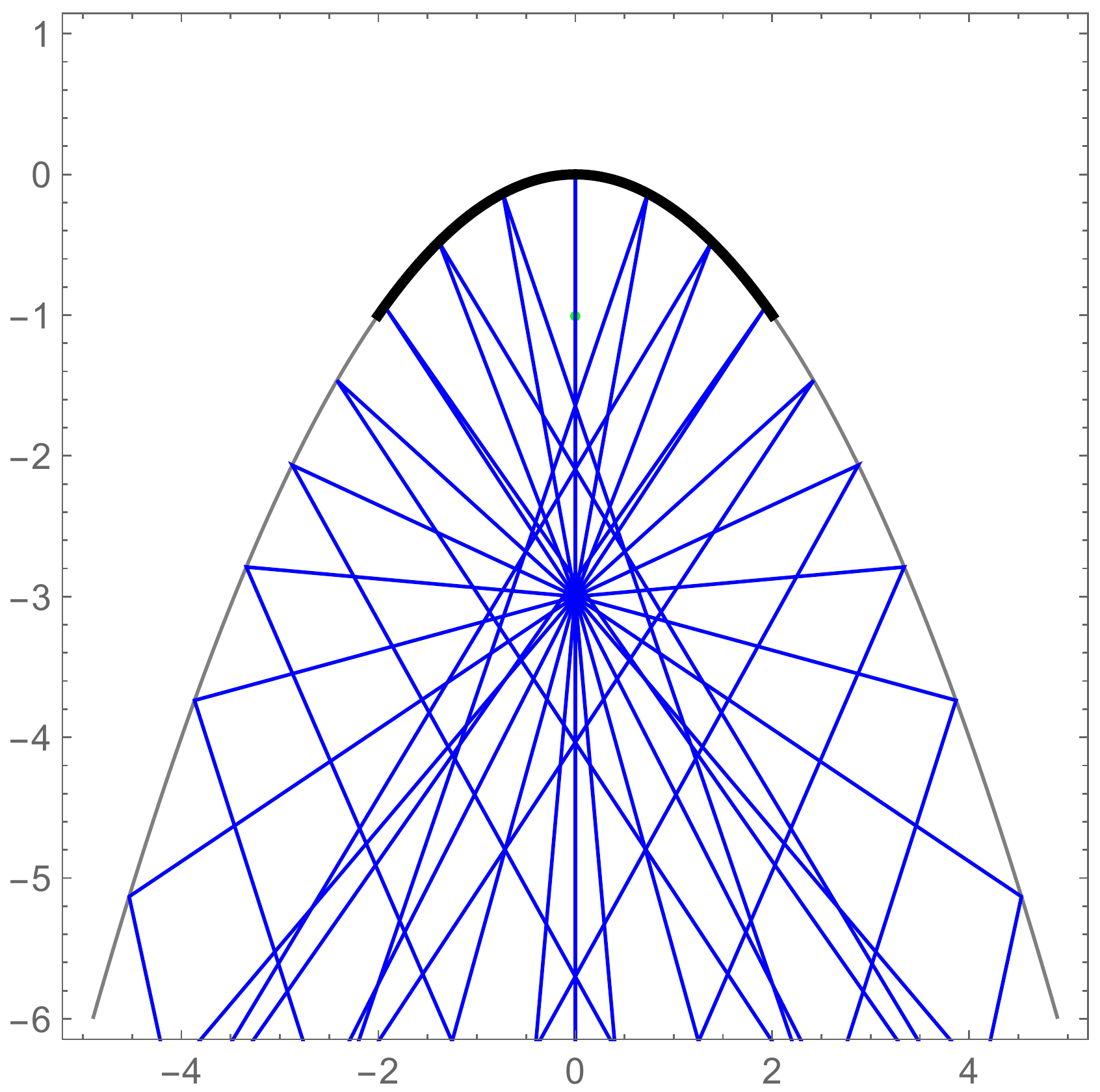}
			\caption{$a=1, d=3$} %\implies |x| <2$ }
		\end{subfigure}%
		\begin{subfigure}{0.3\textwidth}
			\centering
			\includegraphics[width=\linewidth]{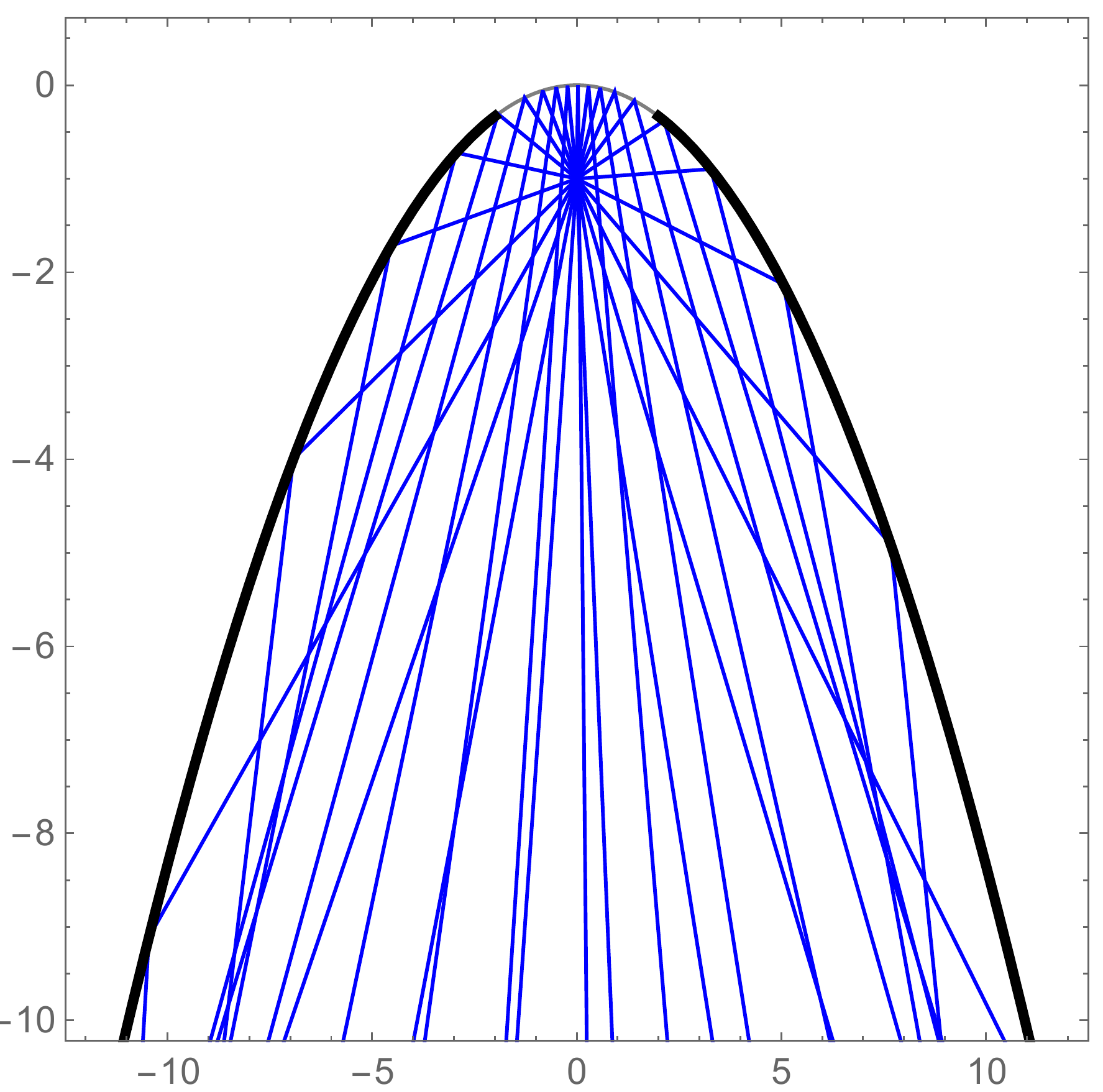}
			\caption{$a=3, d=1$}%\implies|x|>2$} %The red part is the intersection region where the incoming rays hit $\partial \Omega$ and reflect with conjugate points.%}
		\end{subfigure}
		\begin{subfigure}{0.3\textwidth}
			\centering
			\includegraphics[width=\linewidth]{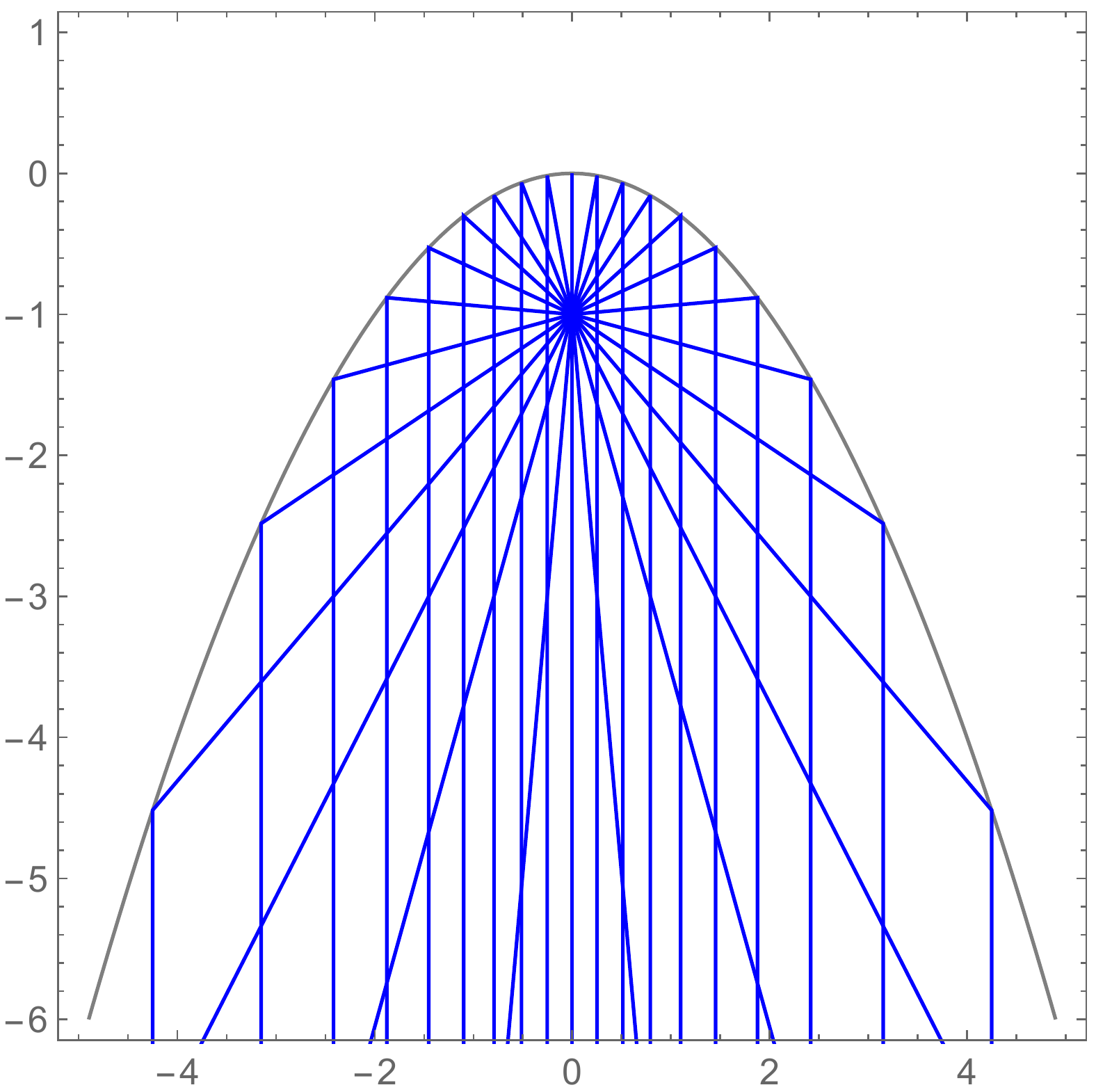}
			\caption{$a=1$, $d=1$}
		\end{subfigure}
		\caption{In (a) and (b), the bold part is the intersection region where the incoming rays hit there and reflect with conjugate points. }
		\label{parabolic}
	\end{figure}
	
\end{eg}

\begin{eg}\label{egdisk0}
	The second example is to illustrate that we have different type of conjugate points, specifically fold and cusps, if we have a circular mirror with the light source inside.
	For simplification, we use $\alpha$ to replace $\alpha_1$ in this example.
	We follow the notations in the paper \cite{MR3339183}. The \textit{tangent conjugate locus} $S(p)$ is the set of all vectors $V$ such that the differential of the exponential mapat $p$ is not a isomorphism. The kernel of $d \text{\( \expo \)}_{{p}}(V)$ is denoted by $N_p(V)$.
	By the previous results, 
	\begin{equation*}
	S(p) = \{V = t(\cos\alpha, \sin\alpha), \text{s.t.}  \ F(\alpha, t) = (\frac{2\kappa  t_1}{\langle w(\alpha), \dot{\gamma} \rangle} -1 )(t-t_1) - t_1 =0\}
	\end{equation*}
	We fix $V_0 \in S(p)$. By (\ref{expmatrix}), the differential $d \text{\( \expo \)}_{{p}}(V_0)$ has the matrix form $ [v(\alpha_2),0]$, which shows $N_p(V_0)$ is spanned by $\frac{\partial }{\partial \alpha}$. The conjugate vector $V_0$ is called \textit{of fold type}, if $\frac{\partial F}{\partial \alpha} \neq 0$ for all $(\alpha,t)$ that satisfies $F(\alpha,t) = 0$. Otherwise, we may have cusps. We will show in the following that the cusps exist in some cases.
	
	We assume the origin $O$ is at the center of the circle and the source is not there.  Suppose the mirror has radius $1$, then $\kappa = -1 $. The tangent conjugate locus is the zero level set of $F(\alpha, t)$ given by 
	\begin{equation*}
	F = (\frac{2 t_1}{\cos \beta}-1) t-2 \frac{t_1 ^ 2}{\cos \beta}.
	\end{equation*} 
	Recall $t_1$ and $\cos \beta$ are smooth functions of $\alpha$. A straightforward calculation shows
	\begin{equation*}
	\frac{\partial F}{\partial \alpha} = \frac{6t_1^2 \sin \beta (t_1 - \cos \beta )}{\cos ^2 \beta(2 t_1 - \cos \beta)}.
	\end{equation*}
	
Suppose we have conjugate points, then $2 t_1 - \cos \beta >0$. The incidence angle $\beta \in (-\frac{\pi}{2},\frac{\pi}{2})$ so we at most have two zeros for $\frac{\partial F}{\partial \alpha}$,
\begin{itemize}
	\item $\beta = 0$, which means the incoming ray and reflected ray coincide. This is a simple zero, because $\frac{d }{d \alpha} \sin \beta= t_1 - \cos \beta = t_1 - 1 \neq 0$. 
	\item $\cos \beta -t_1 = 0$ is true for some $\alpha_0$. This happens when $pO$ is perpendicular to the incoming ray. We check $\frac{d }{d \alpha}(\cos \beta -t_1) = \sin \beta \neq 0$. This is also a simple zero. 
\end{itemize}
	
%	Suppose we have conjugate points, then $2 t_1 - \cos \beta >0$. The incidence angle $\beta \in (-\frac{\pi}{2},\frac{\pi}{2})$ so we at most have two zeros for $\frac{\partial F}{\partial \alpha}$, that is, if $\cos \beta -t_1 = 0$ is true for some $\alpha_0$. This happens  when $pO$ is perpendicular to the incoming ray. We check that $\frac{d }{d \alpha}(\cos \beta -t_1) = \sin \beta \neq 0$, which means this is a simple zero. By definition we have a cusp.	
\end{eg}

%%%%%%%%%%%%%%%%%%%%%%%% part31_Recoverable %%%%%%%%%%%%%%%%%%%%%%%%%%%%%%%
%%%%%%%%%%%%%%%%%%%%%%%% part31_Recoverable %%%%%%%%%%%%%%%%%%%%%%%%%%%%%%%
%%%%%%%%%%%%%%%%%%%%%%%% part31_Recoverable %%%%%%%%%%%%%%%%%%%%%%%%%%%%%%%
\subsection{Numerical Examples}
This subsection aims to illustrate the artifacts arising in the reconstruction of V-line Radon transform by numerical experiments.
We use $\Gamma$ to denote the smooth family of all broken rays chosen for tomography. 
We say $(x,\xi)$ is \textit{visible} if there is a broken ray $\gamma$ in the family of tomography such that $(x,\xi)$ is in the conormal bundle of $\gamma$ excluding the connecting part. The fact that $(x,\xi)$ is visible does not necessarily imply that $(x,\xi)$ is recoverable.
%when we use $\mathcal{B}^*\mathcal{B}$ as an reconstruction attempt.
% and $\chi$ is the reflection on the boundary.
\begin{eg}\label{egdisk}
	In this example, we use filtered backprojection to recover $f$, which usually serves as the first attempt of reconstruction. We choose the domain as a disk with radius $R$ and suppose the boundary is negatively oriented. 
	The set of tomography $\Gamma$ contains all broken rays whose incoming part has positive projection onto the tangent of the boundary.
	We choose $f_1$ to be a Gaussian concentrated near a single point, as an approximation of a delta function and $f_2$ to be zero. 
	The support of $f = f_1+f_2$ is in this disk. 
	
	\begin{figure}[h]
		\centering
		\begin{subfigure}{0.33\textwidth}
			\centering
			\includegraphics[width=0.9\linewidth]{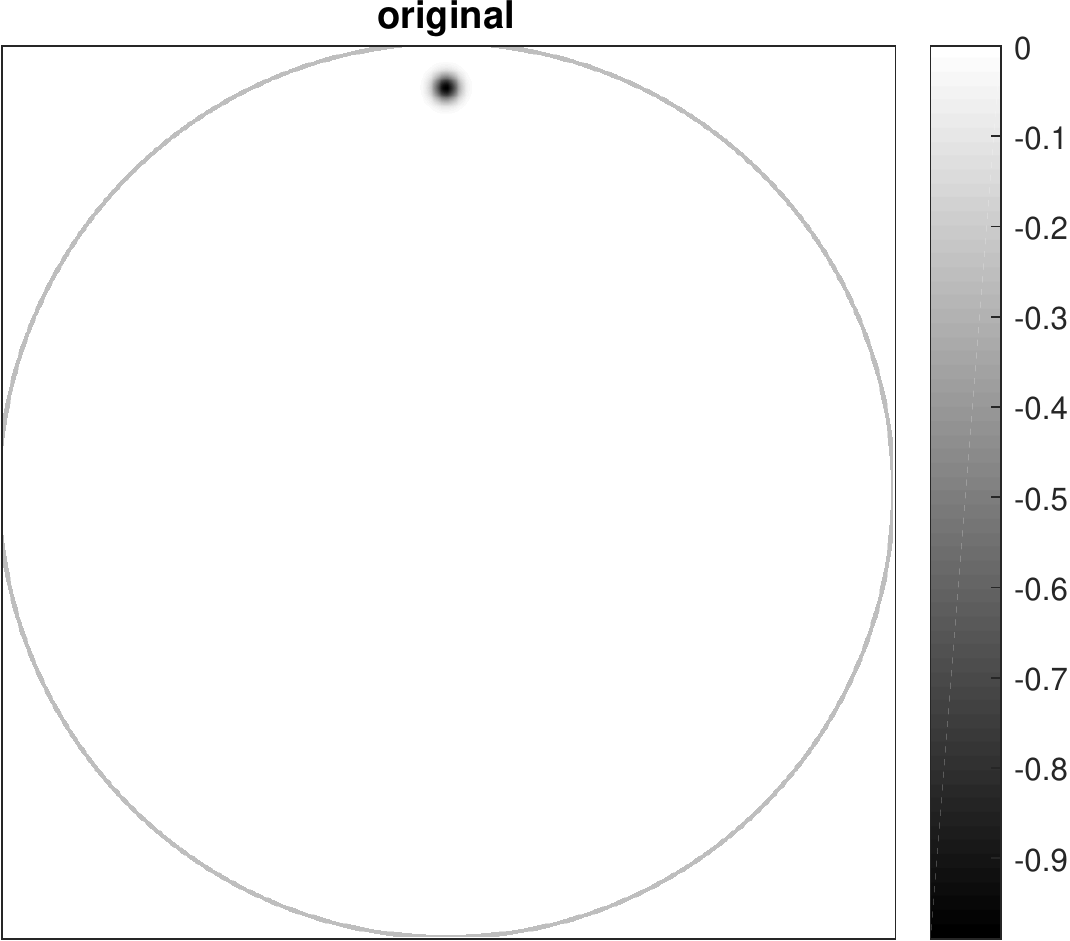}
		\end{subfigure}%
		\begin{subfigure}{0.33\textwidth}
			\centering
			\includegraphics[width=0.9\linewidth]{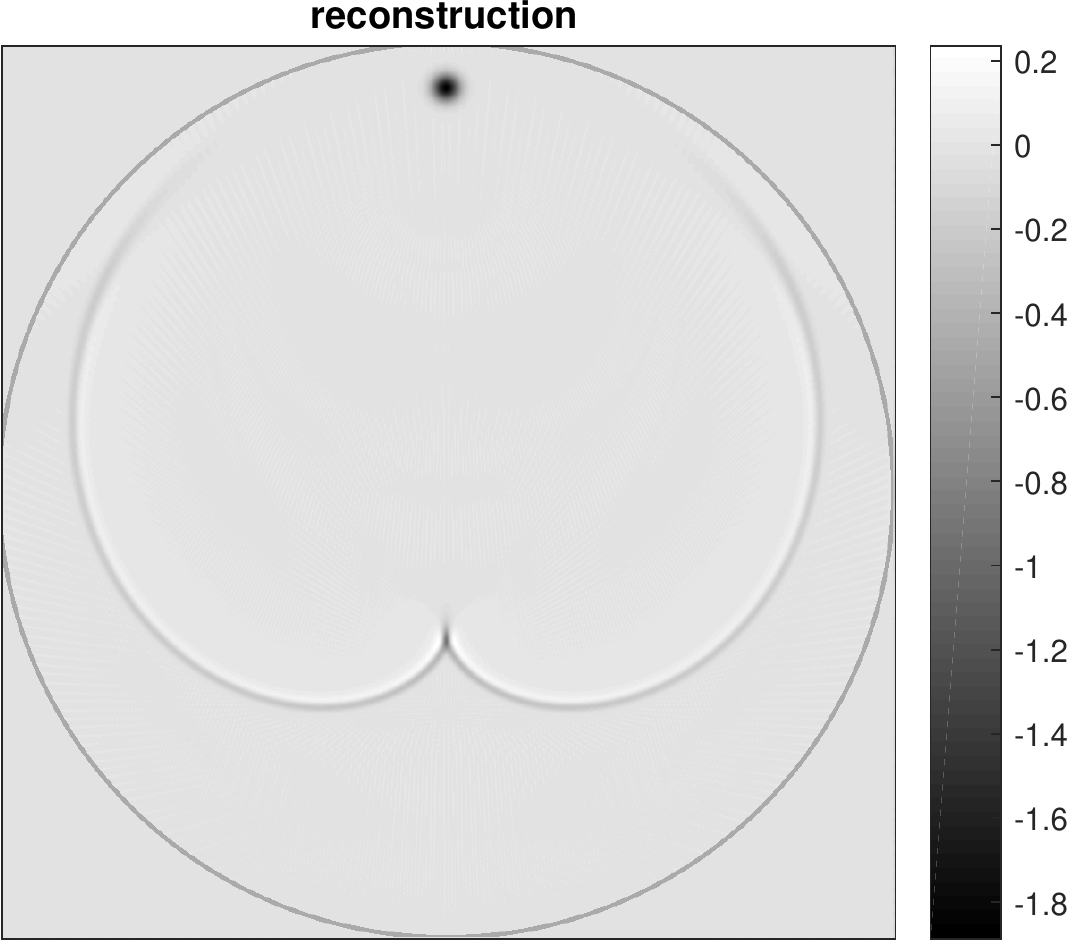}
		\end{subfigure}
		\begin{subfigure}{0.33\textwidth}
			\centering
			\includegraphics[width=0.83\linewidth,angle=90]{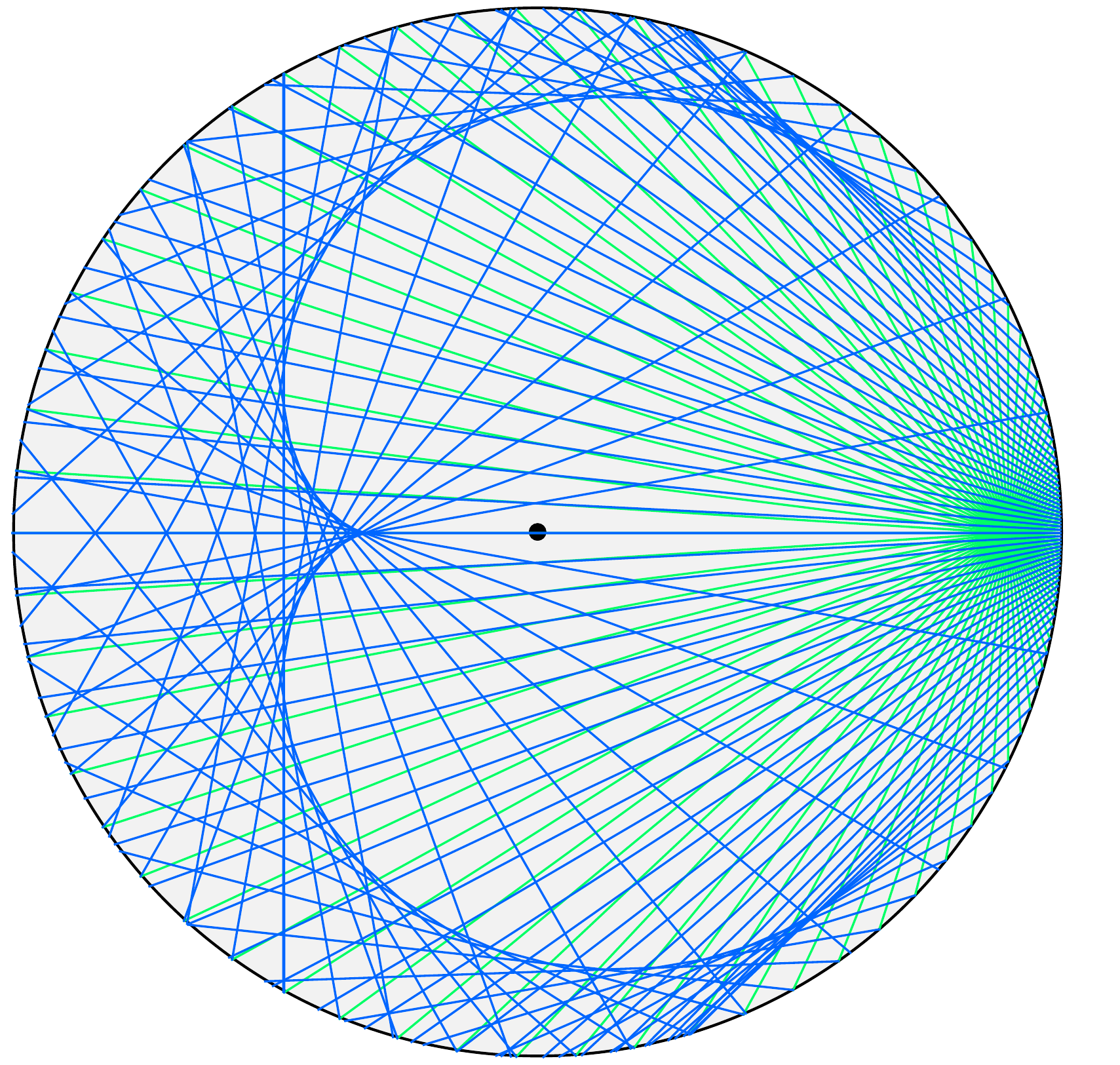}
		\end{subfigure}
		\caption{Artifacts and caustics. Form left to right: $f$, $ \mathcal{B}^*\Lambda \mathcal{B} f$, and caustics caused by reflected light.  }
		\label{causticscompare}
	\end{figure}
	
	In the code, $\mathcal{B}f$ is parameterized in the coordinate $(x_p, \alpha) \in [-R, R]\times[0,2\pi)$. Here $(x_p, \alpha)$ refers to the incoming part of a broken ray and we use it to represent the broken ray. 
	This parameterization follows the convention in Radon transform in MATLAB. 
	The radial coordinate $x_p$ is the value along the $x'$-axis, which is oriented at $\alpha$ degree counterclockwise from the $x$-axis. 
	We use the function \textit{radon} to numerically construct our operator $\mathcal{B}$ by the following formula
	$$
	\mathcal{B}f(x_p,\alpha) = Rf(x_p,\alpha)  + Rf(x_p',\alpha'),
	$$
	where $(x_p',\alpha')$ is given by the reflection. 
	Since numerically $Rf$ is known on discrete values of $(x_p,\alpha)$, we use interpolation methods to approximate $Rf(x_p',\alpha')$. Similarly, $\mathcal{B}^*$ is numerically constructed by the function \textit{iradon} and interpolation methods. 
	To better recover $f$, we apply the filter $\Lambda$ to the data before applying the adjoint operator. The plots are shown in the Figure \ref{causticscompare}. We can clearly see the artifacts appear exactly in the location of conjugate points, compared \ with the caustics caused by a light source. Furthermore, they are expained by equation (\ref{normallocal}). 
	
\end{eg}

\begin{eg}
	This example is to illustrate the reconstruction from local data by Landweber iteration. Assume for each $(x,\xi)$ in $\text{\( \WF \)}(f)$, it is visible and is perceived by only one broken ray. Then it has at most one conjugate point. 
	To make it true, we use part of the circle as the reflection boundary. The tomography family $\Gamma$ is the set of all broken rays which comes from the left side with vertices on the boundary. 
	
	By \cite{Holman2017} , we choose $f$ to be a modified Gaussian with singularities located both in certain space and in direction, that is, a coherent state, as is shown in Figure \ref{localcompare}(a). We use the Landweber iteration to reconstruct $f$. The artifacts are still there after $100$ iterations and the error becomes stable. Then we rotate $f$ or move it to see what happens to the artifacts. Specifically, in Figure \ref{localcompare}(c) and (d), $f$ remains in the location but is rotated by some angles. In (e) and (f), we move $f$ closer to the center and rotate it a bit. As the wave front set of $f$ changes, the artifacts changes and always appear in the location of their conjugate vectors.
	%\begin{figure}[h]
	%	\centering
	%	\includegraphics[width=\linewidth]{fig/lo_orig30.eps}
	%\end{figure}
	\begin{figure}[h]
		\centering
		\begin{subfigure}{0.3\textwidth}
			\centering
			\includegraphics[width=0.9\linewidth]{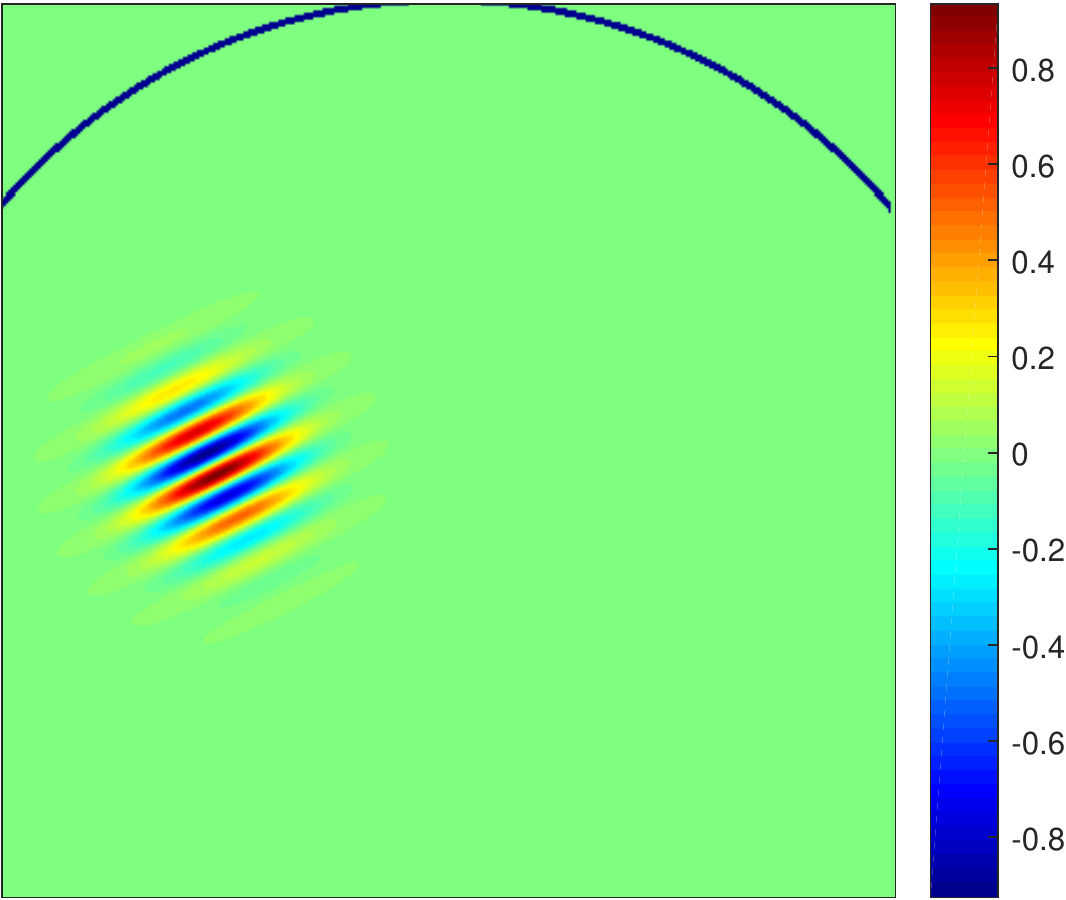}
			\subcaption{true $f$}
		\end{subfigure}%
		\begin{subfigure}{0.3\textwidth}
			\centering
			\includegraphics[width=0.9\linewidth]{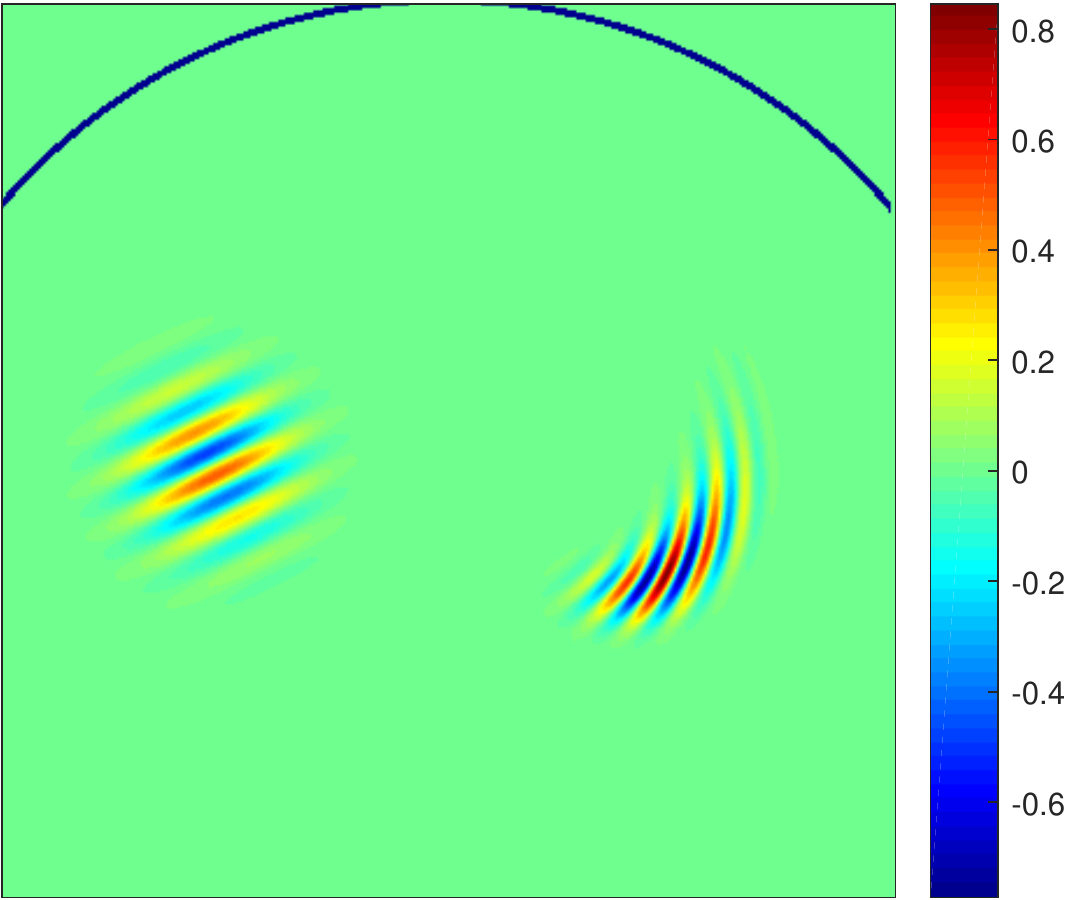}
			\subcaption{$f^{(100)}$ with $f$ rotated}
		\end{subfigure}
		\begin{subfigure}{0.3\textwidth}
			\centering
			\includegraphics[width=0.9\linewidth]{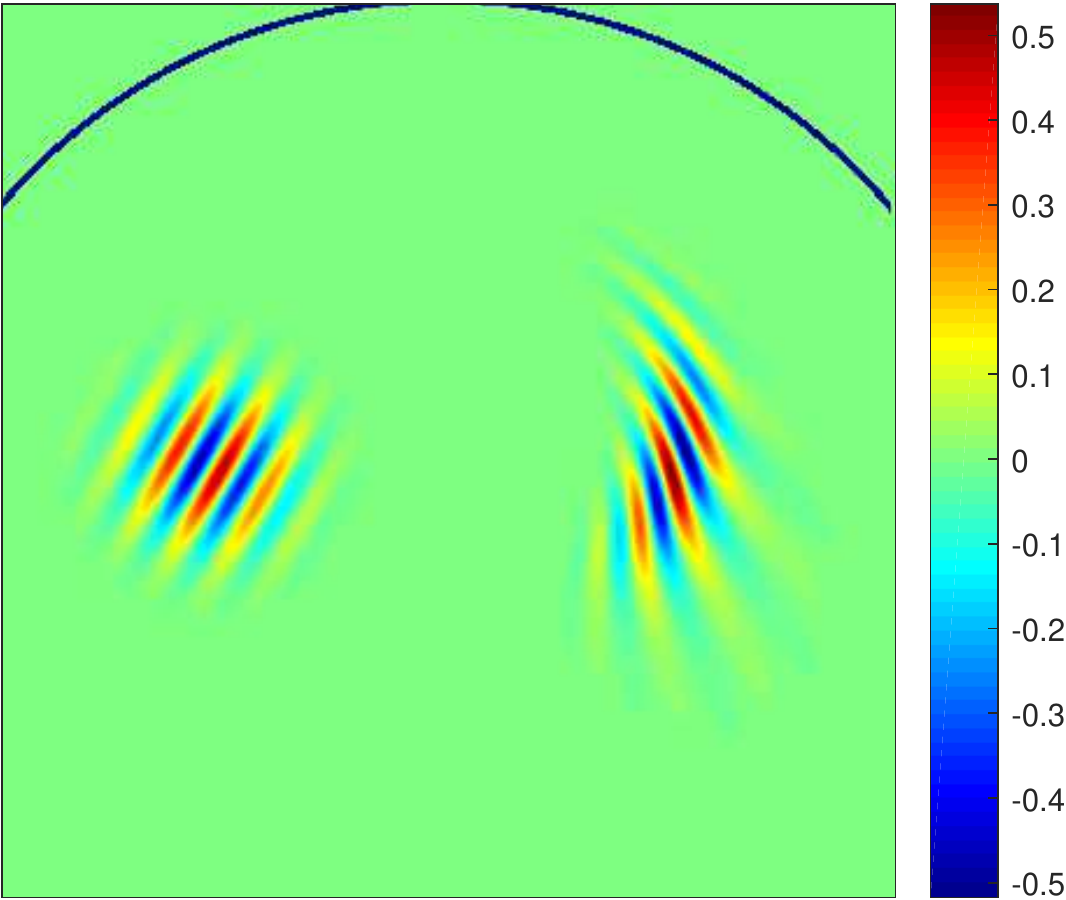}
			\subcaption{$f^{(100)}$ with $f$ rotated}
		\end{subfigure}
		\medskip
		\centering
		\begin{subfigure}{0.3\textwidth}
			\centering
			\includegraphics[width=0.9\linewidth]{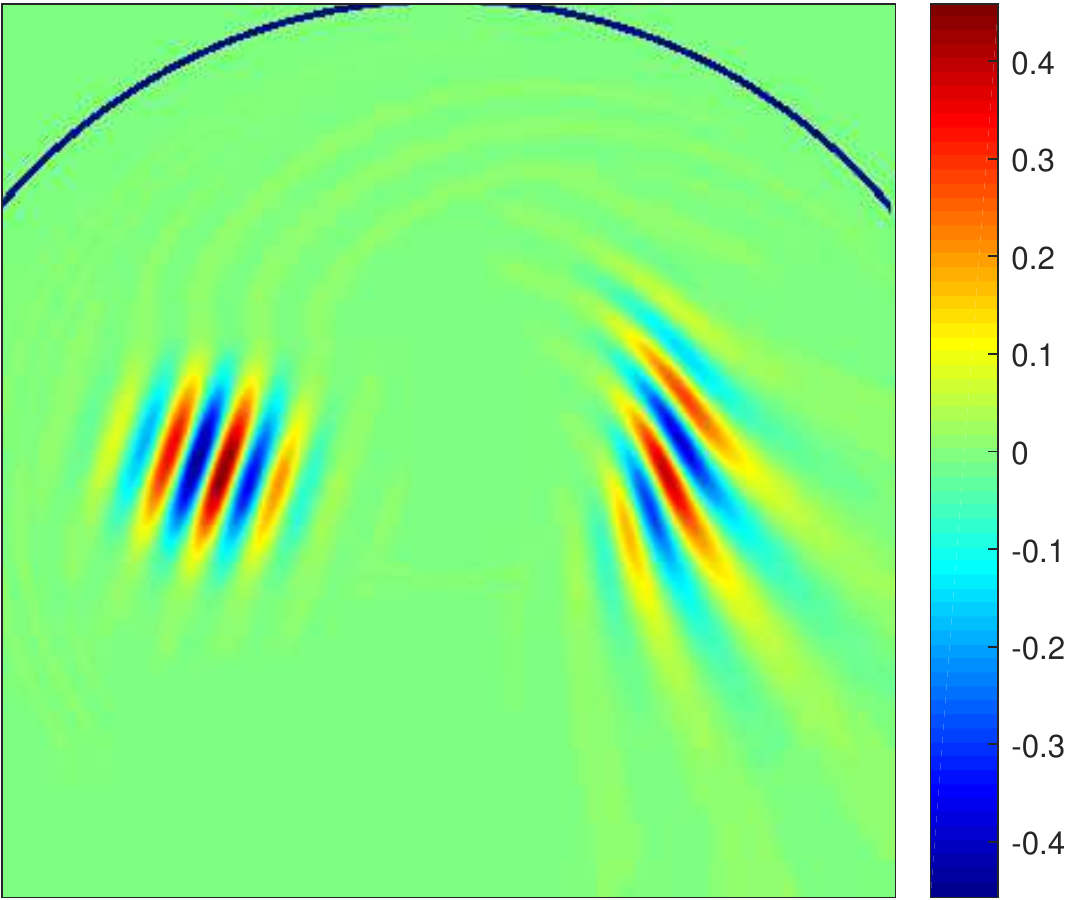}
			\subcaption{$f^{(100)}$ with $f$ rotated}
		\end{subfigure}
		\begin{subfigure}{0.3\textwidth}
			\centering
			\includegraphics[width=0.9\linewidth]{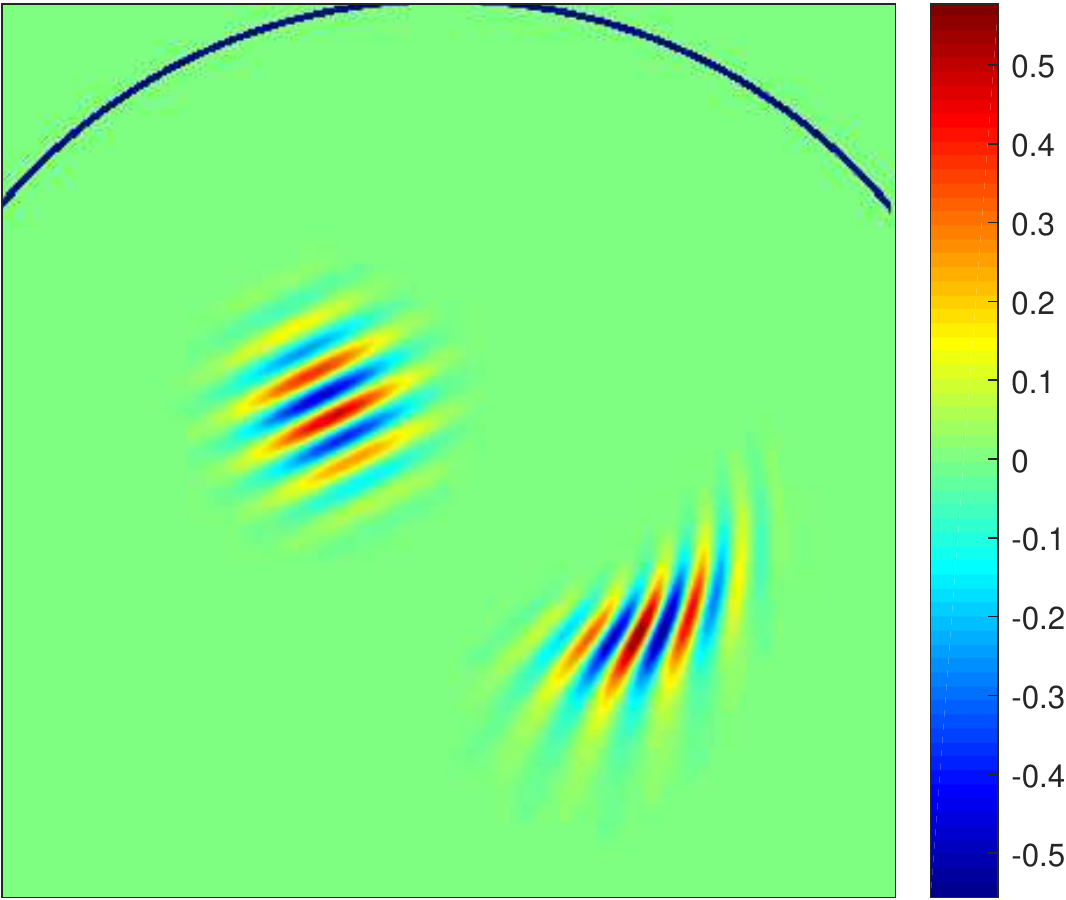}
			\subcaption{$f^{(100)}$ with $f$ moved}
		\end{subfigure}%
		\begin{subfigure}{0.3\textwidth}
			\centering
			\includegraphics[width=0.9\linewidth]{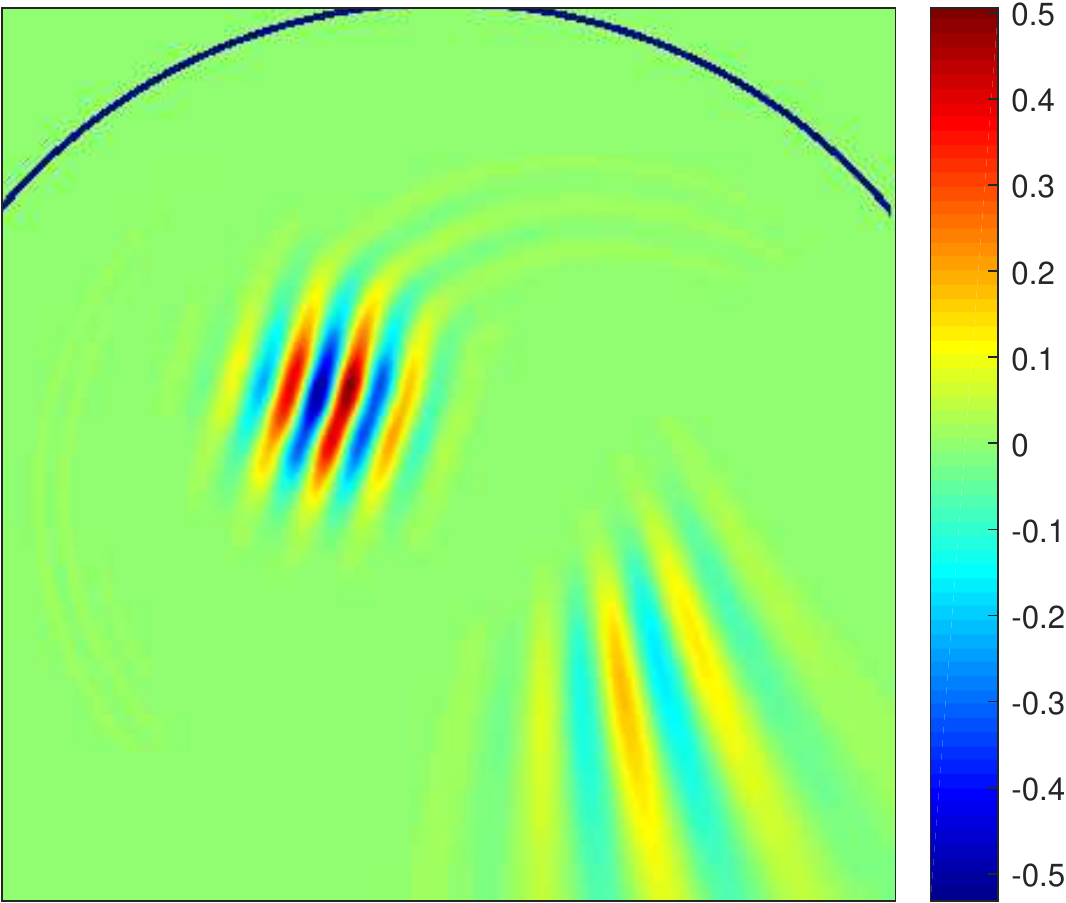}
			\subcaption{$f^{(100)}$ with $f$ rotated}
		\end{subfigure}
		\caption{Local reconstruction by Landweber iteration.}
		\label{localcompare}
	\end{figure}
	
\end{eg}
\subsection{Global Problems}
In this section, we consider the artifacts when we use full data for reconstruction.
Suppose $(x_0,\xi^0)$ in $\text{\(\WF\)}(f)$. There are two broken rays in $\Gamma$ that carry this singularity.
%as is shown in Figure \ref{pic1}. 
One broken ray $\nu_0$ represented by $(s_0,\alpha_0)$ has it in the incoming part, and the other one $\nu_{-1}$ represented by $(s_{-1},\alpha_{-1})$ has it in the reflected part. Suppose $(x_1,\xi^1)$ and $(x_{-1},\xi^{-1})$ are its conjugate covectors along $\nu_0$ and $\nu_{-1}$, if they exist. We have the following cases.

If at least one of $(x_1,\xi^1)$ and $(x_{-1},\xi^{-1})$ does not exist, for example $(x_1,\xi^1)$, then the singularity caused by $(x_0,\xi^0)$ in $V^0$ cannot be canceled via $\nu_0$. With the assumption that $\mathcal{B}(f)$ is smooth, this indicates $(x_0,\xi^0) \in \text{\(\WF \)}(f)$ impossible.

If both $(x_1,\xi^1)$ and $(x_{-1},\xi^{-1})$ exist, then the singularities might be canceled by them. We continue to consider $\nu_1$, $\nu_{-2}$ and so on. Then we get a sequence of broken rays and conjugate covectors. 
We define the set of all conjugate covectors related to $(x_0,\xi^0)$ in the following
$$
\mathcal{M}(x_0,\xi^0) = \{ (x_i,\xi^i), \text{if it exists and conjugate to $(x_{i-1},\xi^{i-1})$, for } i = 0, \pm 1, \pm 2, \ldots \}.
$$
If $\mathcal{M}(x_0,\xi^0)$ contains finitely many $(x_i,\xi^i)$ whose index $i$ is positive(or negative), we say it is \textit{incomplete} in positive(or negative) direction. Otherwise, we say it is \textit{complete}.

\begin{eg}\label{egdisk2}
	As is shown in the Figure \ref{globalbrt}, we use the same domain and family of tomography as in Example \ref{egdisk0}. Especially, we suppose the disk is centered at the origin for simplification.	
	\begin{figure}[h]
		\includegraphics[height=0.3\textwidth]{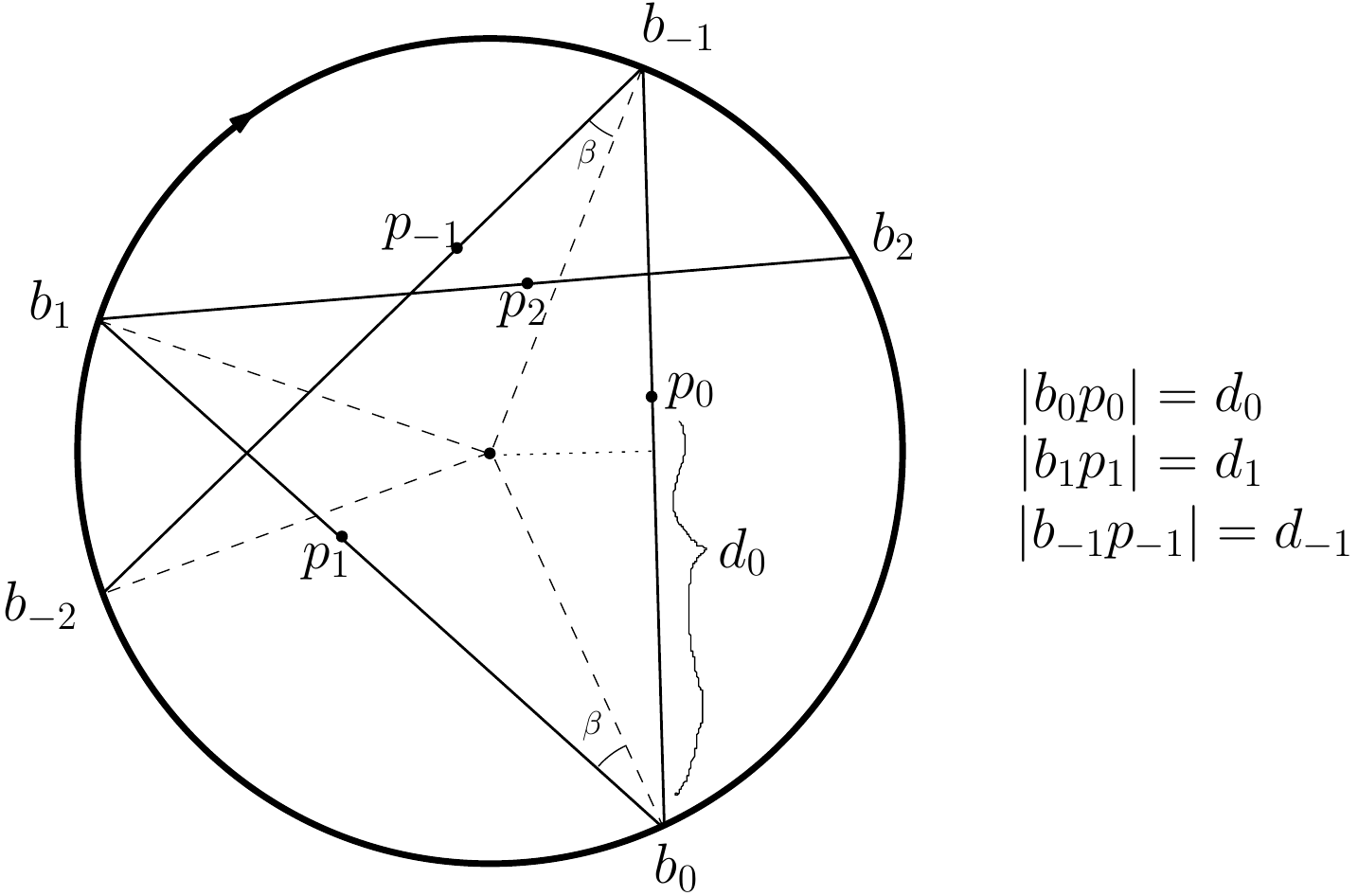}
		\caption{Inside a circular mirror, a sequence of broken rays and conjugate points on them.}
		\label{globalbrt}
	\end{figure}	
	Considering a point $(p_0,\xi^0)$, we have a sequence of broken rays 
	$$\ldots, b_{-2} b_{-1} b_0,\ b_{-1} b_{0} b_1,\ b_{0} b_{1} b_2,\ldots\ b_{i-1} b_{i} b_{i+1},\ldots ,$$
	as well as the set  $\mathcal{M}(p_0,\xi^0)$.
	\begin{pp}
		We say $(p_0,\xi^0)$ is radial if $p_0$ is the midpoint of a chord such that $\xi_0$ is in its conormal. Then $\mathcal{M}(p_0,\xi^0)$ is complete if and only if $(p_0,\xi^0)$ is radial.
	\end{pp}
	\begin{proof}
		Fix a point $p_i$. It might have a conjugate point $p_{i+1}$ along $b_{i-1}b_{i}b_{i+1}$ or $p_{i-1}$ along $b_{i-2}b_{i-1}b_{i}$. Let $d_i = |b_ip_i|$ be the distance along the ray from $p_i$ to the boundary point $b_i$. Notice all incidence and reflection angles are equal (call them $\beta$). Then $|b_ib_{i+1}| = 2\cos\beta$ for all index $i$.
		
		Recall Corollary \ref{conjugatecor}. In this case, we have $\Delta t_1=d_1$, $\Delta t_2=2\cos \beta -d_2$, and
		$\frac{d \alpha_2}{d \alpha_1} = \frac{2 d_1}{\cos \beta} - 1$.
		Then $p_i$ has a conjugate point $p_{i+1}$ inside the domain if and only if $d_{i+1}$ given by 
		$$\frac{1}{d_i} + \frac{1}{2 \cos \beta - d_{i+1}} = \frac{2}{\cos \beta}$$
		has a solution in $(0,2\cos\beta)$.
		To simplify, we change the variable that $d_i = \cos \beta(a_i + 1)$. Thus,
		\begin{equation}\label{aiter}
		\frac{1}{1+a_i} + \frac{1}{1 - a_{i+1}} = 2 \implies 2a_ia_{i+1} +a_{i+1} - a_i = 0 .
		\end{equation}
		% \implies \frac{1}{a_{i+1}} = \frac{1}{a_i} + 2.$$
		%This is also true for all pairs $(d_{i-1},d_{i})$ if they exist. 
		The requirement that $p_{i}$  is inside the domain means we are finding solutions for $a_i \in (-1,1)$.
		
		\textit{case 1.} $a_0 = 0$, which is followed by $a_i = 0$ for any integer $i$. This is the case when
		we have $p_0$ at the midpoint of some chord and $\xi^0$ is the conormal of the chord. The same is true with all $(p_i,\xi^i)$. We have a complete $\mathcal{M}(p_0,\xi^0)$.
		
		\textit{case 2.} $a_i \neq 0$. Then (\ref{aiter}) can be reduced to the following iteration formula 
		$$
		\frac{1}{a_{i+1}} = \frac{1}{a_{i}} + 2
		$$
		Suppose we start from some $a_0$. Each time, the next $\frac{1}{a_{i}}$ increases or decreases by $2$.
		With $\frac{1}{a_0} \in (-\infty,-1) \bigcup (1,\infty)$, finally we must have some $\frac{1}{a_i}$ belonging to the interval $(-1,1)$, which mean $p_i$ goes out of the domain. In this case, $\mathcal{M}(p_0,\xi^0)$ is always incomplete.
	\end{proof}	
\end{eg}

Next, let $V^i$ be a small conic neighborhoods of a fixed $(x_i,\xi^i) \in \mathcal{M}(x_0,\xi^0)$ and $U_i = \pi(V^i)$. Let $f_i$ be the restriction of $f$ on $U_i$. By shrinking $V^i$ carefully, we have $C(V^i)= V^{i-1}$, if both $(x_{i-1},\xi^{i-1})$ and $(x_i,\xi^i)$ exist. Then the cancellation of singularities shows,
$$
R_{i-1}f_{i-1} + R_i f_i = 0 \mod C^\infty.
$$
For $\mathcal{M}(x_0,\xi^0)$ that is finite in positive direction, finally we have
$$
%\begin{cases}
R_{i_0}f_{i_0} = 0 \mod C^\infty.\\
%Rf_{i-1} + \chi^*R f_i = 0 \mod C^\infty\\
%Rf_{i-2} + \chi^*R f_{i-1} = 0 \mod C^\infty\\
%\ldots\\
%\end{cases}
$$
By applying the diffeomorphism $\chi^*$ and forward substitution, we can show all $f_i$ must be smooth. It is similar if  $\mathcal{M}(x_0,\xi^0)$ is incomplete in negative direction. 
This proves when $\mathcal{M}(x_0,\xi^0)$ is incomplete, $(x_0,\xi^0)$ is a recoverable singularity. 

\begin{corollary}
	Suppose everything as in Example \ref{egdisk2}. Then $(x_0,\xi^0)$ is recoverable if $(x_0,\xi^0)$ is not radial.
\end{corollary}

\begin{eg}
	With the same set up as above, we first choose $f_1$ to be a modified Gaussian of coherent state whose singularities are not radial. To compare, then we choose $f_2$ to be  with radial singularities. 
	\begin{figure}[h]
		\centering
		\begin{subfigure}{0.3\textwidth}
			\centering
			\includegraphics[width=0.9\linewidth]{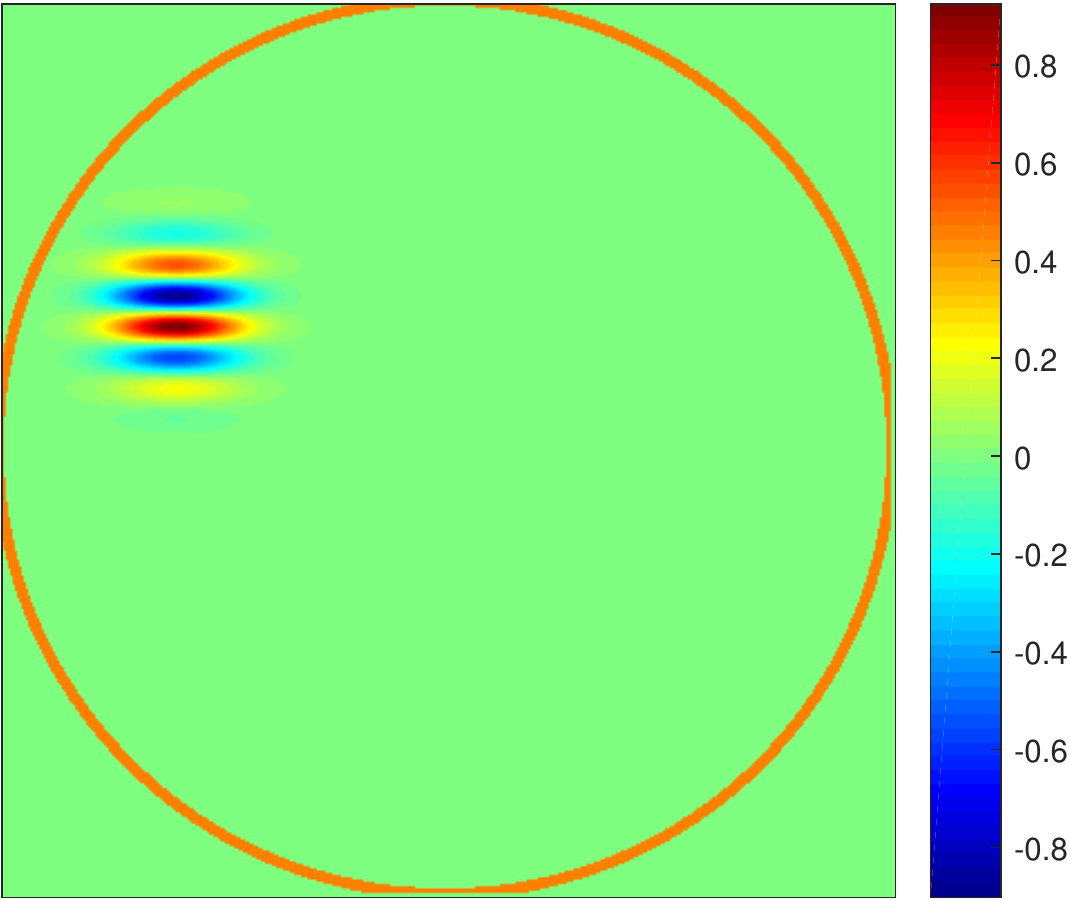}
			\subcaption{non-radial singularities $f_1$}
		\end{subfigure}%
		\begin{subfigure}{0.3\textwidth}
			\centering
			\includegraphics[width=0.9\linewidth]{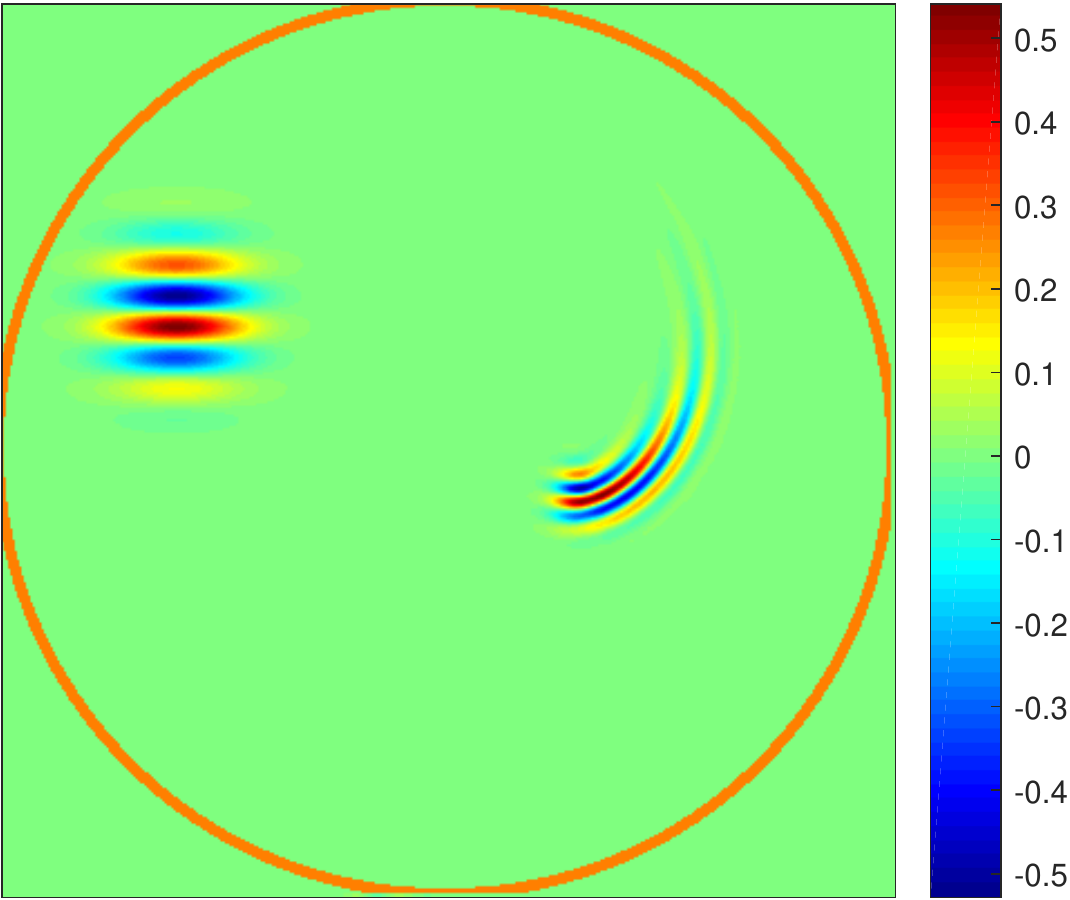}
			\subcaption{backprojection $f_1^{(1)}$}
		\end{subfigure}
		\begin{subfigure}{0.3\textwidth}
			\centering
			\includegraphics[width=0.9\linewidth]{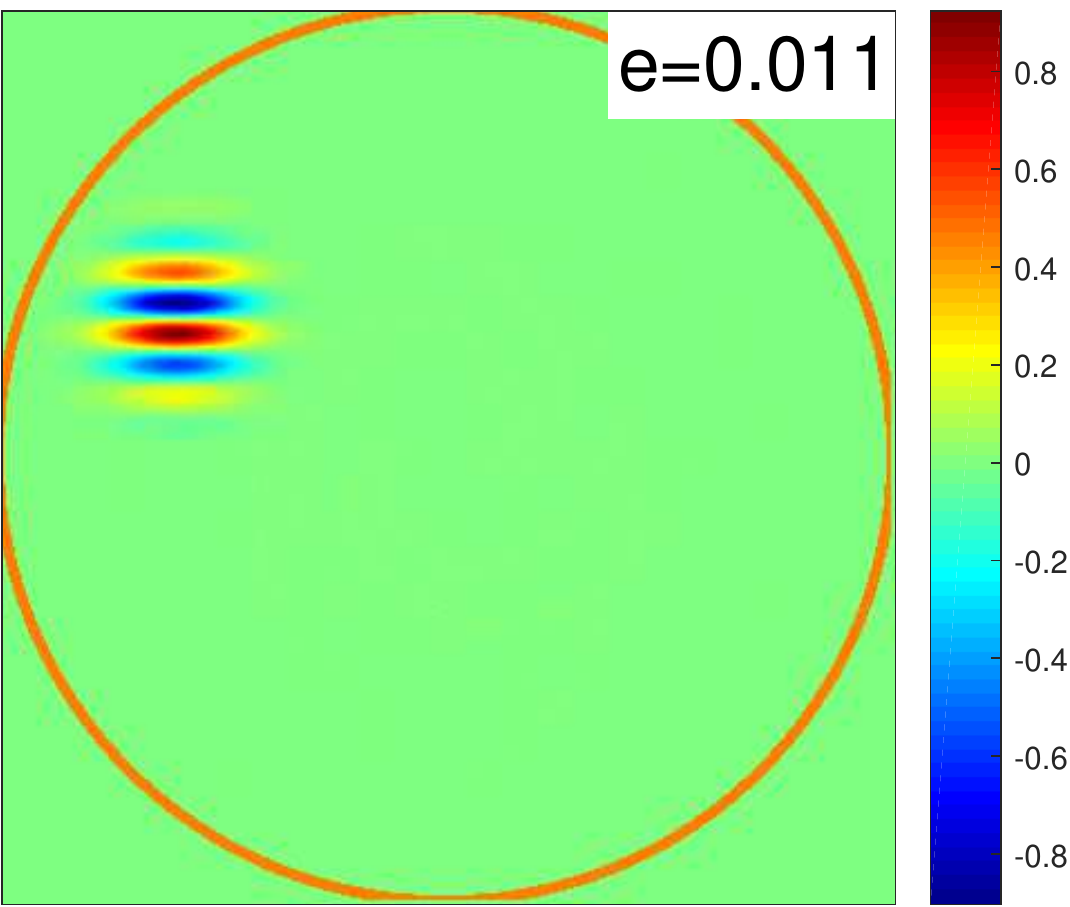}
			\subcaption{$f_1^{(100)}$}
		\end{subfigure}
		
		\centering
		\begin{subfigure}{0.3\textwidth}
			\centering
			\includegraphics[width=0.9\linewidth]{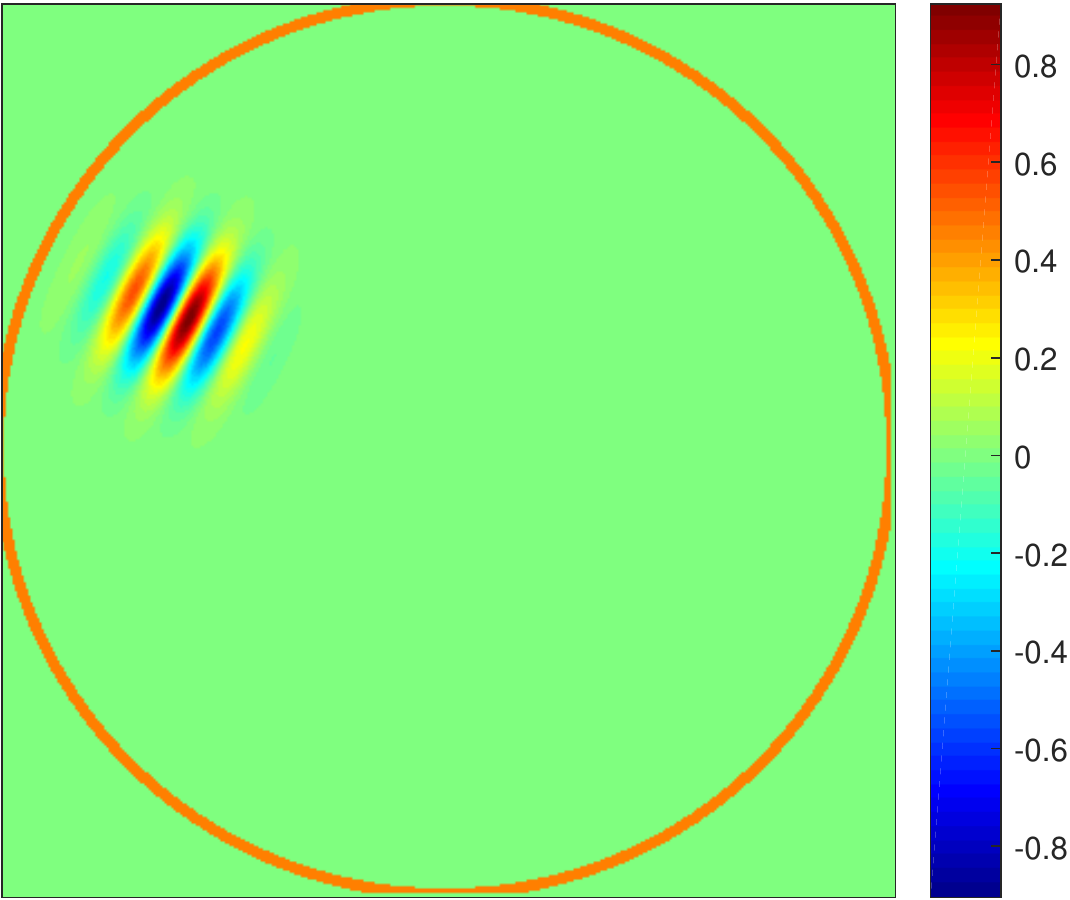}
			\subcaption{radial singularities $f_2$}		
		\end{subfigure}%
		\begin{subfigure}{0.3\textwidth}
			\centering
			\includegraphics[width=0.9\linewidth]{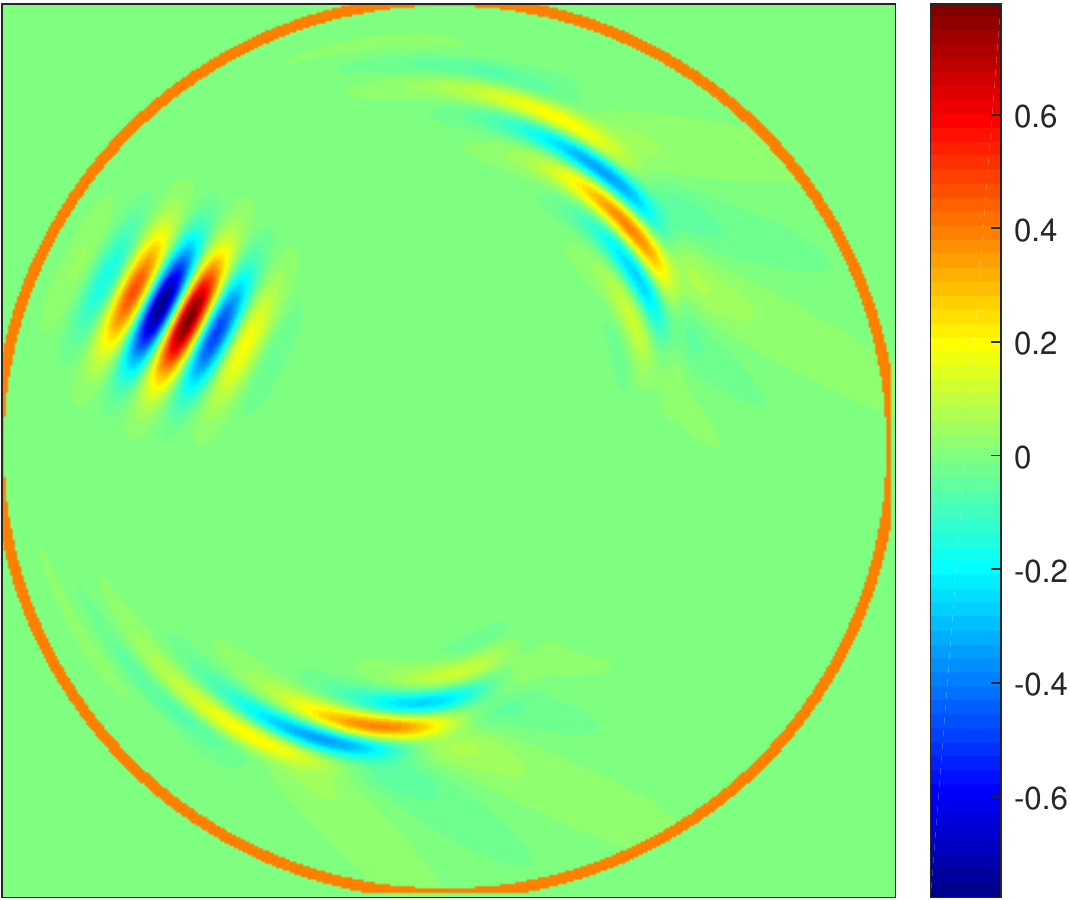}
			\subcaption{backprojection $f_2^{(1)}$}
		\end{subfigure}
		\begin{subfigure}{0.3\textwidth}
			\centering
			\includegraphics[width=0.9\linewidth]{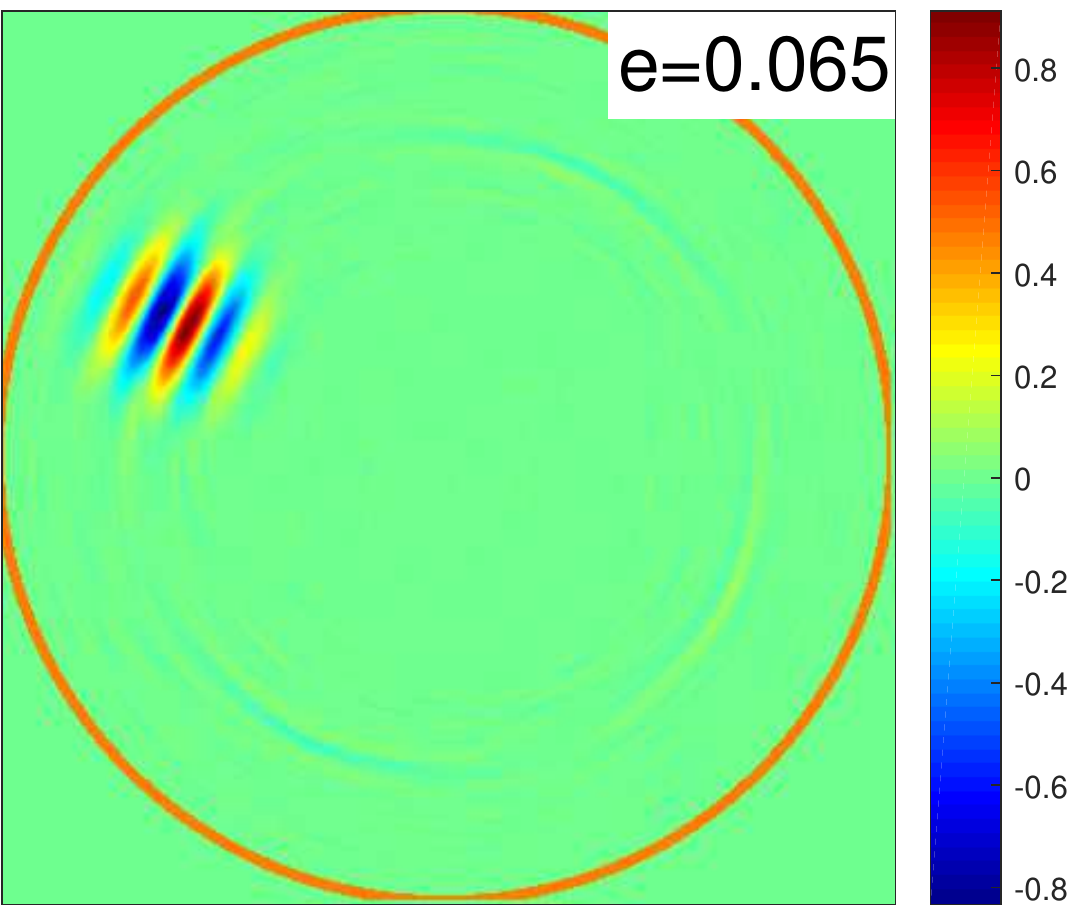}
			\subcaption{$f_2^{(100)}$}
		\end{subfigure}
		\caption{Reconstruction of $f_1$ and $f_2$ from global data, where $e = \frac{\|f - f^{(100)}\|_2}{\|f\|_2}$ is the relative error. }
		\label{globalfigure}
	\end{figure}
	
	As is shown in Figure \ref{globalfigure}, after performing Landweber iteration of $100$ steps, all artifacts fade out and the reconstruction has a small error if $f$ has non-radial singularities. On the contrary, if $f$ has radial singularities, the error still decreases as the iteration but in a much slower speed. In these two cases, since $f$ is only supported in a small set, the artifacts arising in the reconstruction may seem not so obvious. However, when $f$ is more complicated, the artifacts might be unignorable. In the following we choose $f_3$ to be a Modified Shepp-Logan phantom. 
	\begin{figure}[h]
		\centering
		\begin{subfigure}{0.3\textwidth}
			\centering
			\includegraphics[width=0.9\linewidth]{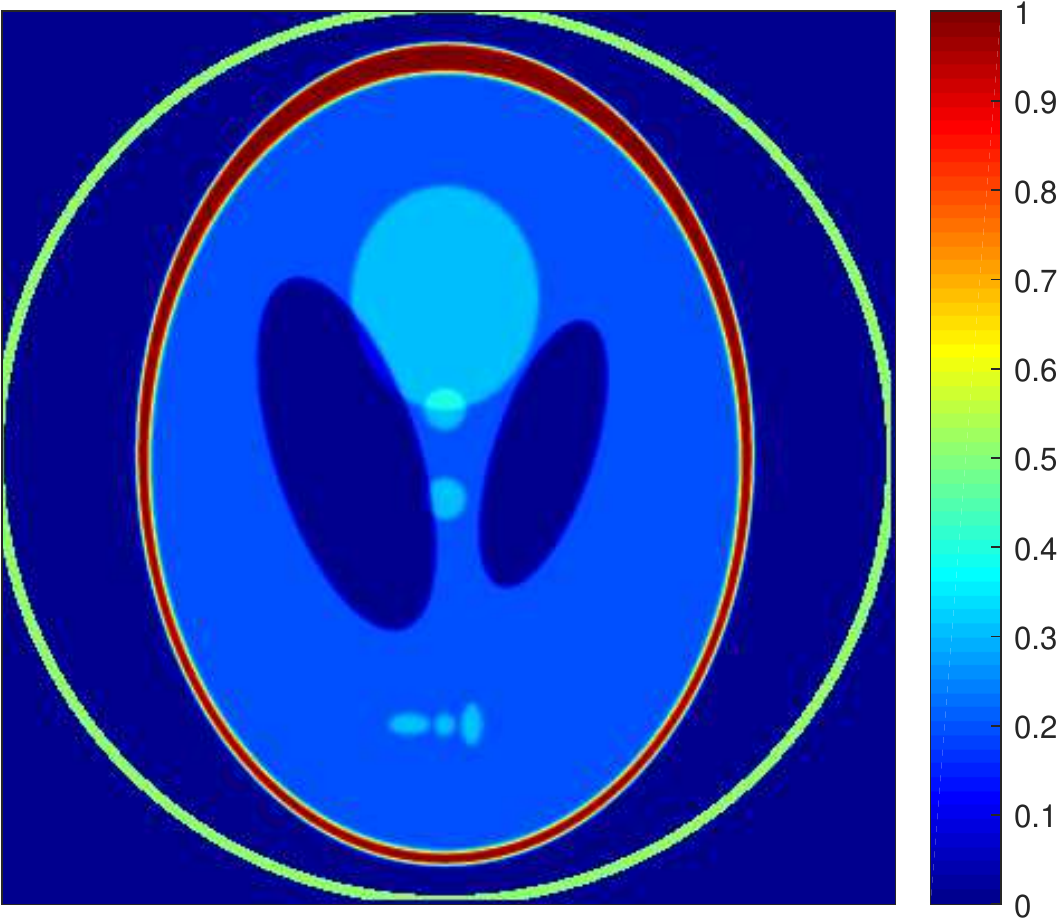}
			\caption{true $f_3$}
		\end{subfigure}%
		\begin{subfigure}{0.3\textwidth}
			\centering
			\includegraphics[width=0.9\linewidth]{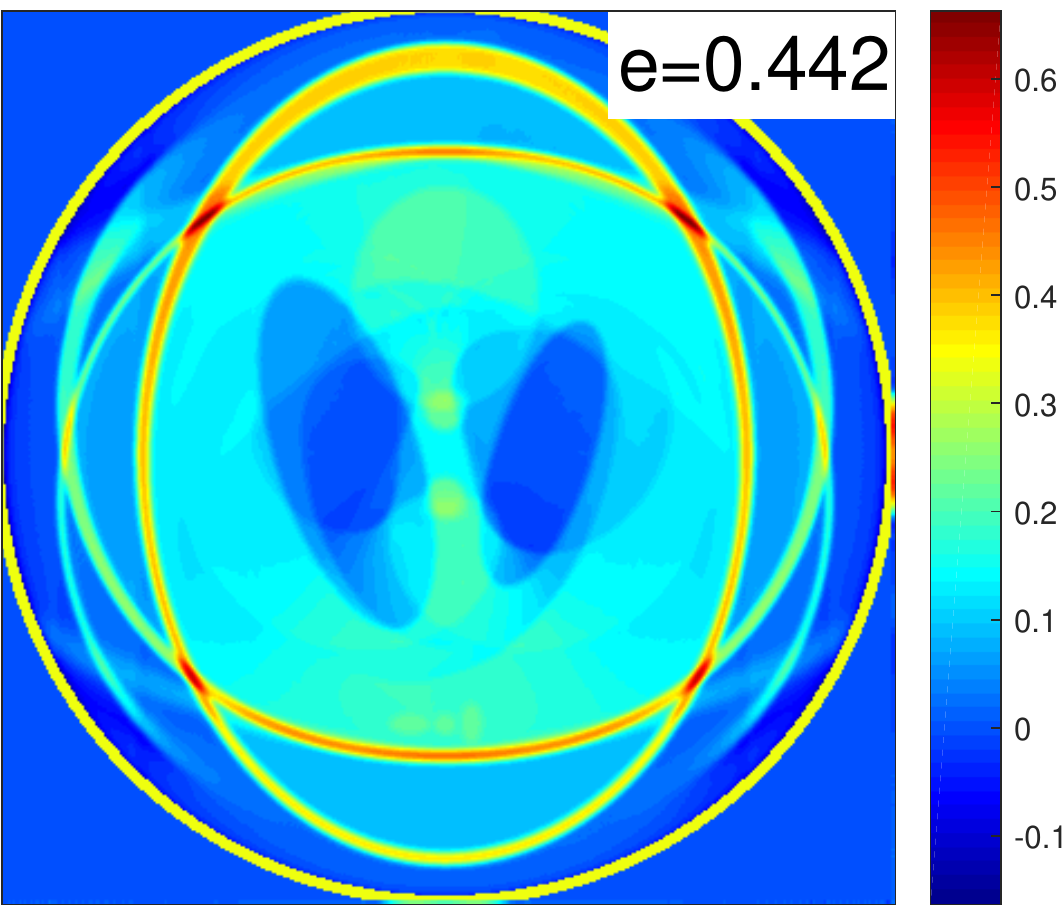}
			\subcaption{backprojection $f_3^{(1)}$}
		\end{subfigure}
		\begin{subfigure}{0.3\textwidth}
			\centering
			\includegraphics[width=0.9\linewidth]{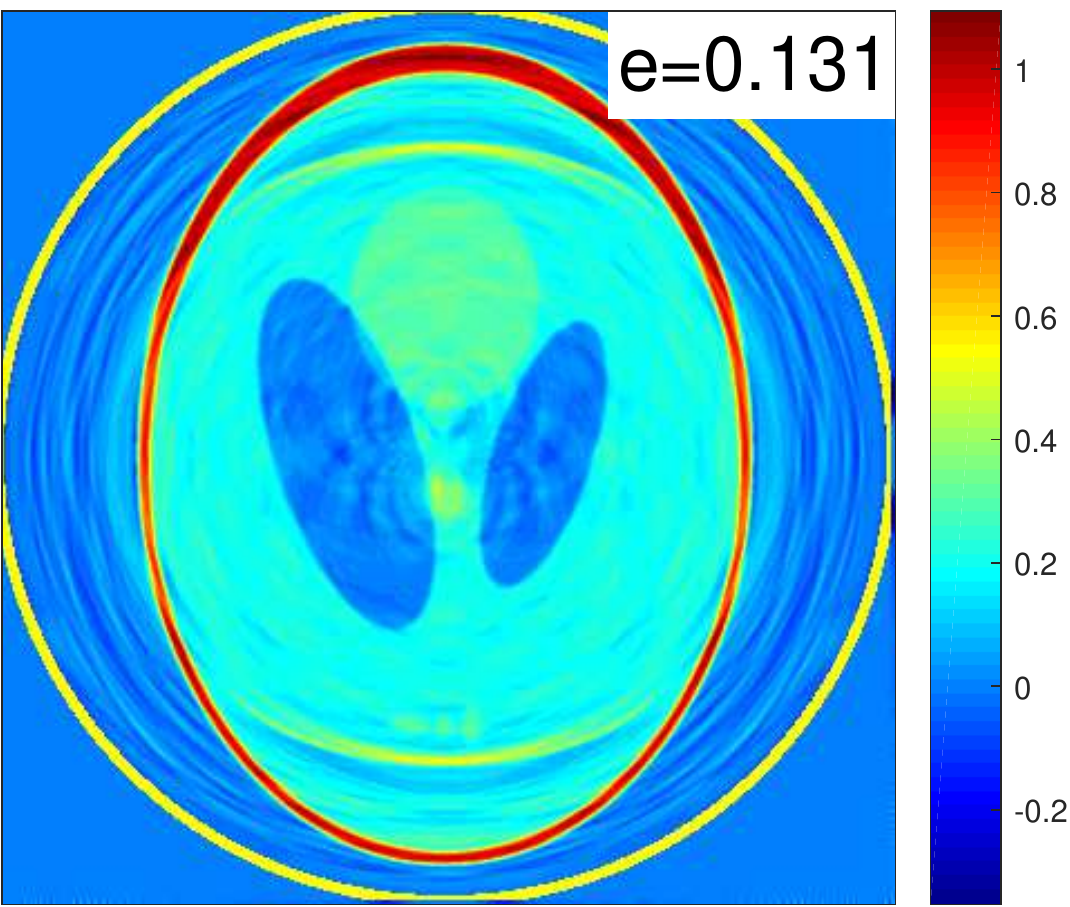}
			\subcaption{$f_3^{(100)}$}
		\end{subfigure}
		\caption{Reconstruction from global data for Modified Shepp-Logan phantom $f_3$, where $e = \frac{\|f - f^{(100)}\|_2}{\|f\|_2}$ is the relative error.}
		\label{test}
	\end{figure}
	
	The error plots of these three cases are in Figure \ref{globalfigerror} to better illustrate the difference between radial and non-radial singularities. They also show where the artifacts appear( for more details, see \ref{subconnection}). It is clear to see the error of reconstruction is much smaller when we have non-radial singularities than radial ones. 
	%It might not be true if we move the coherent state closer to the center. One reason is that when we have more conjugate covectors, the strength of each 
	\begin{figure}[h]
		\centering
		\begin{subfigure}{0.3\textwidth}
			\centering
			\includegraphics[width=0.9\linewidth]{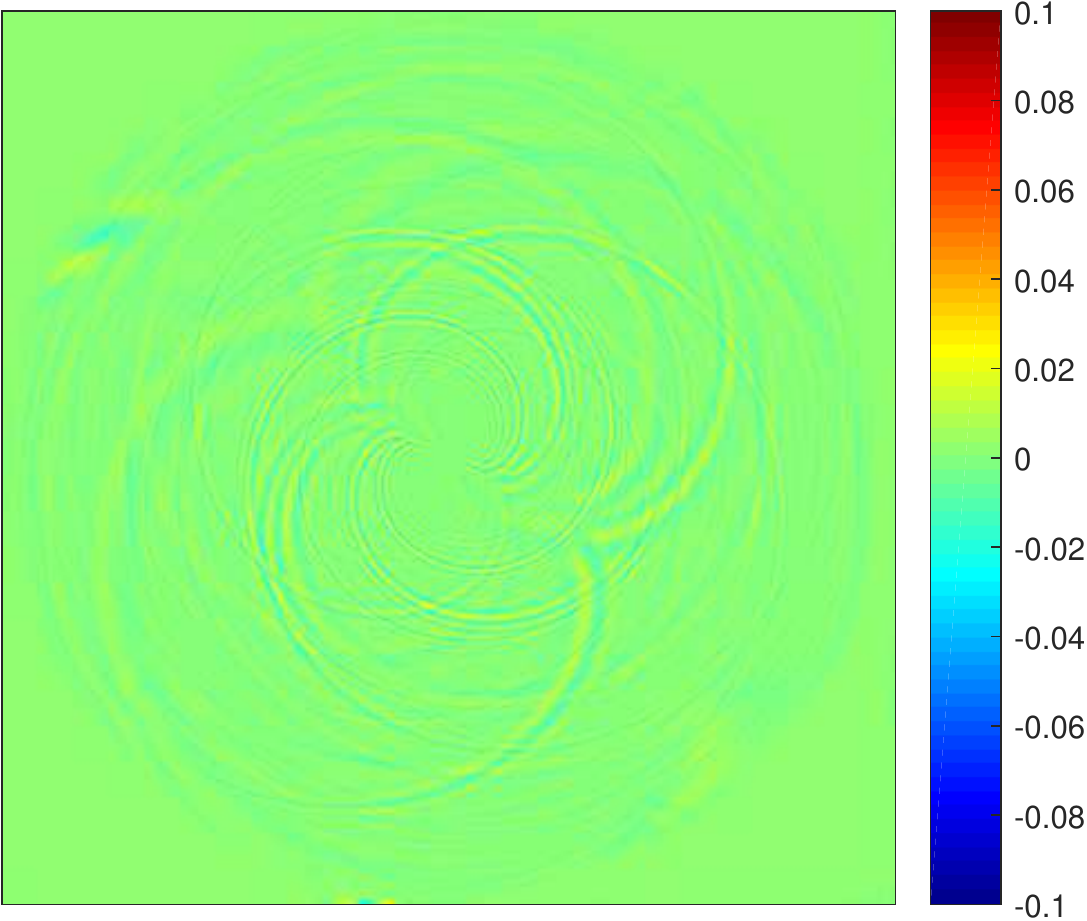}
		\end{subfigure}%
		\begin{subfigure}{0.3\textwidth}
			\centering
			\includegraphics[width=0.9\linewidth]{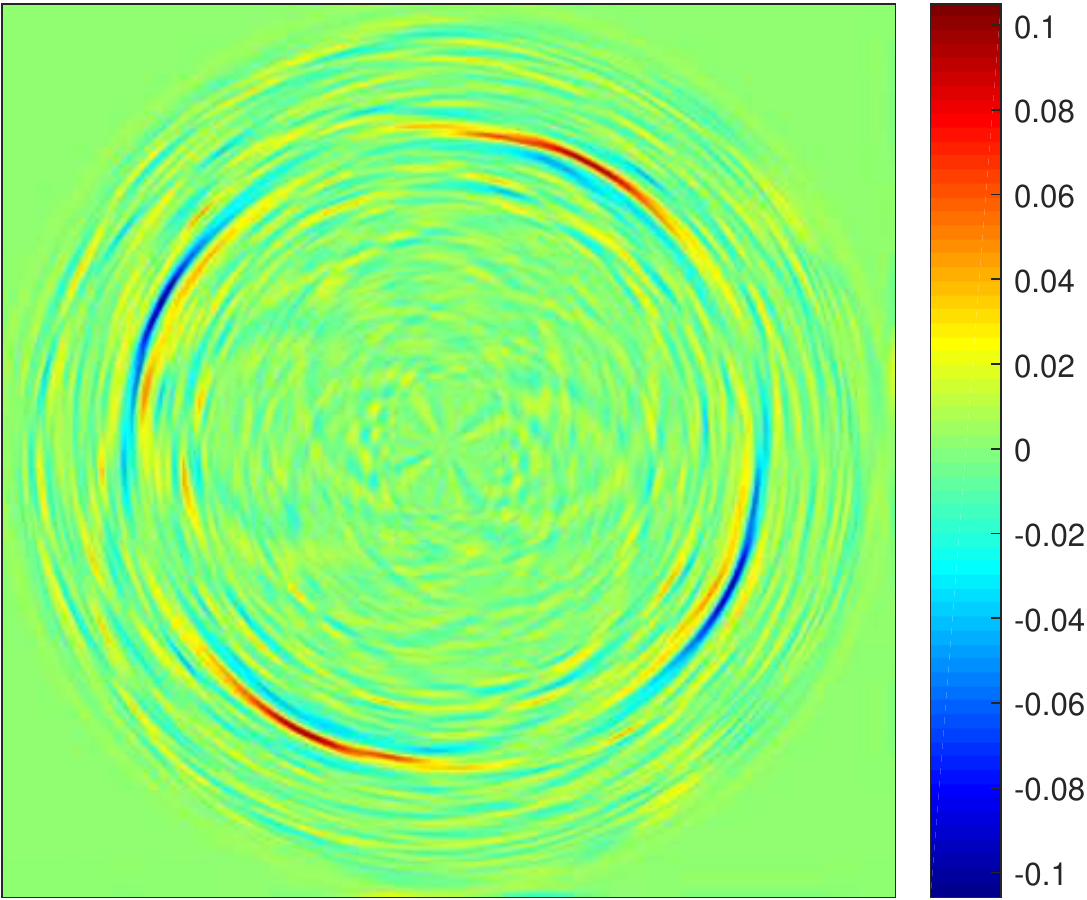}
		\end{subfigure}
		\begin{subfigure}{0.3\textwidth}
			\centering
			\includegraphics[width=0.9\linewidth]{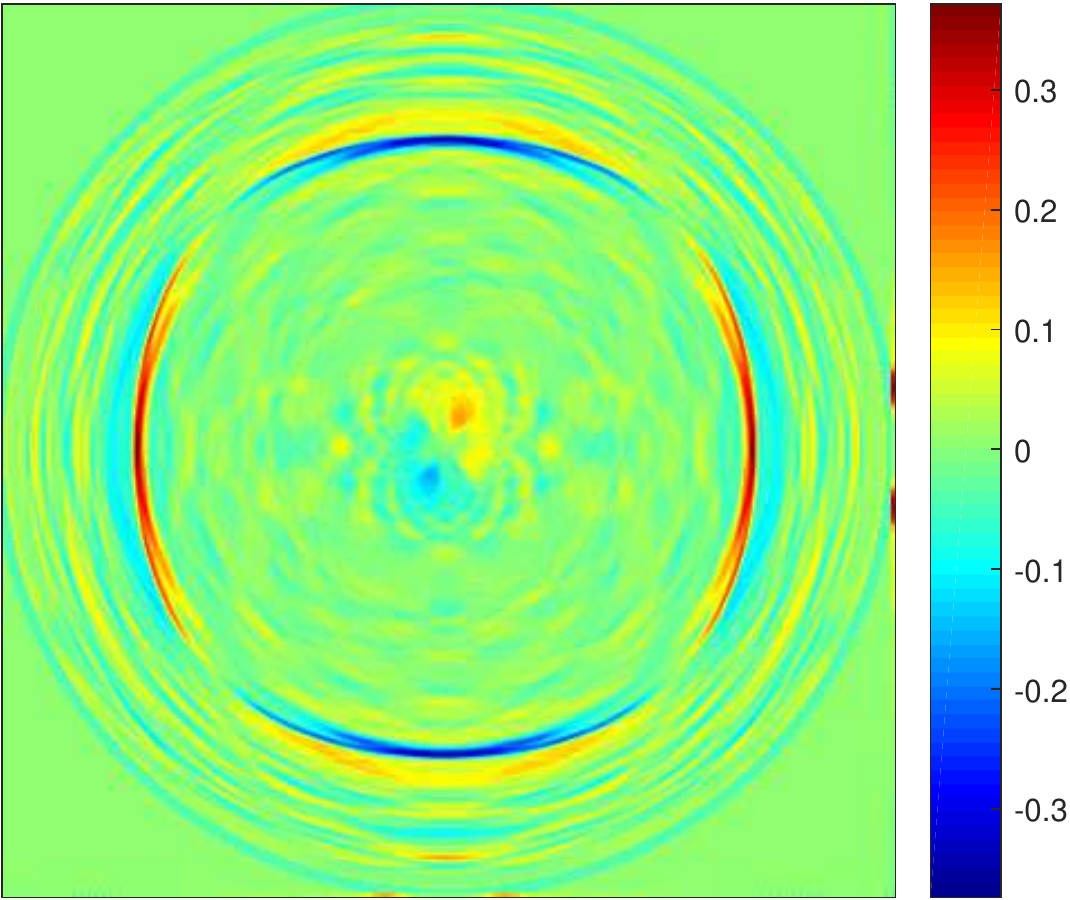}
		\end{subfigure} 	
		\caption{The error plot for the reconstruction of $f_1,f_2,f_3$ in order. The first two has the same range of color bar.}
		\label{globalfigerror}
	\end{figure}
\end{eg}
\begin{eg}
	In this example we consider the reconstruction of the V-line Radon transform in an elliptical domain $Q$ from global data. 
	By \cite{SUNWOO}, the billiard trajectory in an elliptical table has the following cases. If the trajectory crosses one of the focal points, then it converges to the major axis of $Q$. 	If the trajectory crosses the line segment between the two focal points, then it is tangent to a unique hyperbola, which is determined by the trajectory and shares the same focal points with $Q$.
	If it does not cross the line segment between the two focal points, then it is tangent
	to a unique ellipse, which shares the same focal points with $Q$.
	
	In the following numerical experiments, we choose $f$ as a coherent state. It is located and rotated such that the trajectory carrying its singularities falls into the last two cases above. We use Landweber iterations to reconstruct $f$ by iterating 100 steps. As in Figure \ref{ellipse}, in the reconstruction of the first coherent state, the artifacts disappear as we iterate, since some conjugate points are outside the domain. On the contrary, with conjugate points staying in the domain at least for the first reflection, there is a relative larger error in the reconstruction of the second one. 
	A more complete analysis of the ellipse case is behind the scope of this work.
	\begin{figure}[h]
		\centering
		\begin{subfigure}{0.25\textwidth}
			\centering
			\includegraphics[width =0.9\textwidth]{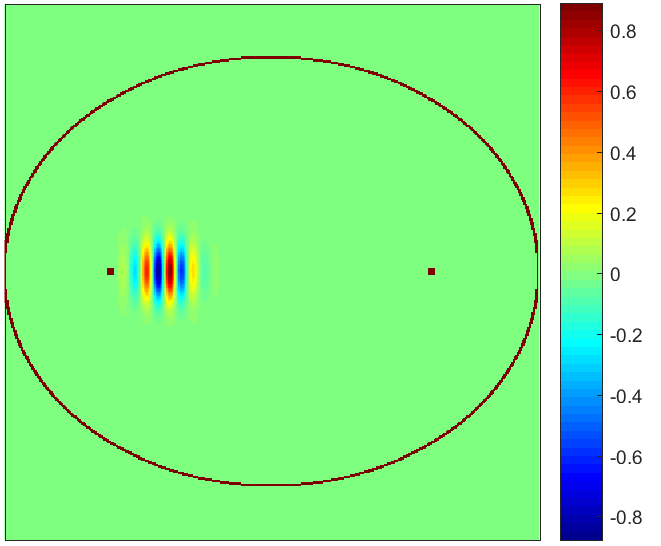}
		\end{subfigure}%				
		\begin{subfigure}{0.25\textwidth}
			\centering
			\includegraphics[width =0.765\textwidth]{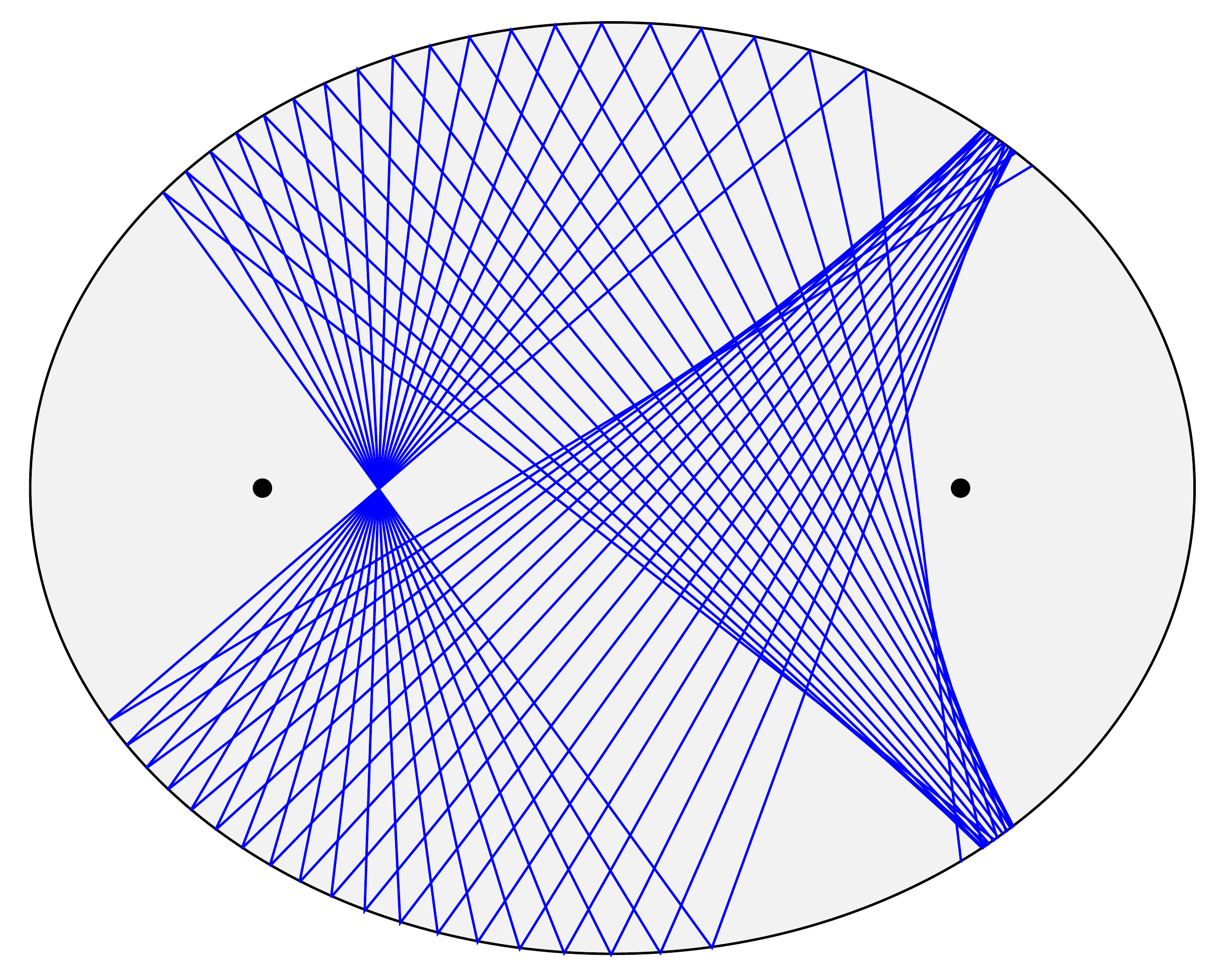}
		\end{subfigure}%				
		\begin{subfigure}{0.25\textwidth}
			\centering
			\includegraphics[width =0.9\textwidth]{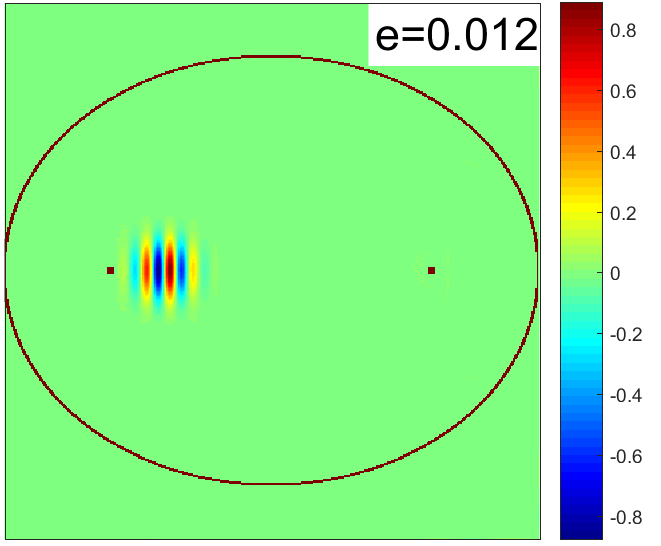}
		\end{subfigure}%
		\begin{subfigure}{0.25\textwidth}
			\centering
			\includegraphics[width =0.9\textwidth]{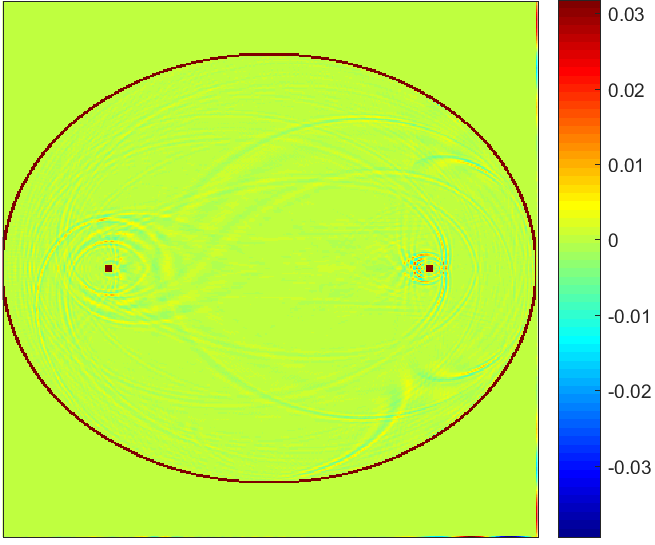}
		\end{subfigure}%
		
		\centering
		\begin{subfigure}{0.25\textwidth}
			\centering
			\includegraphics[width =0.9\textwidth]{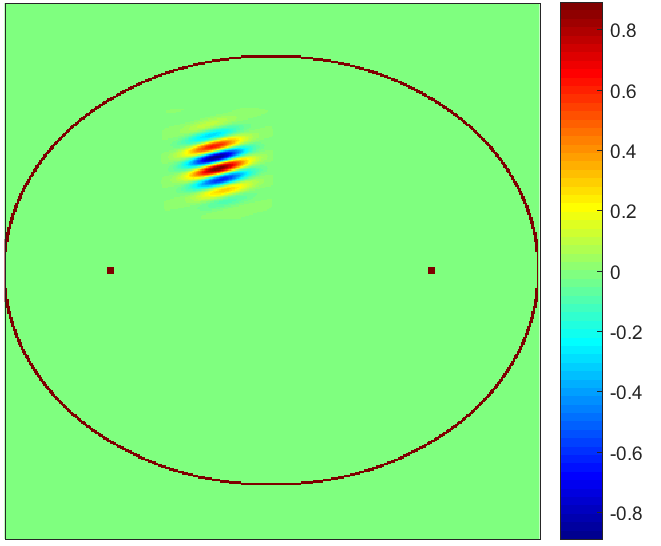}
		\end{subfigure}%				
		\begin{subfigure}{0.25\textwidth}
			\centering
			\includegraphics[width =0.765\textwidth]{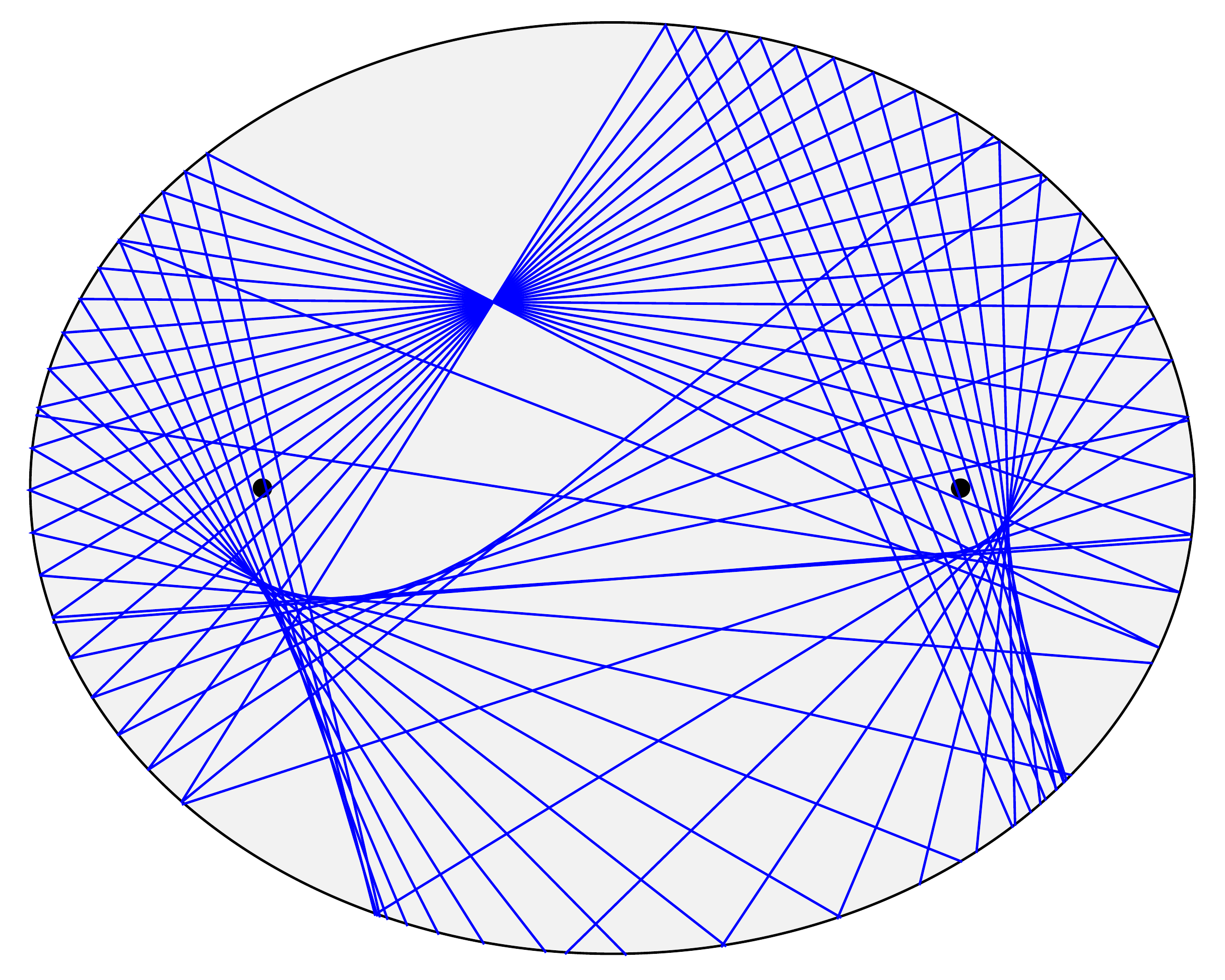}
		\end{subfigure}%				
		\begin{subfigure}{0.25\textwidth}
			\centering
			\includegraphics[width =0.9\textwidth]{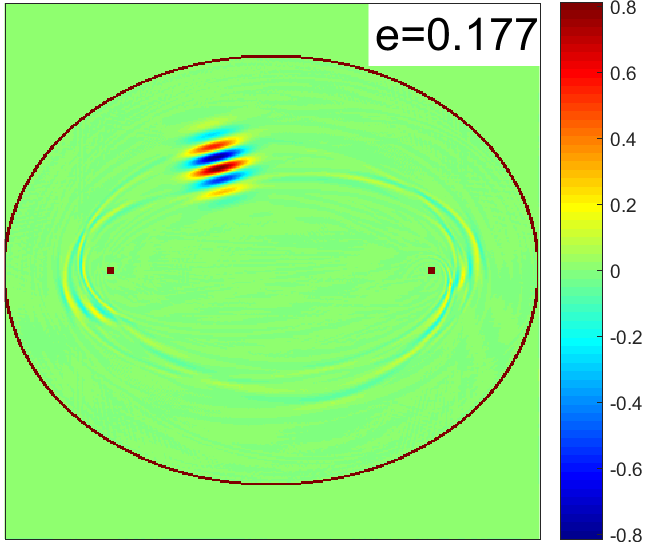}
		\end{subfigure}%
		\begin{subfigure}{0.25\textwidth}
			\centering
			\includegraphics[width =0.9\textwidth]{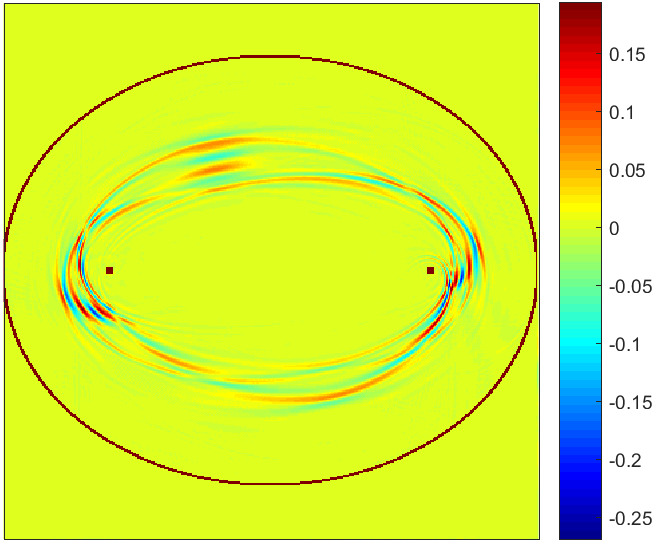}
		\end{subfigure}%
		\caption{Reconstruction of two coherent states. Left to right: true $f$, the envelopes (caused by trajectories that carry singularities and are reflected only once), $f^{(100)}$ (where $e=\frac{\|f - f^{(100)}\|_2}{\|f\|_2}$), the error.}
		\label{ellipse}
	\end{figure}
\end{eg}

\subsection{Comparison with previous results when the boundary is a circle}\label{subconnection}
This subsection is to connect our analysis to the results in \cite{Moon2017a}.
 %and \cite{Schiefeneder2016}. 
 By expanding $f$ and the data $\mathcal{B}f$ as Fourier series with respect to the angular variable, \cite{Moon2017a} gives an inversion formula (2.8) for V-line Radon transform with vertices on the circle. The denominator inside the integral has zeros for certain radius $r$ and with noises it could be very unstable. This indicates we can expect certain patterns of the artifacts in the reconstruction. We are going to show these artifacts predicted by (2.8) coincides with the conjugate covectors of radial singularities in the following.

When $(x_0,\xi^0)$ is radial, $\mathcal{M}(x_0,\xi^0)$ is complete and we have two cases. One is the case that $\mathcal{M}(x_0,\xi^0)$ is a periodic set with period $p$. That is, the broken rays that carry $(x_0,\xi^0)$ after several reflections form a regular polygon of $p$ edges, a convex or star one. The set $P$ of all possible regular polygons can be described by the Schl\"{a}fli symbol \cite{coxeter1961introduction}, 
$$P = \{(p/q),\ p,q \in \mathbb{N},\ 2 \leq 2q < p,\ \gcd(p,q)=1\}.$$ 
Here $(p/q)$ refers to a regular polygon with $p$ sides which winds $q$ times around its center. When $q = 1$, it is a convex regular one; otherwise it is a star one.
For the polygon $(p/q)$, the internal angle equals to $\frac{\pi(p-2q)}{p}$. This implies $|x| = \cos \frac{q\pi}{p}$, where $x$ is the midpoint of one edge. Suppose $\mathcal{B}f$ is smooth. We have
\begin{align*}
&R_{i-1}f_{i-1} + R_i f_i = 0 \mod C^\infty, \quad i = 1,\ldots, p-2 \\	
&R_{p-1}f_{p-1} + R_0 f_0 = 0 \mod C^\infty.
\end{align*}	
By forward substitution, we get 
$$
(1 + (-1)^{p-1})R_0 f_0 = 0 \mod C^\infty.
$$
When $p$ is odd, $f_0$ must be smooth, which implies $f$ is smooth and therefore $(x_0,\xi^0)$ is recoverable. When $p$ is even, it possibly causes artifacts. These artifacts are located at radius $|x|  =  \cos \frac{(2k+1)\pi}{2n}$, where $p = 2n$ and $q = 2k+1$ with $0 \leq 2q < p$. These radius are exactly the positive solution of $s$ such that $\cos(n(\arcsin(s) - \pi/2)) = 0$ in Formula (2.8) in \cite{Moon2017a}.

In the following example, we use the same function as in Figure \ref{globalfigure} but move them closer to the origin. The plot of error shows the artifacts are centered at the midpoint of each edges of regular stars. 

\begin{figure}[h]
	\begin{subfigure}{0.3\textwidth}
		\centering
		\includegraphics[width=0.9\linewidth]{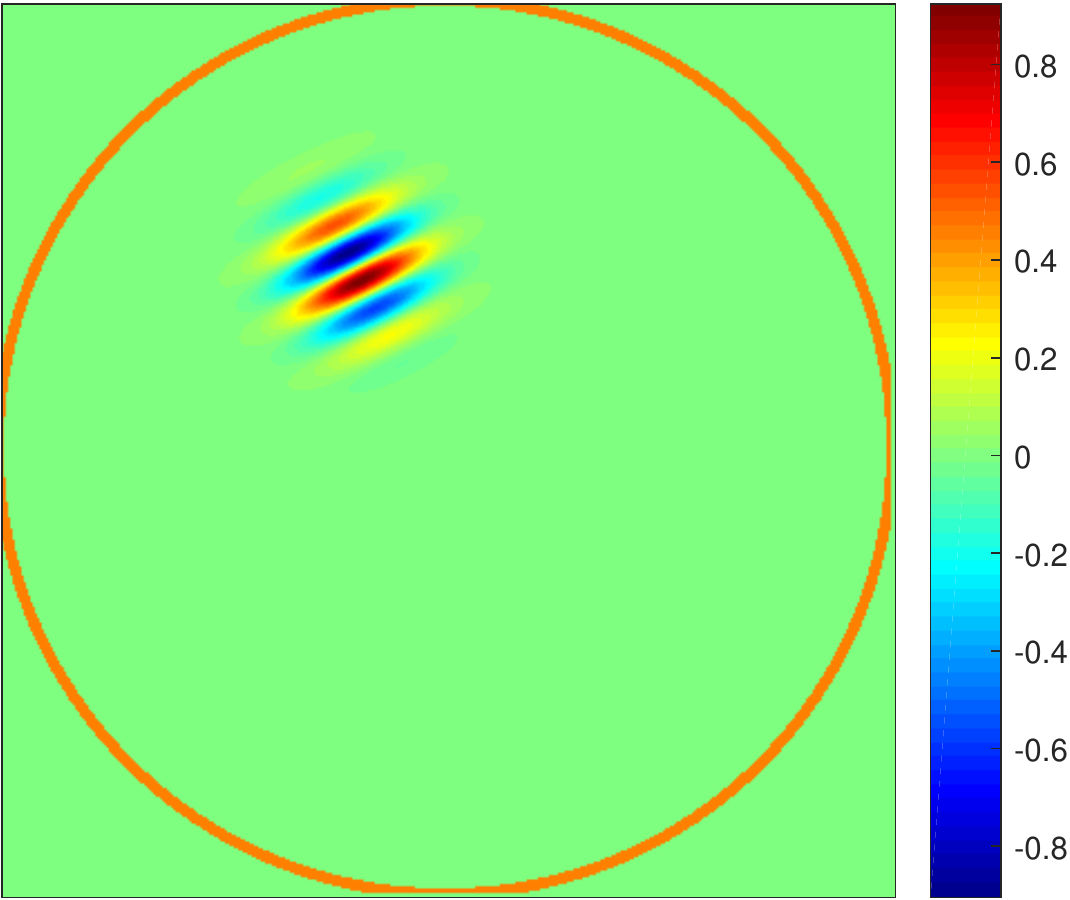}
	\end{subfigure}
	\begin{subfigure}{0.3\textwidth}
		\centering
		\includegraphics[width=0.9\linewidth]{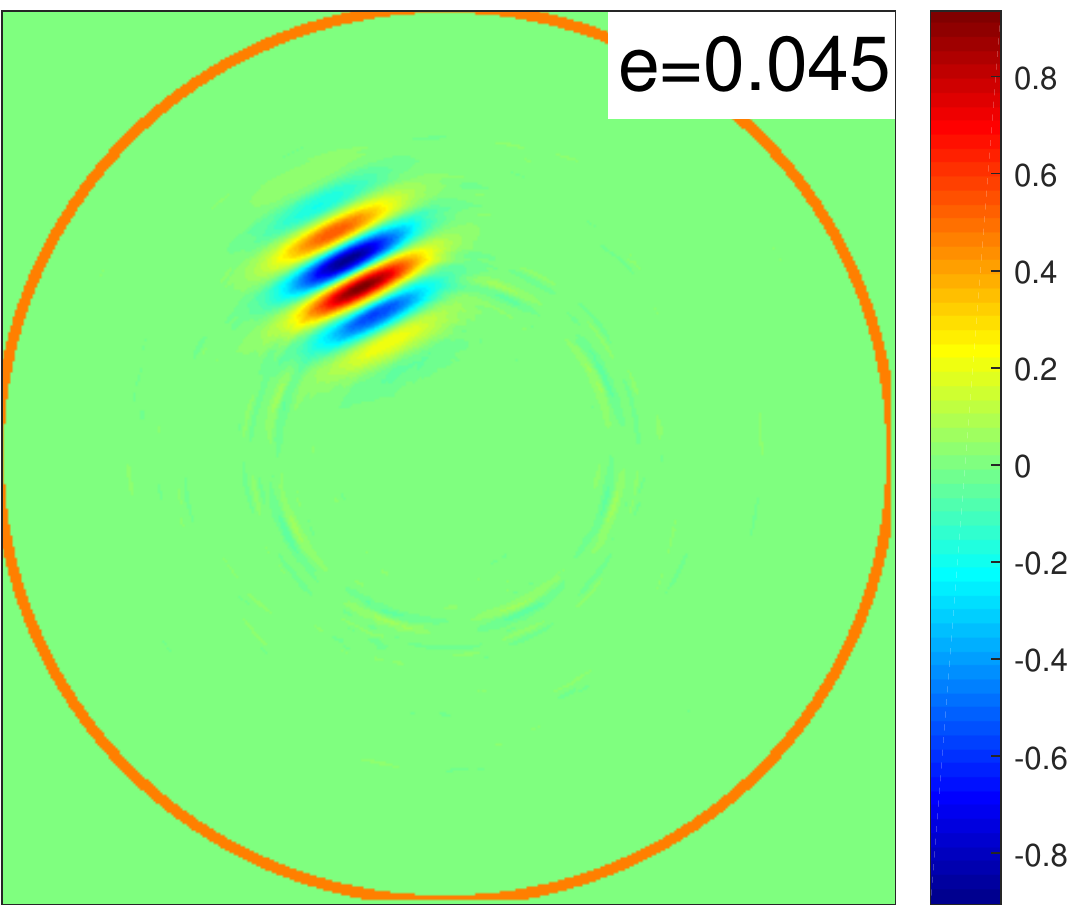}
	\end{subfigure}
	\begin{subfigure}{0.3\textwidth}
		\centering
		\includegraphics[width=0.9\linewidth]{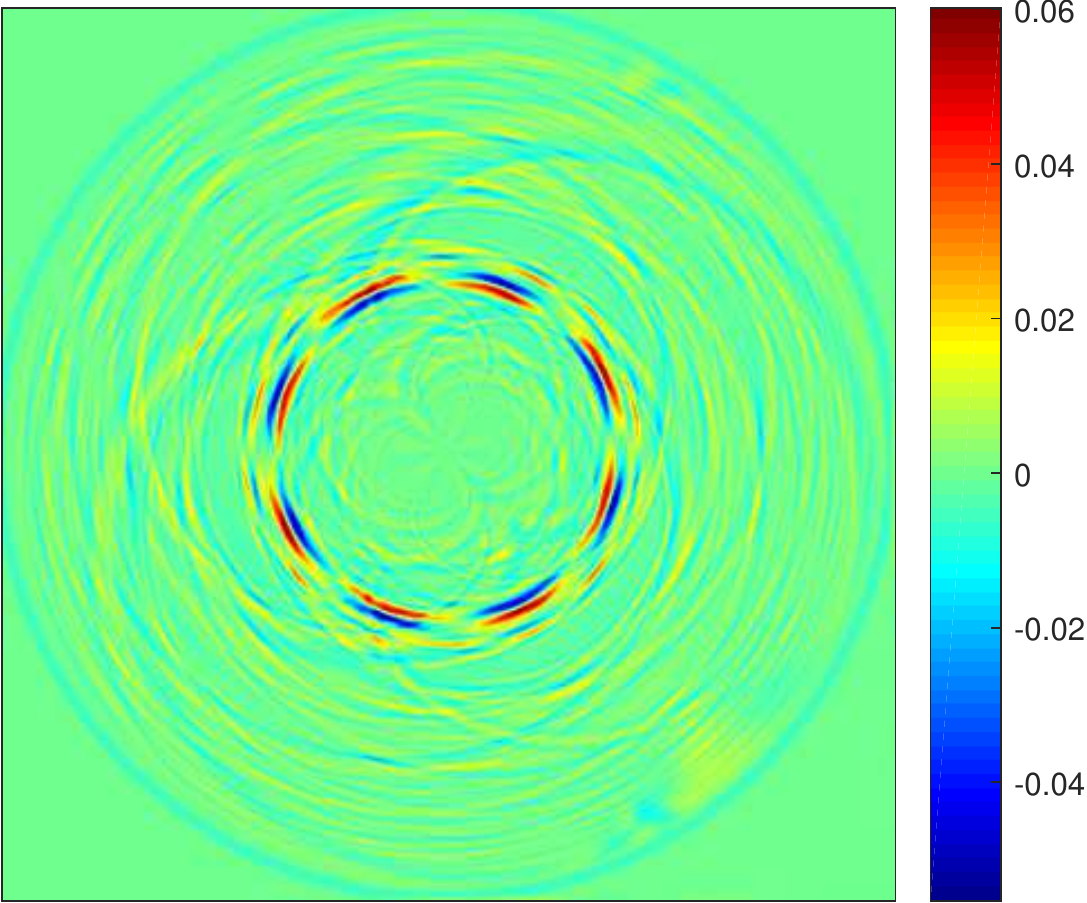}
	\end{subfigure}
	\caption{Another case of radial singularities. Left to right: true $f$, reconstruction $f^{(100)}$, error for $f$ with radial singularities after 100 iterations.  The relative error $e$ is defined as before. }
\end{figure}

We should mention that in the numerical reconstruction in \cite{Moon2017a}, the regularization (2.12) is used to remove the instabilities caused by these zeros. Therefore the artifacts are removed but on the other hand some true singularities are removed as well. In \cite{Schiefeneder2016}, the regularization is also used in the numerical reconstruction of a smiley phantom but we can still see some artifacts caused by the radial singularities (see Figure 2 in \cite{Schiefeneder2016}).

%%%%%%%%%%%%%%%%%%%%%%%% part4_Paralleltransform %%%%%%%%%%%%%%%%%%%%%%%%%%%%%%%
%%%%%%%%%%%%%%%%%%%%%%%% part4_Paralleltransform %%%%%%%%%%%%%%%%%%%%%%%%%%%%%%%
%%%%%%%%%%%%%%%%%%%%%%%% part4_Paralleltransform %%%%%%%%%%%%%%%%%%%%%%%%%%%%%%%

\section{The Parallel Rays Transform}\label{parallel_section}
We define the parallel X-ray transform as an integral transform over two or more equidistant parallel rays. The simplest case is the one over two parallel rays and is defined in the following
$$
\mathcal{P} f (s,\alpha)= \int_{x \cdot \omega(\alpha) = s} f(x)dx + \int_{x \cdot \omega(\alpha) = s+d} f(x)dx.
$$
It can be regarded as one example of the broken ray transform that we defined in Section \ref{setup_sec}, if we suppose the two rays are connected by a smooth curve outside the support of $f$ or simply at the infinity. Additionally, the diffeomorphism $\chi$ is the translation which maps $(s,\alpha)$ to $ (s+d,\alpha)$. Following the previous notations and calculation, we have
$$
\frac{d \alpha_2}{d \alpha_1} = 1, \quad \frac{d s_2}{d s_1} = -\langle p, v(\alpha_1) \rangle.
$$

By Proposition \ref{conjugatepp}, if $p$ in the incoming ray $(s_1,\alpha_1)$ has a conjugate point $q$ in the outgoing ray $(s_2,\alpha_2)$, then $q$ is determined by 
$
\langle q, v(\alpha_2) \rangle = \langle p, v(\alpha_1) \rangle,
$
which implies $ q = p + w(\alpha_1) d$.
Then Theorem \ref{cancelofC} shows a singularities $(x,\xi)$ can be canceled by $(y,\eta)$ if and only if $x$ and $y$ are conjugate points and $\xi = \eta$. It is shown in Figure \ref{pararecoverable} that the artifacts arising when we use the backprojection as the first attempt to recover $f$.

Now we consider the reconstruction by iteration process. Suppose $(x_0,\xi^0) \in \text{\(\WF\)}(f)$ belongs to the ray $(s,\alpha)$. It has two conjugate vectors, that is, $(x_0  \pm w(\alpha) d,\xi^0)$. We follow the same analysis as in the previous section to have  
$$
\mathcal{M}(x_0,\xi^0) = \{ (x_0 + i w(\alpha)d,\xi^0),\ i = 0, \pm 1, \pm 2, \ldots \}.
$$
The typology of $\mathcal{M}(x_0,\xi^0)$ is quite clear. It is a discrete set of points which has equal distance. Assume $\mathcal{P}f$ is smooth. Then $(x_0,\xi^0) \in \text{\(\WF\)}(f)$ implies $\mathcal{M}(x_0,\xi^0) \subset \text{\(\WF\)}(f)$, by the same argument as before. Thus, we have the following result, see also \cite{MR3080199}.
\begin{pp}
	Suppose $f \in \mathcal{D}'(\mathbb{R}^2)$ and assume $\mathcal{P}f$ is smooth. Then for any $(x,\xi)$, either $\mathcal{M}(x,\xi) \subset \text{\(\WF\)}(f)$ or $\mathcal{M}(x,\xi) \cap \text{\(\WF\)}(f) = \varnothing$.
\end{pp}
In particular, with the prior knowledge that $\text{\(\WF\)}(f)$ is in a compact set, the singularities are recoverable.
\begin{corollary}\label{pararecoverable}
	Suppose $f \in \mathcal{E}'(\mathbb{R}^2)$ and assume $\mathcal{P}f$ is smooth. Then $f$ is smooth.
\end{corollary}
\begin{figure}[h]
	\centering
	\begin{subfigure}{0.33\textwidth}
		\centering
		\includegraphics[width=0.9\linewidth]{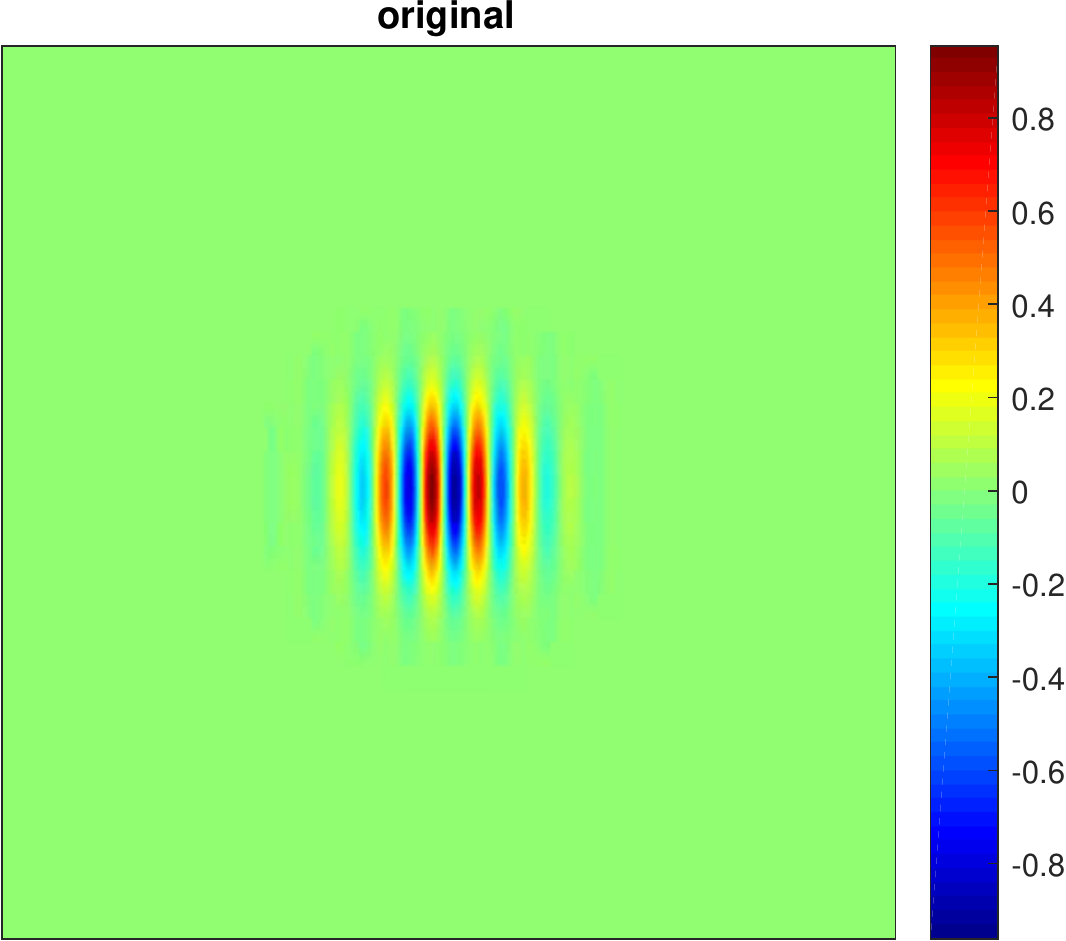}
	\end{subfigure}%
	\begin{subfigure}{0.33\textwidth}
		\centering
		\includegraphics[width=0.9\linewidth]{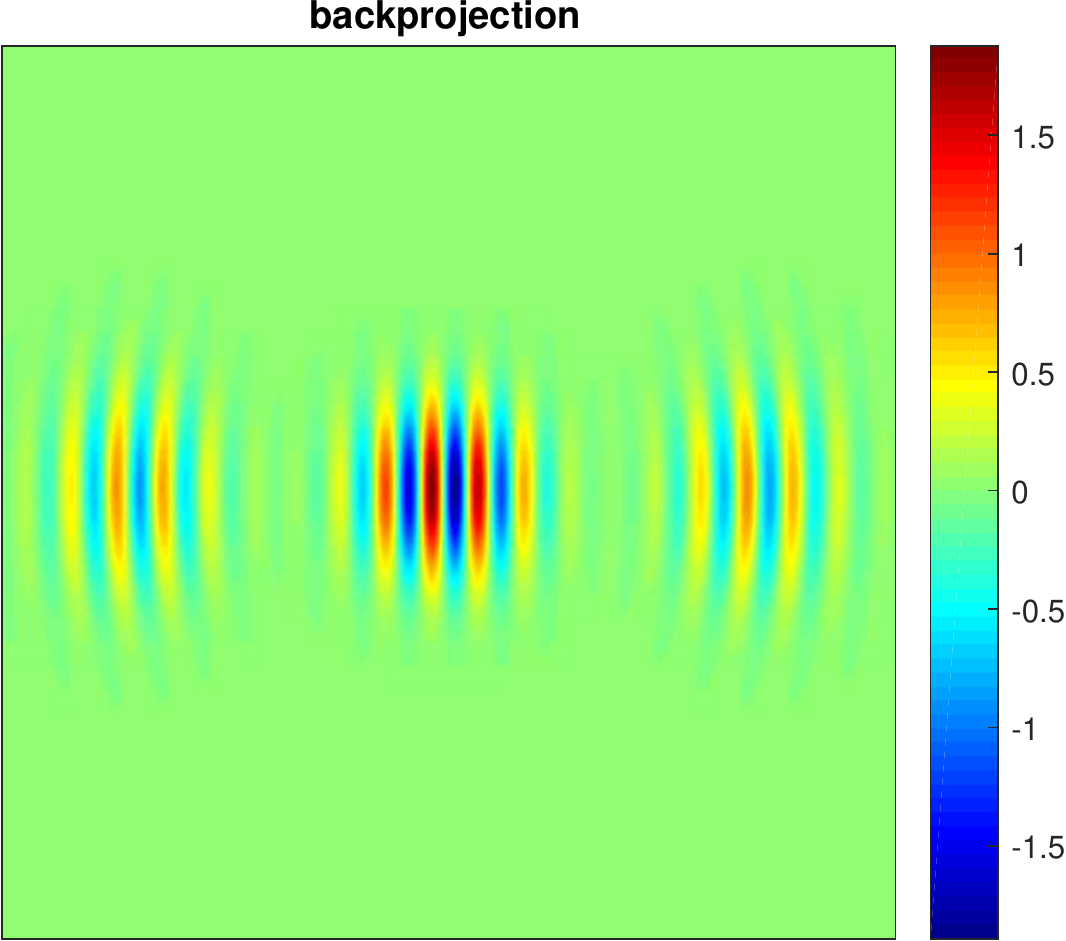}
		\label{pararecon1}
	\end{subfigure}
	\begin{subfigure}{0.33\textwidth}
		\centering
		\includegraphics[width=0.9\linewidth]{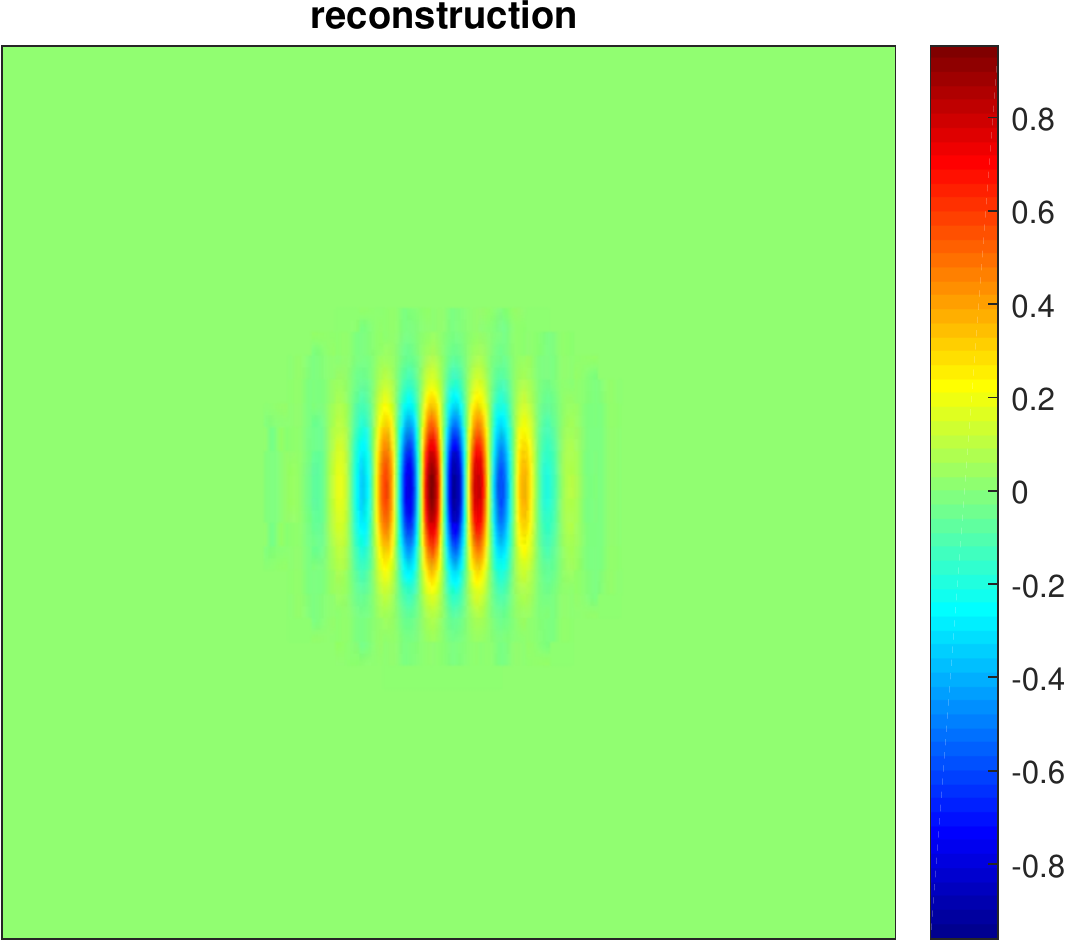}
		\label{pararecon2}
	\end{subfigure}
	\caption{Left to right: true $f$, backprojection $f^{(1)}$, $f^{(100)}$. }
	\label{test}
\end{figure}
In the numerical experiment, we use the Landweber iteration to reconstruct $f$. 
With the assumption that $f \in \mathcal{E}'(\mathbb{R}^2)$, a cutoff operator is performed at every step. 
After $100$ iterations, we get a quite good reconstruction (with $\|f^{(100)}-f\|_\infty=0.003$).

It should be mentioned that Corollary \ref{pararecoverable} shows $f$ with singularities in a compact set could be recovered from the global data. This implies when performing the transform, we move the parallel rays around until all of them leave the compact set. In fact, from our analysis above, the condition that the rays leaving at least one side the compact set is enough. On the other hand, the local problem (illumination of a region of interest only) could create artifacts.

%%%%%%%%%%%%%%%%%%%%%%%% bibliography %%%%%%%%%%%%%%%%%%%%%%%%%%%%%%%
\begin{footnotesize}
	\bibliographystyle{plain}
	\bibliography{microlocal_analysis}
\end{footnotesize}

\end{document}